\documentclass[10pt]{amsart}
\usepackage{amssymb}
\usepackage{amscd,setspace}

\numberwithin{equation}{section}

\newcounter{TmpEnumi}

\def\today{\number\day\space\ifcase\month\or   January\or February\or
   March\or April\or May\or June\or   July\or August\or September\or
   October\or November\or December\fi\   \number\year}

\theoremstyle{definition}
\newtheorem{thm}{Theorem}[section]
\newtheorem{lem}[thm]{Lemma}
\newtheorem{prp}[thm]{Proposition}
\newtheorem{dfn}[thm]{Definition}
\newtheorem{cor}[thm]{Corollary}

\newtheorem{rmk}[thm]{Remark}
\newtheorem{ntn}[thm]{Notation}
\newtheorem{exa}[thm]{Example}

\newcommand{\beq}{\begin{equation}}
\newcommand{\eeq}{\end{equation}}
\newcommand{\beqr}{\begin{eqnarray*}}
\newcommand{\eeqr}{\end{eqnarray*}}
\newcommand{\bal}{\begin{align*}}
\newcommand{\eal}{\end{align*}}
\newcommand{\bei}{\begin{itemize}}
\newcommand{\eei}{\end{itemize}}
\newcommand{\limi}[1]{\lim_{{#1} \to \infty}}

\newcommand{\af}{\alpha}
\newcommand{\bt}{\beta}
\newcommand{\gm}{\gamma}
\newcommand{\dt}{\delta}
\newcommand{\ep}{\varepsilon}

\newcommand{\et}{\eta}

\newcommand{\io}{\iota}
\newcommand{\te}{\theta}
\newcommand{\ld}{\lambda}

\newcommand{\kp}{\kappa}
\newcommand{\ph}{\varphi}
\newcommand{\ps}{\psi}
\newcommand{\rh}{\rho}
\newcommand{\om}{\omega}
\newcommand{\ta}{\tau}

\newcommand{\Z}{{\mathbb{Z}}}
\newcommand{\R}{{\mathbb{R}}}
\newcommand{\C}{{\mathbb{C}}}
\newcommand{\N}{{\mathbb{Z}}_{> 0}}
\newcommand{\Nz}{{\mathbb{Z}}_{\geq 0}}

\newcommand{\id}{{\mathrm{id}}}

\newcommand{\sint}{{\mathrm{int}}}

\newcommand{\dist}{{\mathrm{dist}}}

\newcommand{\spec}{{\mathrm{sp}}}

\newcommand{\supp}{{\mathrm{supp}}}

\newcommand{\sgn}{{\mathrm{sgn}}}

\newcommand{\Aut}{{\mathrm{Aut}}}

\newcommand{\QT}{{\mathrm{QT}}}

\newcommand{\Mi}{M_{\infty}}

\newcommand{\andeqn}{\,\,\,\,\,\, {\mbox{and}} \,\,\,\,\,\,}

\newcommand{\wolog}{without loss of generality}
\newcommand{\Wolog}{Without loss of generality}

\newcommand{\tfae}{the following are equivalent}
\newcommand{\ifo}{if and only if}

\newcommand{\ca}{C*-algebra}
\newcommand{\uca}{unital C*-algebra}

\newcommand{\hm}{homomorphism}

\newcommand{\fd}{finite dimensional}
\newcommand{\tst}{tracial state}
\newcommand{\hsa}{hereditary subalgebra}

\newcommand{\pj}{projection}

\newcommand{\nzp}{nonzero projection}
\newcommand{\mvnt}{Murray-von Neumann equivalent}
\newcommand{\mvnc}{Murray-von Neumann equivalence}

\newcommand{\ct}{continuous}
\newcommand{\cfn}{continuous function}

\newcommand{\chs}{compact Hausdorff space}
\newcommand{\hme}{homeomorphism}
\newcommand{\mh}{minimal homeomorphism}

\newcommand{\cp}{crossed product}

\title[Large subalgebras]{Large subalgebras}

\author{N.~Christopher Phillips}

\date{23~August 2014}

\address{Department of Mathematics, University  of Oregon,
       Eugene OR 97403-1222, USA.}

\email[]{ncp@darkwing.uoregon.edu}

\thanks{This material is based upon work supported by the
  US National Science Foundation under
  Grants DMS-0701076 and DMS-1101742.}

\begin{document}

\begin{abstract}
We define and study large and stably large subalgebras
of simple unital C*-algebras.
The basic example is the orbit breaking subalgebra
of a crossed product by~${\mathbb{Z}}$,
as follows.
Let $X$ be an infinite compact metric space,
let $h \colon X \to X$ be a minimal homeomorphism,
and let $Y \subset X$ be closed.
Let $u \in C^* ({\mathbb{Z}}, X, h)$ be the standard unitary.
The $Y$-orbit breaking subalgebra
is the subalgebra of $C^* ({\mathbb{Z}}, X, h)$
generated by $C (X)$
and all elements $f u$
for $f \in C (X)$ such that $f |_Y = 0$.
If $h^n (Y) \cap Y = \varnothing$
for all $n \in {\mathbb{Z}} \setminus \{ 0 \}$,
then the $Y$-orbit breaking subalgebra
is large in $C^* ({\mathbb{Z}}, X, h)$.
Large subalgebras obtained via generalizations of
this construction have appeared in a number of places,
and we unify their theory in this paper.

We prove the following results
for an infinite dimensional simple unital \ca~$A$
and a stably large subalgebra $B \subset A$:
\begin{itemize}
\item
$B$ is simple and infinite dimensional.
\item
If $B$ is stably finite then so is~$A$,
and if $B$ is purely infinite then so is~$A$.
\item
The restriction maps ${\operatorname{T}} (A) \to {\operatorname{T}} (B)$
and ${\operatorname{QT}} (A) \to {\operatorname{QT}} (B)$
(on tracial states and quasitraces) are bijective.
\item
When $A$ is stably finite,
the inclusion of $B$ in $A$ induces an isomorphism
on the semigroups that remain after deleting
from ${\operatorname{Cu}} (B)$ and ${\operatorname{Cu}} (A)$
all the classes of nonzero projections.
\item
$B$ and $A$ have the same radius of comparison.
\end{itemize}
\end{abstract}

\maketitle

\indent
The purpose of this paper is to define what we call a
large subalgebra~$B$ in a simple \uca~$A$,
and to show how properties of~$B$
can be used to deduce properties of~$A$.
The main applications so far are to the structure of crossed
product \ca{s},
and are treated elsewhere;
see the discussion below.
They work because it is possible to choose large
subalgebras of these crossed products
which are of an accessible form,
such as a direct limit of recursive subhomogeneous algebras.
A strengthening of the condition
(centrally large subalgebras) permits further results
about the containing algebra;
this will also be treated elsewhere \cite{ArPh},~\cite{ArPh2}.

Large subalgebras (and centrally large subalgebras)
are an abstraction of the Putnam subalgebra
of the crossed product by a minimal homeomorphism.
Let $X$ be an infinite compact metric space,
and let $h \colon X \to X$ be a minimal homeomorphism.
Let $u$ be the standard unitary in the
crossed product $C^* (\Z, X, h)$.
Fix $y \in X$.
Then the Putnam subalgebra
of $C^* (\Z, X, h)$
is generated by $C (X)$ and all elements $f u$ with $f \in C (X)$
satisfying $f (y) = 0$.
This algebra was introduced by Putnam in~\cite{Pt1}
when $X$ is the Cantor set.
(Putnam actually used $u f$ rather than $f u$,
but this choice makes the
relationship with Rokhlin towers more awkward.)
In this case,
on the one hand,
the subalgebra is an AF~algebra,
while,
on the other hand,
it is closely enough related to $C^* (\Z, X, h)$
to use information about it to obtain information
about $C^* (\Z, X, h)$.

This method was used
in~\cite{LqP} and Section~4 of~\cite{Ph7}
to obtain information on the order on
$K_0 ( C^* (\Z, X, h) )$
for general \fd~$X$.
The Putnam subalgebra played a key role in~\cite{LhP},
in which it is proved that $C^* (\Z, X, h)$ has tracial rank zero
whenever this property is consistent
with its K-theory and $\dim (X) < \infty$,
and in~\cite{TW},
which gives classifiability of such crossed products
in some cases in which they don't have tracial rank zero.
The paper~\cite{TW} also required a generalization
in which one used two points $y_1$ and $y_2$
on distinct orbits of~$h$,
and in the definition used $f u$
for $f \in C (X)$ such that $f (y_1) = f (y_2) = 0$.
A more recent application appears in~\cite{EN2}.
Versions in which $f$ is required to vanish on
a larger subset are important in \cite{LM1} and~\cite{Sn}.
Further applications of such generalized Putnam algebras
will appear in~\cite{HPT}.
Particular examples of these subalgebras
have been studied in their own right in~\cite{FJLX}.
The subalgebra $A_{\te, \gm}$ considered there
(see the introduction) is large whenever
the zero set of the function $\gm$
intersects each orbit at most once.
Under similar conditions,
the algebras studied in~\cite{Sn2}
are large in the corresponding three dimensional noncommutative tori.

The abstraction to large subalgebras has four motivations.
The first is the use,
as described above, of subalgebras of $C^* (\Z, X, h)$
generated by $C (X)$ and the elements $f u$
with $f$ required to vanish on a subset with more
than one point,
but which meets each orbit of $h$ at most once.
The second is the generalization to
crossed products by automorphisms of $C (X, D)$
in~\cite{Bc}.
Let $X$ be an infinite compact metric space,
let $h \colon X \to X$ be a minimal homeomorphism,
let $D$ be a simple \uca{}
satisfying suitable additional conditions,
and let $\af \in \Aut (C (X, D))$
be an automorphism such that,
in terms of $C (X) \otimes D$,
we have $\af (f \otimes 1) = (f \circ h^{-1}) \otimes 1$
for all $f \in C (X)$.
Let $u \in C^* \big( \Z, \, C (X, D), \, \af)$
be the standard unitary in the \cp,
and fix $y \in Y$.
Then the subalgebra used is the one generated by
$C (X, D)$ and all $f u$ with $f \in C (X, D)$
satisfying $f (y) = 0$.

A third, stronger, motivation for the abstraction
is the construction of large subalgebras
in more general crossed products,
where the subalgebras don't have convenient descriptions.
Large subalgebras (without the name)
play a key role in~\cite{Ph10},
where they are used to prove that
if $\Z^d$ acts freely and minimally on the Cantor set~$X$,
then $C^* (\Z^d, X)$
has real rank zero,
stable rank one,
and order on projections determined by traces.
It is shown in~\cite{Ph11}
that if $X$ above is a finite dimensional compact metric space,
then $C^* (\Z^d, X)$ contains a large subalgebra
which is a simple direct limit,
with no dimension growth,
of recursive subhomogeneous \ca{s}.
Although this paper is still unpublished,
this was the first proof that, for such~$X$,
the crossed product has strict comparison of positive elements.
A more abstract version is needed because
there is no known easy description of the subalgebra;
rather, there is just an existence proof.

A fourth reason for the abstraction is the role played by
large subalgebras in~\cite{EN1}.
This paper considers \ca{s} obtained from irrational
rotation algebras by ``cutting'' each of the standard unitary generators
at one or more points in its spectrum,
say by adding logarithms of them or adding some
spectral projections.
The new algebras are shown to be~AF.
One of the technical tools is that
the original irrational rotation algebra
is a large subalgebra the new algebra.
In this case,
the containing algebra is not even given as a crossed product.

In this paper,
we prove the following results,
for an infinite dimensional simple unital \ca~$A$
and a stably large subalgebra $B \subset A$.
(All the large subalgebras discussed above are in fact stably large.)
\begin{enumerate}
\item\label{4814_Large_Simp}
$B$ is simple and infinite dimensional.
\item\label{4814_Large_Fin}
If $B$ is stably finite then so is~$A$,
and if $B$ is purely infinite then so is~$A$.
\item\label{4814_Large_TQT}
The restriction maps ${\operatorname{T}} (A) \to {\operatorname{T}} (B)$
and ${\operatorname{QT}} (A) \to {\operatorname{QT}} (B)$
(on tracial states and quasitraces) are bijective.
\item\label{4814_Large_Cu}
When $A$ is stably finite,
the inclusion of $B$ in $A$ induces an isomorphism
on the semigroups that remain after deleting
from ${\operatorname{Cu}} (B)$ and ${\operatorname{Cu}} (A)$
all the classes of nonzero projections.
\item\label{4814_Large_RC}
When $A$ is stably finite,
$B$ and $A$ have the same radius of comparison.
\end{enumerate}
At least heuristically,
the basic result is~(\ref{4814_Large_Cu}),
and the others follow from it.
We also show that the following basic example
is a large subalgebra.
Let $X$ be an infinite compact metric space,
let $h \colon X \to X$ be a minimal homeomorphism,
and let $Y \subset X$ be closed.
The {\emph{$Y$-orbit breaking subalgebra}} of $C^* (\Z, X, h)$
associated to $Y$
is the subalgebra generated by $C (X)$ and all $f u$
with $f \in C (X)$ and $f |_Y = 0$.
If $Y$ meets each orbit at most once,
we prove that this subalgebra is large in $C^* (\Z, X, h)$.

Stable rank one and $Z$-stability seem to require
the stronger condition of central largeness,
and will be treated in \cite{ArPh} and~\cite{ArPh2}.

We only define a large subalgebra $B \subset A$
when $A$ is simple.
If $A$ is not simple,
then also $B$ will not be simple,
and
one must be much more careful with what is means
for a positive element $g \in B$
(or a hereditary subalgebra of~$B$)
to be ``small''.
See the discussion after Definition~\ref{D_Large}.

This paper is organized as follows.
The first three sections are mainly about the Cuntz semigroup.
Section~\ref{Sec_CS}
gives some standard results on Cuntz comparison
and the Cuntz semigroup.
We have listed the results, but don't give proofs.
This section also contains some new lemmas on Cuntz comparison.
Among other things,
we need a relation between
$\langle a \rangle$,
$\langle g \rangle$,
and $\langle (1 - g) a (1 - g) \rangle$
for $a \geq 0$ and $0 \leq g \leq 1$,
as well as a version using $(a - \ep)_{+}$ etc.
In Section~\ref{Sec_CSimp},
we give some more specialized results,
related to Cuntz comparison in simple \ca{s}.
Section~\ref{Sec_PPE} is devoted to the
subsemigroup of purely positive elements
in the Cuntz semigroup of a stably finite simple \ca.
In particular,
in some ways this subsemigroup controls
the behavior of the entire Cuntz semigroup.

In Section~\ref{Sec_Dfn} we define large subalgebras,
stably large subalgebras,
and large subalgebras of crossed product type.
The examples used in applications
are mostly of crossed product type.
We will show in~\cite{ArPh}
that large subalgebras of crossed product type
are in fact centrally large.
We then give several convenient variants of the definition.
Section~\ref{Sec_FP} contains some basic properties
of large subalgebras.
They are simple and infinite dimensional.
If the containing algebras are stably finite,
then the minimal tensor product of large subalgebras is large.
In particular, if $B \subset A$ is large
and $A$ is stably finite,
then $M_n (B)$ is large in $M_n (A)$
for all~$n$ (that is, $B$ is stably large).
In Section~\ref{Sec_CuLg},
we prove out main results on stably large subalgebra,
as described above.
Section~\ref{Sec_VK} proves that the $Y$-orbit breaking subalgebra
of a minimal homeomorphism is large
when $Y$ meets each orbit at most once.

We thank George Elliott for questions which led to the realization
that our methods imply Theorem~\ref{T-2720CuSurj}
and Theorem~\ref{T-2725Inj}.
(See (\ref{4814_Large_Cu}) above.)
These statements are much more general and informative
than the original results.

We also thank Julian Buck, Mikael R{\o}rdam,
Andrew Toms,
and particularly Dawn Archey for useful comments,
and Leonel Robert for a number of references and suggestions.

\section{The Cuntz semigroup}\label{Sec_CS}

\indent
In this section, we give a brief summary of the Cuntz semigroup
and some known facts about Cuntz comparison
and the Cuntz semigroup.
We then give some apparently new results,
for example relating
\[
\langle (a - \ep)_{+} \rangle,
\,\,\,\,\,\,
\langle g \rangle,
\andeqn
\big\langle \big[ (1 - g) a (1 - g) - \ep \big]_{+} \big\rangle
\]
for $a \geq 0$ and $0 \leq g \leq 1$.
We further give proofs of results relating
Cuntz comparison to ideals and tensor products.
Finally,
we summarize known results about supremums in the Cuntz semigroup,
functionals, and quasitraces.

Let $M_{\infty} (A)$ denote the algebraic direct limit of the
system $(M_n (A))_{n = 1}^{\infty}$
using the usual embeddings $M_n (A) \to M_{n+1} (A)$,
given by
\[
a \mapsto \left( \begin{array}{cc} a & 0 \\ 0 & 0 \end{array} \right).
\]
If $a \in M_m (A)$ and $b \in M_n (A)$,
we write $a \oplus b$ for the diagonal direct sum
\[
a \oplus b
 = \left( \begin{array}{cc} a & 0 \\ 0 & b \end{array} \right).
\]
By abuse of notation,
we will also write $a \oplus b$ when $a, b \in M_{\infty} (A)$
and we do not care about the precise choice of $m$ and $n$
with $a \in M_m (A)$ and $b \in M_n (A)$.
We further choose some isomorphism $M_2 (K) \to K$,
and for $a, b \in K \otimes A$ we use the resulting isomorphim
$M_2 (K \otimes A) \to K \otimes A$ to interpret
$a \oplus b$ as an element of $K \otimes A$.
Up to unitary equivalence which is trivial on~$A$,
the result does not depend on the choice of
the isomorphism $M_2 (K) \to K$.

The main object of study in this paper is how
comparison in the Cuntz semigroup of a \ca~$A$
relates to comparison in the Cuntz semigroup of a
subalgebra~$B$ satisfying certain conditions.
We therefore include the algebra in the notation for
Cuntz comparison.

If $B$ is a \ca,
or if $B = \Mi (A)$ for a \ca~$A$,
we write $B_{+}$ for the set of positive elemnts of~$B$.

Parts (\ref{D:CzSGp:1}) and~(\ref{D:CzSGp:2})
of the following definition are originally from~\cite{Cz1}.

\begin{dfn}\label{D:CzSGp}
Let $A$ be a \ca.
\begin{enumerate}
\item\label{D:CzSGp:1}
For $a, b \in (K \otimes A)_{+}$,
we say that $a$ is {\emph{Cuntz subequivalent to~$b$
over~$A$}},
written $a \precsim_A b$,
if there is a sequence $(v_n)_{n = 1}^{\infty}$ in $K \otimes A$
such that
$\limi{n} v_n b v_n^* = a$.
\item\label{D:CzSGp:2}
We say that $a$ and $b$ are {\emph{Cuntz equivalent in~$A$}},
written $a \sim_A b$,
if $a \precsim_A b$ and $b \precsim_A a$.
This relation is an equivalence relation,
and we write $\langle a \rangle$ for the equivalence class of~$a$.
\item\label{D:CzSGp:3}
The {\emph{Cuntz semigroup}} of~$A$ is
\[
{\operatorname{Cu}} (A) = (K \otimes A)_{+} / \sim_A,
\]
together with the commutative semigroup operation
\[
\langle a \rangle + \langle b \rangle = \langle a \oplus b \rangle
\]
(the class does not depend on the choice of the isomorphism
$M_2 (K) \to K$)
and the partial order
\[
\langle a \rangle \leq \langle b \rangle \Leftrightarrow a \precsim_A b.
\]
It is taken to be an object of the category~${\mathbf{Cu}}$
given in Definition~4.1 of~\cite{APT}.
\item\label{D:CzSGp:4}
We also define the subsemigroup
\[
W (A) = M_{\infty} (A)_{+} / \sim_A,
\]
with the same operations and order.
(We will see in Remark~\ref{R-2727MnI} that the obvious
map $W (A) \to {\operatorname{Cu}} (A)$ is injective.)
We write $0$ for~$\langle 0 \rangle$.
\item\label{D:CzSGp:5}
Let $A$ and $B$ be C*-algebras,
and let $\ph \colon A \to B$ be a \hm.
We use the same letter for the induced maps
$M_n (A) \to M_n (B)$
for $n \in \N$,
$\Mi (A) \to \Mi (B)$,
and $K \otimes A \to K \otimes B$.
We define
${\operatorname{Cu}} (\ph)
 \colon {\operatorname{Cu}} (A) \to {\operatorname{Cu}} (B)$
and $W (\ph) \colon W (A) \to W (B)$
by $\langle a \rangle \mapsto \langle \ph (a) \rangle$
for $a \in (K \otimes A)_{+}$
or $M_{\infty} (A)_{+}$ as appropriate.
\end{enumerate}
\end{dfn}

It is easy to verify that,
in Definition~\ref{D:CzSGp},
the maps ${\operatorname{Cu}} (\ph)$ and $W (\ph)$ are well defined
homomorphisms of ordered semigroups which send
$0$ to~$0$.

The semigroup ${\operatorname{Cu}} (A)$ generally has better properties.
For example, certain supremums exist
(Theorem~4.19 of~\cite{APT};
see Theorem~\ref{T_2Y25_CxSup}(\ref{T_2Y25_CxSup_Ex}) below),
and, when understood as an object of the category~${\mathbf{Cu}}$,
it behaves properly
with respect to direct limits
(Theorem~4.35 of~\cite{APT}).
We will use $W (A)$ as well because,
when $A$ is unital,
the dimension function $d_{\ta}$
associated to a normalized quasitrace~$\ta$,
of Definition~\ref{D:dtau} below, is finite on $W (A)$,
but usually not on ${\operatorname{Cu}} (A)$.

We will not need the details
of the definition of the category~${\mathbf{Cu}}$.

\begin{rmk}\label{R-2727MnI}
We make the usual identifications
$A \subset M_n (A) \subset M_{\infty} (A) \subset K \otimes A$.
If $a, b \in A_{+}$ and $a \precsim_A b$,
then we claim that there
is a sequence $(v_n)_{n = 1}^{\infty}$ in~$A$
such that
$\limi{n} v_n b v_n^* = a$.
To see this,
choose a sequence $(w_n)_{n = 1}^{\infty}$ in $K \otimes A$
such that
$\limi{n} w_n b w_n^* = a$,
let
$(e_{j, k})_{j, k \in \N}$ be the standard system
of matrix units for~$K$,
and set $v_n = (e_{1, 1} \otimes 1) w_n (e_{1, 1} \otimes 1)$.

Similar reasoning
shows that if $a, b \in M_n (A)_{+}$ for some $n \in \N$,
then $(v_n)_{n = 1}^{\infty}$ can be taken to be in $M_n (A)$,
and similarly with $M_{\infty} (A)$ in place of $M_n (A)$.
(This also follows from Lemma 2.2(iii) of~\cite{KR}.)

If $a$ and $b$ are in any of
$A_{+}$, $M_n (A)_{+}$, $M_{\infty} (A)_{+}$, or $(K \otimes A)_{+}$
(not necessarily the same one for both),
we can thus write $a \precsim_A b$
(or $a \sim_A b$)
to mean that this relation holds in $K \otimes A$,
equivalently,
that this relation holds in the smallest of
$A$, $M_n (A)$, $M_{\infty} (A)$, or $K \otimes A$
which contains both $a$ and~$b$.
(This is the same convention as in Definition~2.1 of~\cite{KR}.)
\end{rmk}

\begin{dfn}\label{D:MinusEp}
Let $A$ be a \ca, let $a \in A_{+}$,
and let $\ep > 0$.
Let $f \colon [0, \infty) \to [0, \infty)$ be the function
\[
f (\lambda)
 = (\lambda - \ep)_{+}
 = \begin{cases}
     0               & \hspace{3em}  0 \leq \lambda \leq \ep  \\
     \lambda - \ep   & \hspace{3em} \ep < \lambda.
    \end{cases}
\]
Then define $(a - \ep)_{+} = f (a)$
(using \ct{} functional calculus).
\end{dfn}

The following lemma summarizes some of the known results
about Cuntz subequivalence that we need.
Most of it
is in Section~2 of~\cite{KR},
although not all of it is original there.
A warning on notation:
In~\cite{KR},
the notation $a \sim b$ means that there exists $c$ such that
$c^* c = a$ and $c c^* = b$,
while our $a \sim_A b$ is written $a \approx b$ in~\cite{KR}.
We denote by $A^{+}$ the unitization of a \ca~$A$.
(We add a new unit even if $A$ is already unital.)

\begin{lem}\label{L:CzBasic}
Let $A$ be a \ca.
\begin{enumerate}
\item\label{L:CzBasic:Her}
Let $a, b \in A_{+}$.
Suppose $a \in {\overline{b A b}}$.
Then $a \precsim_A b$.
\item\label{L:CzBasic:LCzOneWay}
Let $a \in A_{+}$ and let $f \colon [0, \infty) \to [0, \infty)$
be a \cfn{} such that $f (0) = 0$.
Then $f (a) \precsim_A a$.
\item\label{L:CzBasic:LCzFCalc}
Let $a \in A_{+}$ and let $f \colon [0, \, \| a \| ] \to [0, \infty)$
be a \cfn{} such that $f (0) = 0$ and $f (\ld) > 0$ for $\ld > 0$.
Then $f (a) \sim_A a$.
\item\label{L:CzBasic:LCzComm} 
Let $c \in A$.
Then $c^* c \sim_A c c^*$.
\item\label{L:CzBasic:3904_UE}
Let $a \in A_{+}$,
and let $u \in A^{+}$ be unitary.
Then $u a u^* \sim_A a$.
\item\label{L:CzBasic:LCzCommEp}
Let $c \in A$ and let $\af > 0$.
Then $(c^* c - \af)_{+} \sim_A (c c^* - \af)_{+}$.
\item\label{L:CzBasic:N5}
Let $v \in A$.
Then there is an isomorphism
$\ph \colon {\overline{v^* v A v^* v}}
   \to {\overline{v v^* A v v^*}}$
such that, for every positive element
$z \in {\overline{v^* v A v^* v}}$,
we have
$z \sim_A \ph (z)$.
\item\label{L:CzBasic:MinIter}
Let $a \in A_{+}$ and let $\ep_1, \ep_2 > 0$.
Then
\[
\big( ( a - \ep_1)_{+} - \ep_2 \big)_{+}
 = \big( a - ( \ep_1 + \ep_2 ) \big)_{+}.
\]
\item\label{L:CzBasic:N6}
Let $a, b \in A_{+}$ satisfy $a \precsim_A b$
and let $\dt > 0$.
Then there is $v \in A$
such that $v^* v = (a - \dt)_{+}$
and $v v^* \in {\overline{b A b}}$.
\item\label{L:CzBasic:LCzWithinEp} 
Let $a, b \in A_{+}$.
Then $\| a - b \| < \ep$ implies $(a - \ep)_{+} \precsim_A b$.
\item\label{L:CzBasic:LMinusEp} 
Let $a, b \in A_{+}$.
Then \tfae:
\begin{enumerate}
\item\label{L:CzBasic:LMinusEp:1}
$a \precsim_A b$.
\item\label{L:CzBasic:LMinusEp:2}
$(a - \ep)_{+} \precsim_A b$ for all $\ep > 0$.
\item\label{L:CzBasic:LMinusEp:3}
For every $\ep > 0$ there is $\dt > 0$ such that
$(a - \ep)_{+} \precsim_A (b - \dt)_{+}$.
\end{enumerate}
\item\label{L:CzBasic:LCzCmpSum} 
Let $a, b \in A_{+}$.
Then $a + b \precsim_A a \oplus b$.
\item\label{L:CzBasic:Orth}
Let $a, b \in A_{+}$ be orthogonal (that is, $a b = 0$).
Then $a + b \sim_A a \oplus b$.
\item\label{L:CzBasic:CmpDSum}
Let $a_1, a_2, b_1, b_2 \in A_{+}$,
and suppose that $a_1 \precsim_A a_2$ and $b_1 \precsim_A b_2$.
Then $a_1 \oplus b_1 \precsim_A a_2 \oplus b_2$.
\end{enumerate}
\end{lem}

\begin{proof}
Part~(\ref{L:CzBasic:Her}) is Proposition 2.7(i) of~\cite{KR}.
Part~(\ref{L:CzBasic:LCzOneWay}) is Lemma~2.2(i) of~\cite{KR}.
For part~(\ref{L:CzBasic:LCzFCalc}),
one sees easily that $a$ and $f (a)$ generate the same
hereditary subalgebra of~$A$.
The claim then follows from part~(\ref{L:CzBasic:Her}).

Part~(\ref{L:CzBasic:LCzComm})
is in the discussion after Definition~2.3 of~\cite{KR}.
For part~(\ref{L:CzBasic:3904_UE}),
set $c = u a^{1/2}$.
Then $c \in A$, $c^* c = a$, and $c c^* = u a u^*$.
Apply part~(\ref{L:CzBasic:LCzComm}).
Part~(\ref{L:CzBasic:LCzCommEp})
is Proposition 2.3(ii) of~\cite{ERS}.
(We are grateful to Julian Buck for pointing out this reference.)
Part~(\ref{L:CzBasic:N5})
is the last part of Lemma~3.8 of~\cite{PP}
(which is essentially 1.4 of~\cite{Cu0}).

Part~(\ref{L:CzBasic:MinIter}) is immediate
(and is Lemma~2.5(i) of~\cite{KR}).
For part~(\ref{L:CzBasic:N6}),
use the condition in Proposition 2.4(iv) of~\cite{Rd0}
to find $\rh > 0$ and $w \in A$
such that $w^* (b - \rh)_{+} w = (a - \dt)_{+}$.
Then take $v = [ (b - \rh)_{+} ]^{1/2} w$.
Part~(\ref{L:CzBasic:LCzWithinEp})
is Lemma~2.5(ii) of~\cite{KR}.

Part~(\ref{L:CzBasic:LMinusEp})
is contained in Proposition~2.6 of~\cite{KR}
(and in a slightly different form in
the earlier Proposition~2.4 of~\cite{Rd0}).
Part~(\ref{L:CzBasic:LCzCmpSum})
is Lemma 2.8(ii) of~\cite{KR},
Part~(\ref{L:CzBasic:Orth}) is Lemma 2.8(iii) of~\cite{KR},
and
Part~(\ref{L:CzBasic:CmpDSum}) is Lemma~2.9 of~\cite{KR}.
\end{proof}

We now collect a number of additional facts about Cuntz comparison.
Some are known, but we have not found references for them.
Others appear to be new.

\begin{lem}\label{L:CzSumEp}
Let $A$ be a \ca, let $a, b \in A$ be positive,
and let $\af, \bt \geq 0$.
Then
\[
\big( (a + b - (\af + \bt) \big)_{+}
   \precsim_A (a - \af)_{+} + (b - \bt)_{+}
   \precsim_A (a - \af)_{+} \oplus (b - \bt)_{+}.
\]
\end{lem}

Proposition 2.3(i) of~\cite{ERS}
contains a weaker version of this statement:
for every $\ep > 0$ there is $\dt > 0$
such that
$( a + b - \ep )_{+} \precsim_A (a - \dt)_{+} + (b - \dt)_{+}$.
This proposition also contains a converse:
for every $\ep > 0$ there is $\dt > 0$
such that
$(a - \ep)_{+} + (b - \ep)_{+} \precsim_A ( a + b - \dt )_{+}$.
(We are grateful to Julian Buck for pointing out this reference.)

\begin{proof}[Proof of Lemma~\ref{L:CzSumEp}]
By Lemma~\ref{L:CzBasic}(\ref{L:CzBasic:LMinusEp})
and Lemma~\ref{L:CzBasic}(\ref{L:CzBasic:LCzCmpSum}),
it suffices to prove that for every $\ep > 0$,
we have
\[
\big[ \big( (a + b) - (\af + \bt) \big)_{+} - \ep \big]_{+}
   \precsim_A (a - \af)_{+} + (b - \bt)_{+}.
\]

Let $\ep > 0$.
We have
\[
\| a - (a - \af)_{+} \| \leq \af
\andeqn
\| b - (b - \bt)_{+} \| \leq \bt,
\]
so
\[
\big\| a + b - \big[ (a - \af)_{+} + (b - \bt)_{+} \big] \big\|
 < \af + \bt + \ep.
\]
Therefore, using
Lemma~\ref{L:CzBasic}(\ref{L:CzBasic:MinIter}) at the first step
and
Lemma~\ref{L:CzBasic}(\ref{L:CzBasic:LCzWithinEp})
at the second step,
we have
\[
\big[ \big( (a + b) - (\af + \bt) \big)_{+} - \ep \big]_{+}
  = \big[ (a + b) - (\af + \bt + \ep) \big]_{+}
  \precsim_A (a - \af)_{+} + (b - \bt)_{+}.
\]
This completes the proof.
\end{proof}

The following corollary is a useful generalization of
Lemma~\ref{L:CzBasic}(\ref{L:CzBasic:LCzWithinEp})
and seems not to have been known.

\begin{cor}\label{C:MMvsM}
Let $A$ be a \ca,
and let $\ep > 0$ and $\ld \geq 0$.
Let $a, b \in A$ satisfy $\| a - b \| < \ep$.
Then $(a - \ld - \ep)_{+} \precsim_A (b - \ld)_{+}$.
\end{cor}

\begin{proof}
The hypotheses imply $a - b + \ep \geq 0$
and $(a - b - \ep)_{+} = 0$.
Apply Lemma~\ref{L:CzSumEp}
with $a - b + \ep$ in place of~$a$,
with $b$ as given,
with $\af = 2 \ep$,
and with $\bt = \ld$,
getting
\[
(a - \ld - \ep)_{+}
  = \big[ (a - b + \ep) + b - (2 \ep + \ld) \big]_{+}
  \precsim_A (a - b - \ep)_{+} + (b - \ld)_{+}
  = (b - \ld)_{+}.
\]
This completes the proof.
\end{proof}

\begin{lem}\label{L:CzCompIneq}
Let $A$ be a \ca,
and let $a, b \in A$ satisfy $0 \leq a \leq b$.
Let $\ep > 0$.
Then $(a - \ep)_{+} \precsim_A (b - \ep)_{+}$.
\end{lem}

It is usually not true that $(a - \ep)_{+} \leq (b - \ep)_{+}$.

The following proof, which considerably simplifies our original proof,
was suggested by Leonel Robert, and is used here with his permission.

\begin{proof}[Proof of Lemma~\ref{L:CzCompIneq}]
Multiply the inequality
\[
a - \ep \leq b - \ep \leq (b - \ep)_{+}
\]
on both sides by $(a - \ep)_{+}$,
and use $(a - \ep)_{+} (a - \ep) (a - \ep)_{+} = [(a - \ep)_{+}]^3$,
to get the second step in the following computation:
\[
(a - \ep)_{+}
  \sim_A [(a - \ep)_{+}]^3
  \leq (a - \ep)_{+} (b - \ep)_{+} (a - \ep)_{+}
  \precsim_A (b - \ep)_{+}.
\]
This is the required result.
\end{proof}

\begin{lem}\label{L:C2}
Let $A$ be a \ca,
let $a, g \in A_{+}$ with $0 \leq g \leq 1$,
and let $\ep > 0$.
Then
\[
(a - \ep)_{+}
  \precsim_A \big[ (1 - g) a (1 - g) - \ep \big]_{+} \oplus g.
\]
\end{lem}

\begin{proof}
Set $h = 2 g - g^2$,
so that $(1 - g)^2 = 1 - h$.
We claim that $h \sim_A g$.
Since $0 \leq g \leq 1$,
this follows from Lemma~\ref{L:CzBasic}(\ref{L:CzBasic:LCzFCalc}),
using the \cfn{}
\[
\ld \mapsto \begin{cases}
     2 \ld - \ld^2   & \hspace{3em}  0 \leq \ld \leq 1 \\
         1           & \hspace{3em}  1 < \ld.
    \end{cases}
\]

Set $b = \big[ (1 - g) a (1 - g) - \ep \big]_{+}$.
Using Lemma~\ref{L:CzSumEp} at the second step,
Lemma~\ref{L:CzBasic}(\ref{L:CzBasic:LCzCommEp})
and Lemma~\ref{L:CzBasic}(\ref{L:CzBasic:LCzComm}) at the third step, 
and Lemma~\ref{L:CzBasic}(\ref{L:CzBasic:CmpDSum}) at the last step,
we get
\begin{align*}
( a - \ep )_{+}
& = \big[ a^{1/2} (1 - h) a^{1/2} + a^{1/2} h a^{1/2} - \ep \big]_{+}
     \\
& \precsim_A \big[ a^{1/2} (1 - h) a^{1/2} - \ep \big]_{+}
         \oplus a^{1/2} h a^{1/2}
     \\
& \sim_A \big[ (1 - g) a (1 - g) - \ep \big]_{+}
         \oplus h^{1/2} a h^{1/2}
     \\
& = b \oplus h^{1/2} a h^{1/2}
  \leq b \oplus \| a \| h
  \precsim_A b \oplus g.
\end{align*}
This completes the proof.
\end{proof}

\begin{lem}\label{L-2720KToMn}
Let $A$ be a \ca,
and let $a \in (K \otimes A)_{+}$.
Then for every $\ep > 0$
there are $n \in \N$ and $b \in (M_n \otimes A)_{+}$
such that $(a - \ep)_{+} \sim_A b$.
\end{lem}

We thank Leonel Robert for suggesting the statement,
which strengthens our original statement,
and the proof.
The result seems to be well known,
but we have not found a proof in the literature.

\begin{proof}[Proof of Lemma~\ref{L-2720KToMn}]
Choose $n \in \N$ and $c \in (M_n \otimes A)_{+}$
such that $\| c - a \| < \ep$.
By Lemma~2.2 of~\cite{KR2},
there is $d \in K \otimes A$ such that $d^* c d = (a - \ep)_{+}$.
Set $b = c^{1/2} d d^* c^{1/2}$.
Then $b \in (M_n \otimes A)_{+}$.
Using Lemma~\ref{L:CzBasic}(\ref{L:CzBasic:LCzComm}) at the first step,
we get
$b \sim_A d^* c d = (a - \ep)_{+}$.
\end{proof}

\begin{lem}\label{L-2726IdemC}
Let $A$ be a \ca,
let $b, c \in A_{+}$ satisfy $b c = c$,
and let $\bt \in [0, 1)$.
Then there is $\gm > 0$ such that
$c \leq \gm (b - \bt)_{+}$.
\end{lem}

\begin{proof}
\Wolog{} $c \neq 0$,
so $\| b \| \geq 1$.
We claim that if $f \colon [0, \, \| b \|] \to [0, \infty)$
is any \cfn,
then $f (b) c = c f (b)= f (1) c$.
By continuity,
it suffices to prove the claim when $f$ is a polynomial.
This case follows from the relation
$b^k c = c b^k = c$ for all $k \in \Nz$.

Apply the claim with the function
$f ( \ld ) = [ ( \ld - \bt )_{+} ]^{1/2}$ for $\ld \in [0, \infty)$.
We get
\[
(1 - \bt) c = f (b) c f (b) \leq \| c \| (b - \bt)_{+}.
\]
The lemma is then proved by taking $\gm = (1 - \bt)^{-1} \| c \|$.
\end{proof}

\begin{lem}\label{L_4624_CuTr}
Let $A$ and $B$ be \ca{s},
and let $A \otimes B$ denote any C*~tensor product.
Let $a_1, a_2 \in (K \otimes A)_{+}$ and let $b \in (K \otimes B)_{+}$.
If $\langle a_1 \rangle \leq \langle a_2 \rangle$
in ${\operatorname{Cu}} (A)$,
then $\langle a_1 \otimes b \rangle \leq \langle a_2 \otimes b \rangle$
in ${\operatorname{Cu}} (A \otimes B)$.
\end{lem}

\begin{proof}
Replacing $A$ with $K \otimes A$
and $B$ with $K \otimes B$,
we see that it is enough to show that if
$a_1, a_2 \in A_{+}$ satisfy $a_1 \precsim_A a_2$,
and if $b \in B_{+}$,
then $a_1 \otimes b \precsim_{A \otimes B} a_2 \otimes b$.

Let $\ep > 0$.
We find $z \in A \otimes B$
such that
$\big\| z^* (a_2 \otimes b) z - a_1 \otimes b \big\| < \ep$.
Set
\[
\dt = \frac{\ep}{\| a_1 \| + \| b \| + 1}.
\]
Using an approximate identity for~$B$,
find $y \in B_{+}$
such that $\| y \| \leq 1$
and $\| y b y - b \| < \dt$.
By definition,
there is $x \in A$ such that $\| x^* a_2 x - a_1 \| < \dt$.
Set $z = x \otimes y$.
Then, using $\| y \| \leq 1$,
we get
\begin{align*}
\big\| z^* (a_2 \otimes b) z - a_1 \otimes b \big\|
& = \big\| x^* a_2 x \otimes y b y - a_1 \otimes b \big\|
\\
& \leq \| x^* a_2 x - a_1 \| \cdot \| y b y \|
          + \| a_1 \| \cdot \| y b y - b \|
\\
&  \leq \dt \| b \| + \| a_1 \| \dt
  < \ep.
\end{align*}
This completes the proof.
\end{proof}

The next several lemmas
will be used to relate Cuntz comparison and ideals.
See Proposition~\ref{P-2728CzIdeal}.

\begin{lem}\label{L_2627_SumDom}
Let $A$ be a \ca,
let $n \in \N$,
and let $a_1, a_2, \ldots, a_n \in A$.
Set $a = \sum_{k = 1}^n a_k$
and $x = \sum_{k = 1}^n a_k^* a_k$.
Then $a^* a \in {\overline{x A x}}$.
\end{lem}

\begin{proof}
\Wolog{} $\| a_k \| \leq 1$ for $k = 1, 2, \ldots, n$.
Let $\ep > 0$.
Set $\dt = \frac{1}{8} \ep^2 n^{-4}$.
Since
$a_1^* a_1, a_2^* a_2, \ldots, a_n^* a_n \in {\overline{x A x}}$,
there exists $c \in {\overline{x A x}}$ such that
$\| c a_k^* a_k - a_k^* a_k \| < \dt$
for $k = 1, 2, \ldots, n$
and  $0 \leq c \leq 1$.
Then
\begin{align*}
\| c a_k^* - a_k^* \|^2
& = \| c a_k^* a_k c - a_k^* a_k c - c a_k^* a_k + a_k^* a_k \|
     \\
& \leq \| c a_k^* a_k - a_k^* a_k \| \cdot \| c \|
         + \| c a_k^* a_k - a_k^* a_k \|
  < 2 \dt,
\end{align*}
so $\| c a_k^* - a_k^* \| < \sqrt{2 \dt}$.
Therefore $\| a_k c - a_k \| < \sqrt{2 \dt}$.
Summing over~$k$,
we get
\[
\| c a^* - a^* \| < n \sqrt{2 \dt}
\andeqn
\| a c - a \| < n \sqrt{2 \dt}.
\]
Using $\| a \| \leq n$ and $\| c \| \leq 1$ at the second step,
we then have
\begin{align*}
\| c a^* a c - a^* a \|
& \leq \| c a^* - a^* \| \cdot \| a \| \cdot \| c \|
       + \| a^* \| \| a c - a \|
     \\
& < n \sqrt{2 \dt} \cdot n + n \cdot n \sqrt{2 \dt}
  = 2 n^2 \sqrt{2 \dt}
  = \ep.
\end{align*}
Since $c a^* a c \in {\overline{x A x}}$
and $\ep > 0$ is arbitrary,
the conclusion follows.
\end{proof}

\begin{lem}\label{L_4619_PosIdeal}
Let $A$ be a \ca{} and let $a \in A_{+}$.
Let $b \in {\overline{A a A}}$ be positive.
Then for every $\ep > 0$ there exist $n \in \N$
and $x_1, x_2, \ldots, x_n \in A$
such that
$\left\|{\rule{0em}{2ex}} b - \sum_{k = 1}^n x_k^* a x_k \right\|
 < \ep$.
\end{lem}

This result is used without proof in the proof
of Proposition 2.7(v) of~\cite{KR}.

\begin{proof}[Proof of Lemma~\ref{L_4619_PosIdeal}]
\Wolog{} $\| b \| \leq 1$ and $\ep < 1$.
Since also $b^{1/2} \in {\overline{A a A}}$,
there are $n \in \N$
and $y_1, y_2, \ldots, y_n, z_1, z_2, \ldots, z_n \in A$
such that
the element $c = \sum_{k = 1}^n y_k a z_k$ satisfies
$\big\| b^{1/2} - c \big\| < \frac{\ep}{4}$.
Then
$\| c \| < 2$,
so
\[
\| b - c^* c \|
  \leq \big\| b^{1/2} - c^* \big\| \cdot \big\| b^{1/2} \big\|
         + \| c^* \| \cdot \big\| b^{1/2} - c \big\|
  < \frac{\ep}{4} + 2 \left( \frac{\ep}{4} \right)
  = \frac{3 \ep}{4}.
\]

Set
\[
r = \sum_{k = 1}^n z_k^* a y_k^* y_k a z_k,
\,\,\,\,
M = \max \big( \| y_1 \|, \, \| y_2 \|, \, \ldots, \, \| y_n \| \big),
\,\,\,\,
{\mbox{and}}
\,\,\,\,
s = M^2 \sum_{k = 1}^n z_k^* a^2 z_k.
\]
Combining Lemma~\ref{L_2627_SumDom}
and Lemma~\ref{L:CzBasic}(\ref{L:CzBasic:Her}),
we get $c^* c \precsim_A r$.
Also $r \leq s$.
So there is $v \in A$ such that
$\| v^* s v - c^* c \| < \frac{\ep}{4}$.
Set $x_k = a^{1/2} z_k v$.
Then
\[
\left\| b - \sum_{k = 1}^n x_k^* a x_k \right\|
  = \| b - v^* s v \|
  \leq \| b - c^* c \| + \| c^* c - v^* s v \|
  < \frac{3 \ep}{4} + \frac{\ep}{4}
  = \ep.
\]
This completes the proof.
\end{proof}

The following corollary will not be needed until later.

\begin{cor}\label{C_4619_PosSmp}
Let $A$ be a simple unital \ca{}
and let $x \in A_{+} \setminus \{ 0 \}$.
Then there exist $n \in \N$
and $b_1, b_2, \ldots, b_n \in A$
such that $\sum_{j = 1}^n b_j x b_j^* = 1$.
\end{cor}

\begin{proof}
Apply Lemma~\ref{L_4619_PosIdeal}
with $a = 1$ and $\ep = \frac{1}{2}$,
getting
$c_1, c_2, \ldots, c_n \in A$
such that $z = \sum_{j = 1}^n c_j x c_j^*$
satisfies $\| z - 1 \| < \frac{1}{2}$.
Then set $b_j = z^{-1/2} c_j$ for $j = 1, 2, \ldots, n$.
\end{proof}

One direction of the following result is essentially in~\cite{KR}.

\begin{prp}\label{P-2728CzIdeal}
Let $A$ be a \ca{} and let $a, b \in A_{+}$.
Then $b$ is in the ideal generated by~$a$
\ifo{} for every $\ep > 0$ there is $n \in \N$
such that $(b - \ep)_{+} \precsim_A 1_{M_n} \otimes a$.
\end{prp}

\begin{proof}
If $b$ is in the ideal of $A$ generated by~$a$
and $\ep > 0$,
then Proposition 2.7(v) of~\cite{KR}
provides $n \in \N$
such that $(b - \ep)_{+} \precsim_A 1_{M_n} \otimes a$.

We prove the converse.
Let $\ep > 0$.
We will find $x$ in the ideal generated by $a$
such that $\| x - b \| < \ep$.
Choose $n \in \N$ such that
$\big( b - \tfrac{\ep}{2} \big)_{+} \precsim_A 1_{M_n} \otimes a$.
Let $(e_{j, k})_{j, k = 1, 2, \ldots, n}$
be the standard system of matrix units for~$M_n$.
By definition,
there is $v \in M_n (A)$ such that
\[
\big\| v (1 \otimes a) v^*
   - e_{1, 1} \otimes \big( b - \tfrac{\ep}{2} \big)_{+} \big\|
< \frac{\ep}{2}.
\]
Then
\begin{equation}\label{Eq_3X08_6A}
\big\| (e_{1, 1} \otimes 1) v (1 \otimes a) v^* (e_{1, 1} \otimes 1)
   - e_{1, 1} \otimes \big( b - \tfrac{\ep}{2} \big)_{+} \big\|
< \frac{\ep}{2}.
\end{equation}
There are $v_{j, k} \in A$
for $j, k = 1, 2, \ldots, n$
such that $v = \sum_{j, k = 1}^n e_{j, k} \otimes v_{j, k}$.
Set $x = \sum_{j, k = 1}^n v_{1, j} a v_{1, k}^*$.
Clearly $x$ is in the ideal generated by~$a$.
The inequality~(\ref{Eq_3X08_6A})
implies that
\[
\big\| e_{1, 1} \otimes x
   - e_{1, 1} \otimes \big( b - \tfrac{\ep}{2} \big)_{+} \big\|
< \frac{\ep}{2}.
\]
So
\[
\| x - b \|
 \leq \big\| x - \big( b - \tfrac{\ep}{2} \big)_{+} \big\|
    + \big\| \big( b - \tfrac{\ep}{2} \big)_{+} - b \big\|
 < \frac{\ep}{2} + \frac{\ep}{2}
 = \ep.
\]
This completes the proof.
\end{proof}

We finish this section by recalling material on supremums
in the Cuntz semigroup,
functional on the Cuntz semigroup,
and quasitraces.

Recall that a subset $S$ of an ordered set is said to be upwards
directed if for every $\et_1, \et_2 \in S$ there is $\mu \in S$
such that $\et_1 \leq \mu$ and $\et_2 \leq \mu$.

\begin{thm}\label{T_2Y25_CxSup}
The Cuntz semigroup has the following properties.
\begin{enumerate}
\item\label{T_2Y25_CxSup_Ex}
Let $A$ be a \ca,
and let $S \subset {\operatorname{Cu}} (A)$
be a countable upwards directed subset.
Then $\sup (S)$ exists in ${\operatorname{Cu}} (A)$.
\item\label{T_2Y25_CxSup_Prv}
Let $A$ and $B$ be C*-algebras,
and let $\ph \colon A \to B$ be a \hm.
Let $S \subset {\operatorname{Cu}} (A)$
be a countable upwards directed subset.
Then
$\sup ( {\operatorname{Cu}} (\ph) (S))
 = {\operatorname{Cu}} (\ph) (\sup (S))$.
\end{enumerate}
\end{thm}

\begin{proof}
Part~(\ref{T_2Y25_CxSup_Ex}) is Theorem~4.19 of~\cite{APT}.
Part~(\ref{T_2Y25_CxSup_Prv}) is contained in
Theorem~4.35 of~\cite{APT}; see Definition~4.1 of~\cite{APT}.
\end{proof}

\begin{ntn}\label{N_4620_QT}
For a \uca~$A$,
we denote by ${\operatorname{T}} (A)$
the set of tracial states on~$A$.
We also denote by ${\operatorname{QT}} (A)$
the set of normalized $2$-quasitraces
on~$A$ (Definition II.1.1 of~\cite{BH};
Definition 2.31 of~\cite{APT}).
\end{ntn}

\begin{dfn}\label{D:dtau}
Let $A$ be a stably finite unital \ca,
and let $\ta \in {\operatorname{QT}} (A)$.
Define $d_{\ta} \colon \Mi (A)_{+} \to [0, \infty)$
by $d_{\ta} (a) = \lim_{n \to \infty} \ta (a^{1/n})$
for $a \in \Mi (A)_{+}$.
Further (the use of the same notation should cause no confusion)
define $d_{\ta} \colon (K \otimes A)_{+} \to [0, \infty]$
by the same formula, but now
for $a \in (K \otimes A)_{+}$.
We also use the same notation for the corresponding functions
on ${\operatorname{Cu}} (A)$ and $W (A)$,
as in Proposition~\ref{P_4819_dtau} below.
\end{dfn}

\begin{prp}\label{P_4819_dtau}
Let $A$ be a stably finite unital \ca,
and let $\ta \in {\operatorname{QT}} (A)$.
Then $d_{\ta}$ as in Definition~\ref{D:dtau} is well defined
on ${\operatorname{Cu}} (A)$ and $W (A)$.
That is, if $a, b \in (K \otimes A)_{+}$
satisfy $a \sim_A b$,
then $d_{\ta} (a) = d_{\ta} (b)$.
\end{prp}

\begin{proof}
This is part of Proposition~4.2 of~\cite{ERS}.
\end{proof}

Also see the beginning of Section~2.6 of~\cite{APT},
especially the proof of Theorem 2.32 there.
It follows that $d_{\ta}$ defines a state on $W (A)$.
Thus (Theorem~\ref{T_2Y25_DTau}(\ref{T_2Y25_DTau_W}) below)
the map $\ta \mapsto d_{\ta}$
is a bijection from ${\operatorname{QT}} (A)$
to the lower semi\ct{} dimension functions on~$A$.
To state the corresponding result
with ${\operatorname{Cu}} (A)$ in place of $W (A)$,
we first recall the following definition
from the beginning of Section~4.1 of~\cite{ERS}.

\begin{dfn}\label{D_2Y25_CzFcnl}
Let $S$ be an ordered semigroup with a zero element
and such that every nondecreasing sequence in~$S$
has a supremum.
Then a {\emph{functional}} on~$S$
is a function $\om \colon S \to [0, \infty]$
which satisfies:
\begin{enumerate}
\item\label{2Y24_CuFcnl_Add}
$\om (\et + \mu) = \om (\et) + \om (\mu)$ for all $\et, \mu \in S$.
\item\label{2Y24_CuFcnl_Order}
If $\et, \mu \in S$ satisfy $\et \leq \mu$,
then $\om (\et) \leq \om (\mu)$.
\item\label{2Y24_CuFcnl_Zero}
$\om (0) = 0$.
\item\label{2Y24_CuFcnl_Sup}
If $\et_0 \leq \et_1 \leq \cdots$ in~$S$,
and $\et = \sup \big( \{ \et_n \colon n \in \Nz \} \big)$,
then $\om (\et) = \sup \big( \{ \om (\et_n) \colon n \in \Nz \} \big)$.
\end{enumerate}
\end{dfn}

\begin{thm}\label{T_2Y25_DTau}
Let $A$ be a \uca.
\begin{enumerate}
\item\label{T_2Y25_DTau_W}
The assignment $\ta \mapsto d_{\ta}$ defines an affine bijection from
${\operatorname{QT}} (A)$
to the space of normalized
lower semi\ct{} dimension functions on~$A$.
\item\label{T_2Y25_DTau_Cu}
The assignment $\ta \mapsto d_{\ta}$ defines a bijection
from ${\operatorname{QT}} (A)$
to the space of functionals $\om$ on ${\operatorname{Cu}} (A)$
such that $\om ( \langle 1 \rangle ) = 1$.
\end{enumerate}
\end{thm}

\begin{proof}
Part~(\ref{T_2Y25_DTau_W})
follows from Theorem II.2.2 of~\cite{BH},
which gives the corresponding bijection
between $2$-quasitraces and dimension functions
which are not necessarily normalized but are finite everywhere.

We prove part~(\ref{T_2Y25_DTau_Cu}).
By Proposition~4.2 of~\cite{ERS},
the assignment $\ta \mapsto d_{\ta}$ defines a bijection
from the space of not necessarily normalized
lower semi\ct{} $2$-quasitraces on~$A$
to the space of functionals on ${\operatorname{Cu}} (A)$.
Therefore it suffices to show that if $\ta$
is a $2$-quasitrace on $A$ with $\ta (1) = 1$,
then $\ta$ is lower semi\ct.
This follows from Corollary II.2.5(iii) of~\cite{BH},
according to which quasitraces which are finite everywhere,
even on a not necessarily unital \ca,
are automatically \ct.
\end{proof}

The following result is well known,
but we do not know a reference.

\begin{lem}\label{L_4815_lsc}
Let $A$ be a \uca,
and let $a \in (K \otimes A)_{+}$.
Then the function $\ta \mapsto d_{\ta} (a)$
is a lower semi\cfn{} from $\QT (A)$ to $[0, \infty]$.
\end{lem}

\begin{proof}
\Wolog{} $\| a \| \leq 1$.
If there is $n \in \N$ such that $a \in M_n \otimes A$,
then $\ta \mapsto d_{\ta} (a)$
is the supremum of the \ct{}
real valued functions $\ta \mapsto \ta (a^{1/n})$
on $\QT (A)$.
In general,
for $n \in \N$ let $p_n \in K$ be the identity of $M_n$.
The function $d_{\ta}$ is lower semi\ct{} on $(K \otimes A)_{+}$.
So $\ta \mapsto d_{\ta} (a)$
is the supremum of the lower semi\cfn{s}
$\ta \mapsto d_{\ta} \big( (p_n \otimes 1) a (p_n \otimes 1) \big)$
on $\QT (A)$.
\end{proof}

We will frequently use the following standard fact without comment.
Again, we did not find a reference.

\begin{lem}\label{L_4822_infdTauPos}
Let $A$ be a simple \uca,
and let $a \in (K \otimes A)_{+} \setminus \{ 0 \}$.
Then $\inf_{\tau \in \QT (A)} d_{\tau} (a) > 0$.
\end{lem}

\begin{proof}
Since $\ta \mapsto d_{\ta} (a)$ is lower semi\ct{}
(Lemma~\ref{L_4815_lsc})
and $\QT (A)$ is compact,
it suffices to show that for $\ta \in \QT (A)$
we have $d_{\ta} (a) > 0$.

For $n \in \N$ let $p_n \in K$ be the identity of $M_n$.
The sequence
$\big( d_{\ta} \big(
  (p_n \otimes 1) a (p_n \otimes 1) \big) \big)_{n \in \N}$
is nondecreasing and,
by semicontinuity of $d_{\ta}$ on $(K \otimes A)_{+}$,
converges to $d_{\ta} (a)$.
So it suffices to consider $a \in M_{\infty} (A)_{+} \setminus \{ 0 \}$.
Replacing $A$ with $M_n (A)$ for suitable $n \in \N$,
and renormalizing~$\ta$,
we may assume $a \in A_{+} \setminus \{ 0 \}$.

By scaling,
\wolog{} $\| a \| = 1$.
Then $d_{\tau} (a) \geq \tau (a)$.
We have $\ta (a) > 0$
by the comment before Theorem 2.32 of~\cite{APT}.
\end{proof}

We recall the following definition.
It is in condition~{\textbf{(O4)}} in Definition~4.1 of~\cite{APT},
but the name there is different
(``way below'' instead of ``compactly contained in'').

\begin{dfn}\label{D_4621_CC}
Let $S$ be an ordered commutative semigroup
in which supremums of countable upwards directed sets exist.
Let $\et, \mu \in S$.
We write $\mu \ll \et$ if whenever
$\et_0 \leq \et_1 \leq \cdots$ in~$S$,
and $\et \leq \sup \big( \{ \et_n \colon n \in \Nz \} \big)$,
then there exists $n$ such that $\mu \leq \et_n$.
We say that $\mu$ is {\emph{compactly contained in}}~$\et$.
\end{dfn}

\begin{lem}\label{L_3905_CuWay}
Let $A$ be a \ca.
\begin{enumerate}
\item\label{L_3905_CuWay_SupEp}
Let $a \in (K \otimes A)_{+}$.
Let $(\ep_n)_{n \in \N}$
be any sequence in $(0, \infty)$ which decreases to zero.
Then
\[
\langle a \rangle
 = \sup \big(
    \big\{ \langle (a - \ep_n)_{+} \rangle \colon n \in \N \big\} \big).
\]
\item\label{L_3905_CuWay_EpWB}
Let $a \in (K \otimes A)_{+}$ and let $\ep > 0$.
Then $\langle (a - \ep)_{+} \rangle \ll \langle a \rangle$.
\item\label{L_3905_CuWay_PjWB}
Let $p \in (K \otimes A)_{+}$ be a \pj.
Then $\langle p \rangle \ll \langle p \rangle$.
\item\label{L_3905_CuWay_Ord}
If $\et, \mu \in {\operatorname{Cu}} (A)$
and $\et \ll \mu$,
then $\et \leq \mu$.
\item\label{L_3905_CuWay_Tr}
If $\et, \ld, \mu \in {\operatorname{Cu}} (A)$
satisfy $\et \leq \ld$ and $\ld \ll \mu$,
then $\et \ll \mu$.
\end{enumerate}
\end{lem}

\begin{proof}
Theorem 4.33 of~\cite{APT} implies that
${\operatorname{Cu}} (A)$ as defined here
(namely,
${\operatorname{Cu}} (A) = W (K \otimes A)$)
is the same as in Definition~4.5 of~\cite{APT}.
Given this, parts (\ref{L_3905_CuWay_SupEp})
and~(\ref{L_3905_CuWay_EpWB}) are Lemma~4.36 of~\cite{APT}.
Part~(\ref{L_3905_CuWay_PjWB}) is immediate from
part~(\ref{L_3905_CuWay_EpWB}).
Part~(\ref{L_3905_CuWay_Ord}) is immediate
from Definition~\ref{D_4621_CC},
and part~(\ref{L_3905_CuWay_Tr})
follows from the comments after
Definition~4.1 of~\cite{APT},
together with the fact
(Theorem~4.20 of~\cite{APT})
that ${\operatorname{Cu}} (A)$
is in fact in the category ${\mathbf{Cu}}$
of Definition~4.1 of~\cite{APT}.
\end{proof}

\section{Cuntz comparison in simple C*-algebras}\label{Sec_CSimp}

\indent
In this section,
we give results on Cuntz comparison which are special
to simple \ca{s} not of type~I,
or at least to \ca{s} not of type~I.
In some of them,
Cuntz comparison plays only a secondary role.

The main results are a strong form of the existence of many
orthogonal equivalent elements
(see Lemma~\ref{L:NT1Orth}),
a kind of weak approximate divisibility result
(Lemma~\ref{L:Smaller}),
and Lemma~\ref{L-2726Big},
which is a form of the statement that
in a finite simple \uca,
if $0 \leq g \leq 1$
and $g$ is in a ``small'' \hsa,
then $1 - g$ is ``large''.

We first give some results depending on the existence of
comparable orthogonal elements.
We record the following useful fact from~\cite{AS}.

\begin{lem}\label{L-2817Sp01}
Let $A$ be a simple \ca{} which is not of type~I.
Then there exists $a \in A_{+}$ such that
$\spec (a) = [0, 1]$.
\end{lem}

\begin{proof}
The discussion before~(1) on page~61 of~\cite{AS}
shows that $A$ is not scattered in the sense of~\cite{AS}.
The conclusion therefore follows
from the argument in~(4) on page~61 of \cite{AS}.
\end{proof}

\begin{lem}\label{L:NT1Orth}
Let $A$ be a unital \ca{} which is not of type~I.
Let $n \in \N$.
Then there exists a unitary $u \in A$ which is homotopic to~$1$,
and a nonzero positive element $a \in A$,
such that the elements
\[
a, \, u a u^{-1}, \, u^2 a u^{-2}, \, \ldots, \, u^n a u^{-n}
\]
are pairwise orthogonal.
\end{lem}

The proof uses heavy machinery,
and there ought to be a simpler proof,
particularly when $A$ is simple.

\begin{proof}[Proof of Lemma~\ref{L:NT1Orth}]
Fix $n \in \N$.

We first prove the result for the unitized cone $(C M_{n + 1})^{+}$
in place of~$A$.
We make the identification
\[
(C M_{n + 1})^{+}
  = \big\{ f \in C ([0, 1], \, M_{n + 1}) \colon
                f (0) \in \C \cdot 1 \big\}.
\]
Let $(e_{j, k})_{0 \leq j, k \leq n}$ be the standard system
of matrix units for $M_{n + 1}$.
(The indexing starts at~$0$.)
Define $a \in (C M_{n + 1})^{+}$ by
$a (\ld) = \big( \ld - \tfrac{1}{2} \big)_{+} e_{0, 0}$
for $\ld \in [0, 1]$.
Let $s$ be the cyclic shift unitary
\[
s = e_{0, n} + \sum_{j = 1}^n e_{j, \, j - 1}.
\]
Choose a \ct{} path $\ld \mapsto w (\ld)$
in the unitary group of $M_{n + 1}$
such that $w (0) = 1$ and $w (1) = s$.
Define a unitary $u \in (C M_{n + 1})^{+}$ by
\[
u (\ld) = \begin{cases}
     w (2 \ld)   & \hspace{3em}  0 \leq \ld \leq \tfrac{1}{2}  \\
         s       & \hspace{3em}  \tfrac{1}{2} \leq \ld \leq 1.
    \end{cases}
\]
Then $u$ and $a$ satisfy the conclusion of the lemma.

To prove the lemma for a general \ca~$A$,
we prove the existence of an injective unital \hm{}
from $(C M_{n + 1})^{+}$ to~$A$.
Let
\[
D_0 = \bigotimes_{m = 1}^{\infty} M_{n + 1}
\andeqn
D = D_0 \otimes M_{n + 1}.
\]
(Of course $D \cong D_0$.)
Corollary 6.7.4 of~\cite{Pd}
provides a subalgebra $B \subset A$
and a surjective \hm{} $\pi \colon B \to D$.
Replacing $B$ by $B + \C \cdot 1$
and extending $\pi$ in the obvious way,
if necessary,
we may assume that $B$ contains the identity of~$A$.
Let $\io \colon B \to A$ be the inclusion.

Choose (Lemma~\ref{L-2817Sp01})
some $b \in (D_0)_{+}$ such that $\spec (b) = [0, 1]$.
There is a \hm{} $\ph_0$ from
$C M_{n + 1} = C_0 ( (0, 1]) \otimes M_{n + 1}$ to $D$
such that $\ph_0 (f \otimes x) = f (b) \otimes x$
for all $f \in C_0 ( (0, 1])$
and all $x \in M_{n + 1}$.
Let $\ph \colon (C M_{n + 1})^{+} \to D$
be the unitization of~$\ph_0$.
Then $\ph$ is injective.

Since $C M_{n + 1}$ is a projective \ca{}
(see Theorem 10.2.1 of~\cite{Lr}),
there exists a \hm{}  $\ps_0 \colon C M_{n + 1} \to B$
such that $\pi \circ \ps_0 = \ph_0$.
Let $\ps \colon (C M_{n + 1})^{+} \to B$
be the unitization of~$\ps_0$.
Since $\ph$ is injective, so is~$\ps$.
Then $\io \circ \ps \colon (C M_{n + 1})^{+} \to A$
is an injective unital \hm, as required.
\end{proof} 

\begin{lem}\label{L:NT1OrthNonUnital}
Let $A$ be a nonunital \ca{} which is not of type~I.
Let $n \in \N$.
Then there exists a unitary $u \in A^+$ which is homotopic to~$1$,
and a nonzero positive element $a \in A$,
such that the elements
\[
a, \, u a u^{-1}, \, u^2 a u^{-2}, \, \ldots, \, u^n a u^{-n}
\]
are pairwise orthogonal and in $A$.
\end{lem}

\begin{proof}
Apply Lemma \ref{L:NT1Orth} to $A^{+}$,
obtaining $a \in A_{+} \setminus \{ 0 \}$
and a unitary $u \in A$
such that
$a, \, u a u^{-1}, \, u^2 a u^{-2}, \, \ldots, \, u^n a u^{-n}$
are pairwise orthogonal.
Let $\pi$ be the standard unital homomorphism $A^+ \rightarrow \C$.
Then,
using commutativity of~$\C$ at the first step,
\[
\pi (a)^2
 = \pi (a)  \left( \pi (u) \pi (a) \pi (u^*) \right)
 = \pi (a) \cdot \pi ( u a u^* )
 = 0.
\]
Therefore $\pi (a) = 0$.
So $a \in A$. 
Since $A$ is an ideal of $A^{+}$,
we get
\[
u a u^{-1}, \, u^2 a u^{-2}, \, \ldots, \, u^n a u^{-n} \in A
\]
as well.
\end{proof}

\begin{lem}\label{L:DivInSmp}
Let $A$ be a simple \ca{} which is not of type~I.
Let $a \in A_{+} \setminus \{ 0 \}$,
and let $l \in \N$.
Then there exist $b_1, b_2, \ldots, b_l \in A_{+} \setminus \{ 0 \}$
such that $b_1 \sim_A b_2 \sim_A \cdots \sim_A b_l$,
such that $b_j b_k = 0$ for $j \neq k$,
and such that $b_1 + b_2 + \cdots + b_l \in {\overline{a A a}}$.
\end{lem}

\begin{proof}
Replacing $A$ by ${\overline{a A a}}$,
it suffices to prove the result
without the conclusion
$b_1 + b_2 + \cdots + b_l \in {\overline{a A a}}$.
Use
Lemma~\ref{L:NT1OrthNonUnital}
to find $b \in A_{+} \setminus \{ 0 \}$
and a unitary $u \in A^{+}$
such that
\[
b_1 = b,
\,\,\,\,\,\,
b_2 = u b u^{-1},
\,\,\,\,\,\,
\ldots,
\,\,\,\,\,\,
b_l = u^{l - 1} b u^{- (l - 1)}
\]
are pairwise orthogonal.
Lemma \ref{L:CzBasic}(\ref{L:CzBasic:3904_UE})
implies that $b_1 \sim_A b_2 \sim_A \cdots \sim_A b_l$.
\end{proof}

\begin{cor}\label{C-2718CuDiv}
Let $A$ be a simple unital infinite dimensional \ca.
Then for every $\ep > 0$ there is $a \in A_{+} \setminus \{ 0 \}$
such that for all $\ta \in {\operatorname{QT}} (A)$
we have $d_{\ta} (a) < \ep$.
\end{cor}

\begin{proof}
Choose $n \in \N$ such that $\frac{1}{n} < \ep$.
Use Lemma~\ref{L:DivInSmp} to choose
$b_1, b_2, \ldots, b_n \in A_{+} \setminus \{ 0 \}$
such that $b_1 \sim_A b_2 \sim_A \cdots \sim_A b_n$,
and such that $b_j b_k = 0$ for $j \neq k$.
Then for every $\ta \in {\operatorname{QT}} (A)$
we have
\[
\sum_{k = 1}^n d_{\ta} (b_k)
 = d_{\ta} \left( \sum_{k = 1}^n b_k \right)
 \leq 1
\andeqn
d_{\ta} (b_1)
 = d_{\ta} (b_2) = \cdots = d_{\ta} (b_n).
\]
So, with $a=b_1$, we have $d_{\ta} (a) \leq \frac{1}{n} < \ep$.
\end{proof}

\begin{lem}\label{L-2718CuSub}
Let $A$ be a simple \ca,
and let $B \subset A$ be a nonzero \hsa.
Let $n \in \N$,
and let $a_1, a_2, \ldots, a_n \in A_{+} \setminus \{ 0 \}$.
Then there exists $b \in B_{+} \setminus \{ 0 \}$
such that $b \precsim_A a_j$ for $j = 1, 2, \ldots, n$.
\end{lem}

\begin{proof}
We prove this by induction on~$n$,
for convenience requiring in addition that $\| b \| \leq 1$.
For $n = 0$,
the Cuntz subequivalence condition is vacuous,
so we can take $b$ to be any nonzero positive element of~$B$
such that $\| b \| \leq 1$.

Suppose now the result is known for some~$n$,
and let $a_1, a_2, \ldots, a_{n + 1} \in A_{+} \setminus \{ 0 \}$.
\Wolog{} $\| a_j \| \leq 1$ for $j = 1, 2, \ldots, n + 1$.
The induction hypothesis provides $b_0 \in B_{+} \setminus \{ 0 \}$
such that $b_0 \precsim_A a_j$ for $j = 1, 2, \ldots, n$.
Since $A$ is simple, there is $x \in A$
such that the element $z = b_0^{1/2} x a_{n + 1}^{1/2}$ is nonzero.
We may require $\| x \| \leq 1$.
Set $b = z^* z \neq 0$.
Then $b \leq b_0$,
so $b \in B$
and $b \precsim_A b_0 \precsim_A a_j$ for $j = 1, 2, \ldots, n$.
Also $z z^* \leq a_{n + 1}$,
so, using Lemma~\ref{L:CzBasic}(\ref{L:CzBasic:LCzComm})
at the first step,
$b \sim_A z z^* \precsim_A a_{n + 1}$.
This completes the proof.
\end{proof}

\begin{lem}\label{L:Smaller}
Let $A$ be a simple infinite dimensional \ca{} which is not of type~I.
Let $b \in A_{+} \setminus \{ 0 \}$, let $\ep > 0$, and let $n \in \N$.
Then there are $c \in A_{+}$ and $y \in A_{+} \setminus \{ 0 \}$
such that, in $W (A)$,
we have
\[
n \langle (b - \ep)_{+} \rangle \leq (n + 1) \langle c \rangle
\andeqn
\langle c \rangle + \langle y \rangle \leq \langle b \rangle.
\]
\end{lem}

\begin{proof}
We divide the proof into two cases.
First assume that $\spec (b) \cap (0, \ep) \neq \varnothing$.
Then there is a \cfn{} $f \colon [0, \infty) \to [0, \infty)$
which is zero on $\{ 0 \} \cup [\ep, \infty)$
and such that $f (b) \neq 0$.
We take $c = (b - \ep)_{+}$ and $y = f (b)$.

Now suppose that $\spec (b) \cap (0, \ep) = \varnothing$.
Define a \cfn{} $f \colon [0, \infty) \to [0, \infty)$ by
\[
f (\ld) = \begin{cases}
     \ep^{-1} \ld   & \hspace{3em}  0 \leq \ld \leq \ep  \\
         1          & \hspace{3em}  \ep \leq \ld \leq 1.
    \end{cases}
\]
Then $f (b)$ is a projection
and Lemma \ref{L:CzBasic}(\ref{L:CzBasic:LCzFCalc})
implies that $f (b) \sim_A b$.
Also
\[
(b - \ep)_{+}
 \leq b
 \sim f (b)
 \sim \big( f (b) - \tfrac{1}{2} \big)_{+}.
\]
Replacing $b$ by $f (b)$ and $A$ by $f (b) A f (b)$,
we may therefore assume that $A$ is unital,
that $b = 1$, and that $\ep = \tfrac{1}{2}$.
Thus $(b - \ep)_{+} \sim 1$.

Lemma~\ref{L:NT1Orth} provides $a \in A_{+}$ and a unitary $u \in A$
such that the elements
\[
a, \, u a u^{-1}, \, u^2 a u^{-2}, \, \ldots, \, u^n a u^{-n}
\]
are pairwise orthogonal.
\Wolog{} $\| a \| = 1$.
Define \cfn{s} $g_1, g_2, g_3 \colon [0, \infty) \to [0, 1]$ by
\[
g_1 (\ld) = \begin{cases}
     3 \ld   & \hspace{3em}  0 \leq \ld \leq \tfrac{1}{3}  \\
       1     & \hspace{3em}  \tfrac{1}{3} \leq \ld,
    \end{cases}
\]
\[
g_2 (\ld) = \begin{cases}
     0           & \hspace{3em}  0 \leq \ld \leq \tfrac{1}{3}  \\
     3 \ld - 1   & \hspace{3em}  \tfrac{1}{3} \leq \ld \leq \tfrac{2}{3}
                     \rule{0em}{2.5ex}  \\
     1           & \hspace{3em}  \tfrac{2}{3} \leq \ld,
                     \rule{0em}{2.5ex}
    \end{cases}
\]
and
\[
g_3 (\ld) = \begin{cases}
     0           & \hspace{3em}  0 \leq \ld \leq \tfrac{2}{3}  \\
     3 \ld - 2   & \hspace{3em}  \tfrac{2}{3} \leq \ld \leq 1  \\
     1           & \hspace{3em}  1 \leq \ld.
    \end{cases}
\]
Then $g_1 g_2 = g_2$ and $g_2 g_3 = g_3$.
Define $x = g_2 (a)$, $c = 1 - x$, and $y = g_3 (a)$.
Then $x y = y$ so $c y = 0$.
It follows
from Lemma \ref{L:CzBasic}(\ref{L:CzBasic:Orth}) that
$\langle c \rangle + \langle y \rangle \leq \langle 1 \rangle$.

It remains to prove that
$n \langle 1 \rangle \leq (n + 1) \langle c \rangle$.
Let $C$ be the unital subalgebra of $A$ generated by the elements
$a, \, u a u^{-1}, \, u^2 a u^{-2}, \, \ldots, \, u^n a u^{-n}$.
Then there is a compact metric space~$X$ and an isomorphism
$\ph \colon C  \to C(X)$.
For $k = 0, 1, \ldots, n$,
let $Z_k \subset X$ be the support of $\ph \big( u^k x u^{-k} \big)$.
The elements
\[
\ph ( g_1 (a)), \, \ph \big( u g_1 (a) u^{-1} \big), \,
  \ph \big( u^2 g_1 (a) u^{-2} \big), \,
   \ldots, \, \ph \big( u^n g_1 (a) u^{-n} \big)
 \in C(X)
\]
are pairwise orthogonal,
and
from $g_1 g_2 = g_2$ we get
\[
\ph \big( u^k g_1 (a) u^{-k} \big) \ph \big( u^k x u^{-k} \big)
   = \ph \big( u^k x u^{-k} \big)
\]
for $k = 0, 1, \ldots, n$.
Therefore the sets $Z_0, Z_1, \ldots, Z_n$ are disjoint.
Set $Z = \bigcup_{k = 0}^n Z_k$.

Let $(e_{j, k})_{0 \leq j, k \leq n}$ be the standard system
of matrix units for $M_{n + 1}$.
(The indexing starts at~$0$.)
For $k = 0, 1, \ldots, n$,
choose a unitary $w_k \in M_{n + 1}$ such that
$w_k e_{k, k} w_k^* = e_{0, 0}$.
The function from $Z$ to $M_{n + 1}$ which takes the
constant value $w_k$ on $Z_k$
is in the identity component of the unitary group of $C (Z, M_{n + 1})$,
so there is a unitary $w \in C (X, M_{n + 1})$ whose restriction
to each $Z_k$ is~$w_k$.
Identifying $C (X, M_{n + 1})$ with $M_{n + 1} \otimes C (X)$,
we find that there is $h \in C (X)$ such that
\[
w \left( \sum_{k = 0}^n e_{k, k} \otimes
             \ph \big( u^k x u^{-k} \big) \right) w^*
  = e_{0, 0} \otimes h.
\]
Since $c = 1 - x$, it follows that
\[
w \left( \sum_{k = 0}^n e_{k, k}
          \otimes \ph \big( u^k c u^{-k} \big) \right) w^*
  = 1 - e_{0, 0} \otimes h
  \geq \sum_{k = 1}^n e_{k, k} \otimes 1.
\]
Applying $\id_{M_{n + 1}} \otimes \ph^{-1}$,
and setting $v = \big( \id_{M_{n + 1}} \otimes \ph^{-1} \big) (w)$,
we get
\[
v \left( \sum_{k = 0}^n e_{k, k} \otimes u^k c u^{-k} \right) v^*
  \geq \sum_{k = 1}^n e_{k, k} \otimes 1.
\]
This implies that 
$n \langle 1 \rangle \leq (n + 1) \langle c \rangle$,
as desired.
\end{proof}

Our next goal is Lemma~\ref{L-2726Big},
which is a version for Cuntz comparison of
Lemma~1.15 of~\cite{PhT1}.

\begin{lem}\label{L-2726Norm1}
Let $A$ be a \ca,
let $x \in A_{+}$ satisfy $\| x \| = 1$,
and let $\ep > 0$.
Then there are positive elements $a, b \in {\overline{x A x}}$
with $\| a \| = \| b \| = 1$,
such that $a b = b$,
and such that whenever $c \in {\overline{b A b}}$
satisfies $\| c \| \leq 1$,
then
$\| x c - c \| < \ep$.
\end{lem}

\begin{proof}
Define \cfn{s} $f_0, f_1 \colon [0, 1] \to [0, 1]$
by
\[
f_0 (\ld)
 = \begin{cases}
   \left( 1 - \tfrac{\ep}{2} \right)^{-1} \ld
                      & \hspace{3em} 0 \leq \ld \leq 1 - \tfrac{\ep}{2}
        \\
   1                  & \hspace{3em} 1 - \tfrac{\ep}{2} \leq \ld
\end{cases}
\]
and
\[
f_1 (\ld)
 = \begin{cases}
   0          & \hspace{3em} 0 \leq \ld \leq 1 - \tfrac{\ep}{2}
        \\
   \tfrac{2}{\ep} \left[ \ld - \left(1 - \tfrac{\ep}{2} \right) \right]
              & \hspace{3em} 1 - \tfrac{\ep}{2} \leq \ld \leq 1.
\end{cases}
\]
Set $a = f_0 (x)$ and $b = f_1 (x)$.
Then $\| x - a \| < \ep$ and $a b = b$. 
Furthermore, $\| b \| = 1$ because $1 \in \spec (x)$.

Let $c \in {\overline{b A b}}$
satisfy $\| c \| \leq 1$.
Then $a c = c$.
Therefore $\| x c - c \| < \ep$.
\end{proof}

\begin{lem}\label{L-2726Big}
Let $A$ be a finite simple infinite dimensional unital \ca.
Let $x \in A_{+}$ satisfy $\| x \| = 1$.
Then for every $\ep > 0$
there is $y \in \big( {\overline{x A x}} \big)_{+} \setminus \{ 0 \}$
such that whenever $g \in A_{+}$ satisfies $0 \leq g \leq 1$
and $g \precsim_A y$,
then $\| (1 - g) x (1 - g) \| > 1 - \ep$.
\end{lem}

\begin{proof}
Choose positive elements $a, b \in {\overline{x^{1/2} A x^{1/2}}}$
as in Lemma~\ref{L-2726Norm1},
with $x^{1/2}$ in place of~$x$
and $\frac{\ep}{3}$ in place of~$\ep$.
Then $a, b \in {\overline{x A x}}$
since ${\overline{x^{1/2} A x^{1/2}}} = {\overline{x A x}}$.
Since $b \neq 0$,
Lemma~\ref{L:DivInSmp}
provides nonzero positive orthogonal elements
$z_1, z_2 \in {\overline{b A b}}$
(with $z_1 \sim_A z_2$).
We may require $\| z_1 \| = \| z_2 \| = 1$.

Define \cfn{s} $f_0, f_1, f_2 \colon [0, \infty) \to [0, 1]$
by
\[
f_0 (\ld)
 = \begin{cases}
   3 \ld   & \hspace{3em} 0 \leq \ld \leq \tfrac{1}{3}
       \\
   1       & \hspace{3em} \tfrac{1}{3} \leq \ld,
\end{cases}
\]
\[
f_1 (\ld)
 = \begin{cases}
   0                        & \hspace{3em} 0 \leq \ld \leq \tfrac{1}{3}
        \\
   3 \big( \ld - \tfrac{1}{3} \big)
                            & \hspace{3em}
                               \tfrac{1}{3} \leq \ld \leq \tfrac{2}{3}
       \\
   1                        & \hspace{3em} \tfrac{2}{3} \leq \ld,
\end{cases}
\]
and
\[
f_2 (\ld)
 = \begin{cases}
   0                                 & \hspace{3em}
                                        0 \leq \ld \leq \tfrac{2}{3}
        \\
   3 \big( \ld - \tfrac{2}{3} \big)  & \hspace{3em}
                                        \tfrac{2}{3} \leq \ld \leq 1
       \\
   1                                 & \hspace{3em} 1 \leq \ld.
\end{cases}
\]
For $j = 1, 2$ define
\[
b_j = f_0 (z_j),
\,\,\,\,\,\,
c_j = f_1 (z_j),
\andeqn
d_j = f_2 (z_j).
\]
Then
\[
0 \leq d_j \leq c_j \leq b_j \leq 1,
\,\,\,\,\,\,
a b_j = b_j,
\,\,\,\,\,\,
b_j c_j = c_j,
\,\,\,\,\,\,
c_j d_j = d_j,
\andeqn
d_j \neq 0.
\]
Also $b_1 b_2 = 0$.
Define $y = d_1$.
Then $y \in \big( {\overline{x A x}} \big)_{+}$.

Let $g \in A_{+}$ satisfy $0 \leq g \leq 1$ and $g \precsim_A y$.
We want to show that
\[
\| (1 - g) x (1 - g) \| > 1 - \ep,
\]
so suppose that $\| (1 - g) x (1 - g) \| \leq 1 - \ep$.
The choice of $a$ and~$b$,
and the relations $(b_1 + b_2)^{1/2} \in {\overline{b A b}}$
and $\big\| (b_1 + b_2)^{1/2} \big\| = 1$,
imply that
\[
\big\| x^{1/2} (b_1 + b_2)^{1/2} - (b_1 + b_2)^{1/2} \big\|
  < \frac{\ep}{3}.
\]
Using this relation and its adjoint at the second step, we get
\begin{align*}
\big\| (1 - g) (b_1 + b_2) (1 - g) \big\|
& = \big\| (b_1 + b_2)^{1/2} (1 - g)^2 (b_1 + b_2)^{1/2} \big\|
  \\
& < \big\| (b_1 + b_2)^{1/2} x^{1/2} (1 - g)^2
       x^{1/2} (b_1 + b_2)^{1/2} \big\|
    + \frac{2 \ep}{3}
  \\
& \leq \big\| x^{1/2} (1 - g)^2 x^{1/2} \big\| + \frac{2 \ep}{3}
  \\
& = \| (1 - g) x (1 - g) \| + \frac{2 \ep}{3}
  \leq 1 - \frac{\ep}{3}.
\end{align*}
In the following calculation, take $\bt = 1 - \frac{\ep}{3}$,
use $(b_1 + b_2) (c_1 + c_2) = c_1 + c_2$
and Lemma~\ref{L-2726IdemC} at the first step,
use Lemma~\ref{L:C2} at the second step,
use the estimate above at the third step,
and use $g \precsim_A y = d_1$ at the fourth step:
\begin{equation}\label{Eq_4819_Star}
c_1 + c_2
  \precsim_A [(b_1 + b_2) - \bt]_{+}
  \precsim_A \big[ (1 - g) (b_1 + b_2) (1 - g) - \bt \big]_{+}
       \oplus g
  = 0 \oplus g
  \precsim_A d_1.
\end{equation}
Set $r = (1 - c_1 - c_2) + d_1$.
Use Lemma \ref{L:CzBasic}(\ref{L:CzBasic:LCzCmpSum})
at the first step,
(\ref{Eq_4819_Star})~at the second step,
and Lemma \ref{L:CzBasic}(\ref{L:CzBasic:Orth})
and $d_1 (1 - c_1 - c_2) = 0$ at the third step,
to get
\[
1 \precsim_A (1 - c_1 - c_2) \oplus (c_1 + c_2)
  \precsim_A (1 - c_1 - c_2) \oplus d_1
  \sim_A (1 - c_1 - c_2) + d_1
  = r.
\]
Thus,
there is $v \in A$ such that $\| v r v^* - 1 \| < \tfrac{1}{2}$.
It follows that $v r^{1/2}$ has a right inverse.
But $v r^{1/2} d_2 = 0$, so $v r^{1/2}$ is not invertible.
We have contradicted finiteness of~$A$,
and thus proved the lemma.
\end{proof}

\section{The semigroup of purely positive elements}\label{Sec_PPE}

\indent
In this section,
$A$ is a stably finite simple unital \ca{}
not of type~I.
We consider the subsemigroup ${\operatorname{Cu}}_{+} (A) \cup \{ 0 \}$
of ${\operatorname{Cu}} (A)$
consisting of $\langle 0 \rangle$
and those elements of ${\operatorname{Cu}} (A)$
which are not the class of a \pj.
The main result of this section
is that ${\operatorname{Cu}}_{+} (A) \cup \{ 0 \}$
is a subsemigroup which
has the same functionals as ${\operatorname{Cu}} (A)$.

For a stably finite simple \ca~$A$,
the subsemigroup ${\operatorname{Cu}}_{+} (A) \cup \{ 0 \}$
is equal to the subsemigroup of purely noncompact elements
of ${\operatorname{Cu}}_{+} (A)$,
as defined before Proposition~6.4 of~\cite{ERS}.
See Proposition 6.4(iv) of~\cite{ERS}.
Unfortunately,
most of the results about it in~\cite{ERS}
have hypotheses that are too strong for our purposes.

We have found the following definition in the literature
only in connection with $W (A)$ rather than ${\operatorname{Cu}} (A)$.
(It appears before Corollary~2.24 of~\cite{APT}.
The subset is called $W (A)_{+}$ there.
The paper~\cite{ERS}
gives no notation for the subsemigroup of purely noncompact elements.)

\begin{dfn}\label{D_2Y24_Pure}
Let $A$ be a \ca.
Let ${\operatorname{Cu}}_{+} (A)$ denote
the set of elements $\et \in {\operatorname{Cu}} (A)$
which are not the classes of projections.
Similarly,
let $W_{+} (A)$ denote
the set of elements $\et \in W (A)$
which are not the classes of projections.
Further call an element $a \in (K \otimes A)_{+}$
{\emph{purely positive}}
if $\langle a \rangle \in {\operatorname{Cu}}_{+} (A)$.
\end{dfn}

The next result does for ${\operatorname{Cu}} (A)$
what Proposition~2.8 of~\cite{PT} does for $W (A)$.
Recall from Remark~\ref{R-2727MnI}
that $W (A) \subset {\operatorname{Cu}} (A)$.

\begin{lem}\label{L_2Y24_WhenP}
Let $A$ be a stably finite simple unital \ca.
Let $a \in (K \otimes A)_{+}$.
Then $a$ is purely positive \ifo{} $0$ is not an isolated point in
$\spec (a)$.
Moreover, if $a \in (K \otimes A)_{+}$
and $\langle a \rangle \not\in W (A)$,
then $a$ is purely positive.
\end{lem}

\begin{proof}
We always have $0 \in \spec (a)$.

If $0$ is isolated in $\spec (a)$,
then functional calculus
and Lemma \ref{L:CzBasic}(\ref{L:CzBasic:LCzFCalc})
show that $a$ is equivalent to a \pj.
Hence $a$ is not purely positive.

Now suppose that $0$ is not isolated in $\spec(a)$, but that 
nevertheless $a$ is not purely positive.
Thus $a$ is equivalent to a projection $p \in K \otimes A$.
Since $p \precsim_A a$,
Lemma~\ref{L:CzBasic}(\ref{L:CzBasic:LMinusEp})
provides $\dt > 0$ such that
$\big( p - \tfrac{1}{2} \big)_{+} \precsim_A (a - \dt)_{+}$.
Since $\big( p - \tfrac{1}{2} \big)_{+} = \tfrac{1}{2} p$,
it follows that $p \precsim_A (a - \dt)_{+}$.
Choose a \cfn{} $f \colon [0, \infty) \to [0, \infty)$
such that:
\begin{enumerate}
\item\label{E_2Y25_f_Supp}
$f (\ld) = 0$ for all $\ld \in [\dt, \infty)$.
\item\label{E_2Y25_f_leT}
$f (\ld) \leq \ld$ for all $\ld \in [0, \infty)$.
\item\label{E_2Y25_f_NZ}
There is $\ld \in \spec (a)$ such that $f (\ld) \neq 0$.
\end{enumerate}
Then $f (a) \neq 0$.
Using $p \precsim_A (a - \dt)_{+}$
at the first step,
and $f (a) (a - \dt)_{+} = 0$
and Lemma~\ref{L:CzBasic}(\ref{L:CzBasic:Orth})
at the second step,
\[
p \oplus f (a)
 \precsim_A (a - \dt)_{+} \oplus f (a)
 \sim_A (a - \dt)_{+} + f (a)
 \leq a
 \precsim_A p.
\]
So $p$ is an infinite projection by Lemma~3.1 of~\cite{KR},
a contradiction.

The second statement follows from the fact that every
\pj{} in $K \otimes A$
is \mvnt,
hence Cuntz equivalent, to a \pj{} in $M_{\infty} (A)$.
\end{proof}

The next result does for ${\operatorname{Cu}} (A)$
what one of the two cases of Corollary~2.9(i) of~\cite{PT}
does for $W (A)$.
(Corollary~2.9(i) of~\cite{PT} is also Corollary 2.24(i) of~\cite{APT},
but the proof given in~\cite{APT} appears to omit the
case of stably finite simple C*-algebras.)
Part of it follows from parts (i) and~(iv)
of Proposition~6.4 of~\cite{ERS}.

\begin{cor}\label{C_2Y24_PureAbs}
Let $A$ be a stably finite simple unital \ca.
Then ${\operatorname{Cu}}_{+} (A)$
is a subsemigroup of ${\operatorname{Cu}} (A)$
which is absorbing in the sense that
if $\et \in {\operatorname{Cu}}_{+} (A)$
and $\mu \in {\operatorname{Cu}} (A)$,
then $\et + \mu \in {\operatorname{Cu}}_{+} (A)$.
Moreover,
${\operatorname{Cu}}_{+} (A) \cup \{ 0 \}$
is a subsemigroup of ${\operatorname{Cu}} (A)$.
\end{cor}

\begin{proof}
The proof of the first statement
is the same as that of Corollary~2.9(i) of~\cite{PT}.
The second statement is immediate from the first.
\end{proof}

\begin{lem}\label{L_3905_UpToEp}
Let $A$ be a stably finite simple unital \ca{}
which is not of type~I.
Let $\om$ be a functional on ${\operatorname{Cu}} (A)$
(Definition~\ref{D_2Y25_CzFcnl}).
Then for every $\et \in {\operatorname{Cu}} (A) \setminus \{ 0 \}$
and every $\af \in (0, \, \om (\et) )$,
there is $\mu \in {\operatorname{Cu}}_{+} (A)$
such that $\mu \ll \et$
(Definition~\ref{D_4621_CC})
and $\om (\mu) > \af$.
\end{lem}

\begin{proof}
Choose $a \in (K \otimes A)_{+}$
such that $\et = \langle a \rangle$.
Using Lemma \ref{L_3905_CuWay}(\ref{L_3905_CuWay_SupEp})
and Definition \ref{D_2Y25_CzFcnl}(\ref{2Y24_CuFcnl_Sup}),
we can find $\dt > 0$ such that
$\om ( \langle (a - 2 \dt)_{+} \rangle ) > \af$.
We have $\langle (a - 2 \dt)_{+} \rangle \ll  \et$
by Lemma \ref{L_3905_CuWay}(\ref{L_3905_CuWay_EpWB}).
If $(a - 2 \dt)_{+}$ is purely positive,  the proof can be completed by
taking $\mu = \langle (a-2 \delta)_+\rangle $.

Otherwise, there is a \pj{} $p \in K \otimes A$
such that $\langle (a - 2 \dt)_{+} \rangle = \langle p \rangle$.
It follows from Theorem \ref{T_2Y25_DTau}(\ref{T_2Y25_DTau_Cu})
that there is a not necessarily normalized
$2$-quasitrace $\ta$ on~$A$
such that $\om = d_{\ta}$.
So $\om ( \langle p \rangle ) < \infty$.
Set $\ep = \om ( \langle p \rangle ) - \af$.
Then $\ep > 0$ by the choice of $\dt$.
Choose $n \in \N$ so large that $n \ep > \af$.
Apply Lemma~\ref{L:Smaller}
with this choice of~$n$,
with $\dt$ in place of~$\ep$,
and with $(a - \dt)_{+}$ in place of~$b$.
Since $( (a - \dt)_{+} - \dt)_{+} = (a - 2 \dt)_{+}$,
we find $c \in (K \otimes A)_{+}$
and $y \in (K \otimes A)_{+} \setminus \{ 0 \}$
such that
\[
n \langle (a - 2 \dt)_{+} \rangle \leq (n + 1) \langle c \rangle
\andeqn
\langle c \rangle + \langle y \rangle
 \leq \langle (a - \dt)_{+} \rangle.
\]
Applying $\om$ to the first inequality,
using the choice of~$n$,
and rearranging, we get $\om ( \langle c \rangle ) > \af$.
Use Lemma~\ref{L-2817Sp01}
to choose a positive element $y_0 \in {\overline{y (K \otimes A) y}}$
such that $\spec (y_0) = [0, 1]$.
Then $y_0 \precsim_A y$ by Lemma \ref{L:CzBasic}(\ref{L:CzBasic:Her})
and $\langle y_0 \rangle \in {\operatorname{Cu}}_{+} (A)$
by Lemma~\ref{L_2Y24_WhenP}.
Set $\mu = \langle c \rangle + \langle y_0 \rangle$,
which is in ${\operatorname{Cu}}_{+} (A)$
by Corollary~\ref{C_2Y24_PureAbs}.
Then,
using Lemma~\ref{L_3905_CuWay}(\ref{L_3905_CuWay_EpWB})
at the last step in the second calculation,
\[
\om ( \mu) \geq \om (\langle c \rangle) > \af
\andeqn
\mu \leq \langle c \rangle + \langle y \rangle
    \leq \langle (a - \dt)_{+} \rangle
    \ll \et.
\]
So $\mu \ll \et$ by Lemma \ref{L_3905_CuWay}(\ref{L_3905_CuWay_Tr}).
This completes the proof.
\end{proof}

The next lemma follows from
parts (i) and~(iv) of Proposition~6.4 of~\cite{ERS},
but we give the easy direct proof here.

\begin{lem}[\cite{ERS}]\label{L_3904_PureSup}
Let $A$ be a stably finite simple unital \ca{}
which is not of type~I.
If $\et_0 \leq \et_1 \leq \cdots$
in ${\operatorname{Cu}}_{+} (A) \cup \{ 0 \}$,
then $\sup \big( \{ \et_n \colon n \in \Nz \} \big)$,
evaluated in ${\operatorname{Cu}} (A)$,
is in ${\operatorname{Cu}}_{+} (A) \cup \{ 0 \}$.
\end{lem}

\begin{proof}
Let $\et = \sup \big( \{ \et_n \colon n \in \Nz \} \big)$,
evaluated in ${\operatorname{Cu}} (A)$.
Suppose $\et \not\in {\operatorname{Cu}}_{+} (A) \cup \{ 0 \}$.
Then,
by definition,
$\et$ is the class of a projection $p \in K \otimes A$.
Combining
Lemma~\ref{L_3905_CuWay}(\ref{L_3905_CuWay_PjWB})
and Definition~\ref{D_4621_CC},
we find $n \in \N$ such that $\et_n \geq \et$.
Therefore $\et_n = \et$.
So $\et_n = \langle p \rangle$,
contradicting $\et_n \in {\operatorname{Cu}}_{+} (A) \cup \{ 0 \}$.
\end{proof}

\begin{lem}\label{L_3905_InterpPj}
Let $A$ be a stably finite simple unital \ca{}
which is not of type~I.
Let $p \in K \otimes A$ be a \nzp,
let $n \in \N$,
and let $\xi \in {\operatorname{Cu}} (A) \setminus \{ 0 \}$.
Then there exist $\mu, \kp \in W_{+} (A)$
such that $\mu \leq \langle p \rangle \leq \mu + \kp$
and $n \kp \leq \xi$.
\end{lem}

\begin{proof}
\Wolog{} there is $n \in \N$ such that
$p \in M_n (A)$.
Using Lemma~\ref{L:DivInSmp}
and Lemma \ref{L:CzBasic}(\ref{L:CzBasic:Her}),
we can find $\xi_0 \in {\operatorname{Cu}} (A) \setminus \{ 0 \}$
such that $n \xi_0 \leq \xi$.
Use Lemma~\ref{L-2718CuSub} with $K \otimes A$ in place of~$A$
to find $b_0 \in (p (M_n \otimes A) p)_{+} \setminus \{ 0 \}$
such that $\langle b_0 \rangle \leq \xi_0$.
Lemma~\ref{L-2817Sp01} provides a positive element
$b \in {\overline{b_0 (M_n \otimes A) b_0}}$
such that $\spec (b) = [0, 1]$.
Then $\langle b \rangle \leq \xi_0$
by Lemma \ref{L:CzBasic}(\ref{L:CzBasic:Her}).
Set $\mu = \langle p - b \rangle \leq \langle p \rangle$
and set $\kp = \langle b \rangle$.
Then $\mu + \kp \geq \langle p \rangle$
by Lemma \ref{L:CzBasic}(\ref{L:CzBasic:LCzCmpSum}).
Clearly $\mu, \kp \in W (A)$,
and $\mu, \kp \in {\operatorname{Cu}}_{+} (A)$
by Lemma~\ref{L_2Y24_WhenP}.
So $\mu, \kp \in W_{+} (A)$.
Finally,
$n \kp \leq n \xi_0 \leq \xi$.
\end{proof}

\begin{lem}\label{L_3904_ExistInCuP}
Let $A$ be a stably finite simple unital \ca{}
which is not of type~I.
Let $\et \in {\operatorname{Cu}}_{+} (A)$.
Then there is a sequence $(\et_n)_{n \in \N}$
in ${\operatorname{Cu}}_{+} (A)$
such that
\[
\et_1 \ll \et_2 \ll \cdots
\andeqn
\et = \sup \big( \{ \et_n \colon n \in \Nz \} \big).
\]
\end{lem}

The point of the lemma is that $\et_n$ is purely positive for all~$n$.

\begin{proof}[Proof of Lemma~\ref{L_3904_ExistInCuP}]
Choose $a \in (K \otimes A)_{+}$
such that $\et = \langle a \rangle$.
Lemma~\ref{L_2Y24_WhenP} implies that $0$ is not isolated
in $\spec (a)$.
Therefore there is a sequence $\ep_1 > \ep_2 > \cdots$
such that $\limi{n} \ep_n = 0$
and $\spec (a) \cap (\ep_{n + 1}, \ep_{n}) \neq \varnothing$
for all $n \in \N$.
In particular,
there is a \ct{} function $f_n \colon [0, \infty) \to [0, \infty)$
with support in $(\ep_{n + 1}, \ep_{n})$
such that $f_n (a) \neq 0$.
Use Lemma~\ref{L-2817Sp01}
to choose a positive element
$y_n \in {\overline{f_n (a) (K \otimes A) f_n (a)}}$
such that $\spec (y_n) = [0, 1]$.
Then, using Lemma \ref{L:CzBasic}(\ref{L:CzBasic:Her}),
we have
\[
(a - \ep_1)_{+}
 \leq (a - \ep_1)_{+} + y_1
 \precsim_A (a - \ep_2)_{+}
 \leq (a - \ep_2)_{+} + y_2
 \precsim_A (a - \ep_3)_{+}
 \leq \cdots.
\]
It follows from Lemma \ref{L_3905_CuWay}(\ref{L_3905_CuWay_SupEp})
that
\[
\langle a \rangle
 = \sup \big(
    \big\{ \langle (a - \ep_n)_{+} + y_n \rangle
       \colon n \in \N \big\} \big).
\]
We have
$\langle (a - \ep_n)_{+} + y_n \rangle \in {\operatorname{Cu}}_{+} (A)$
for all $n \in \N$,
by combining Lemma~\ref{L_2Y24_WhenP}
and Corollary~\ref{C_2Y24_PureAbs}.
\end{proof}

\begin{lem}\label{L_3904_FclBj}
Let $A$ be a stably finite simple unital \ca{}
which is not of type~I.
Then restriction defines a bijection from the functionals
$\om$ on ${\operatorname{Cu}} (A)$
(as in Definition~\ref{D_2Y25_CzFcnl})
such that $\om ( \langle 1 \rangle ) = 1$
to the functionals $\om$ on ${\operatorname{Cu}}_{+} (A) \cup \{ 0 \}$
such that
\[
\sup \big( \big\{ \om (\et) \colon
 {\mbox{$\et \in {\operatorname{Cu}}_{+} (A) \cup \{ 0 \}$
     and $\et \leq \langle 1 \rangle$ in
     ${\operatorname{Cu}} (A)$}} \big\} \big)
 = 1.
\]
\end{lem}

\begin{proof}
Corollary~\ref{C_2Y24_PureAbs}
and Lemma~\ref{L_3904_PureSup}
show that ${\operatorname{Cu}}_{+} (A) \cup \{ 0 \}$
is the kind of object
on which functionals are defined.
It is clear from Definition~\ref{D_2Y25_CzFcnl}
and Lemma~\ref{L_3904_PureSup}
that if $\om$ is a functional on ${\operatorname{Cu}} (A)$,
then $\om |_{{\operatorname{Cu}}_{+} (A) \cup \{ 0 \} }$
is a functional on ${\operatorname{Cu}}_{+} (A) \cup \{ 0 \}$.
To show that the restriction map in the statement of the lemma
makes sense,
it remains
only to show that the normalization conditions agree;
this follows from Lemma~\ref{L_3905_UpToEp}.

For any $\et \in {\operatorname{Cu}} (A)$,
define
\[
H (\et)
 = \big\{ \ld \in {\operatorname{Cu}}_{+} (A) \cup \{ 0 \} \colon
 {\mbox{$\ld \leq \et$ in
     ${\operatorname{Cu}} (A)$}} \big\}.
\]

We prove surjectivity of restriction.
Let
$\om_0 \colon {\operatorname{Cu}}_{+} (A) \cup \{ 0 \} \to [0, \infty]$
be a functional on ${\operatorname{Cu}}_{+} (A) \cup \{ 0 \}$
such that $\sup_{\ld \in H ( \langle 1 \rangle )} \om_0 ( \ld ) = 1$.
Define a function $\om \colon {\operatorname{Cu}} (A) \to [0, \infty]$
by
\begin{equation}\label{Eq_4629_NS}
\om (\et)
 = \sup \big( \big\{ \om_0 (\ld) \colon
  \ld \in H ( \et ) \big\} \big)
\end{equation}
for $\et \in {\operatorname{Cu}} (A)$.

To see that
$\om |_{ {\operatorname{Cu}}_{+} (A) \cup \{ 0 \} } = \om_0$,
let $\et \in {\operatorname{Cu}}_{+} (A) \cup \{ 0 \}$.
Then $\et$ is the largest element of $H (\et)$,
so $\om (\et) = \om_0 (\et)$
because $\om_0$ is order preserving.

We have $\om (\langle 1 \rangle ) = 1$ by definition.

We need to prove
that $\om$ is a functional on ${\operatorname{Cu}} (A)$.
We split the proof into a number of claims,
the first two of which are preparatory.

We claim that
for every $\ld \in {\operatorname{Cu}}_{+} (A)$
such that $\om_0 (\ld) < \infty$
and every $\ep > 0$,
there is $\mu \in {\operatorname{Cu}}_{+} (A)$
such that $\mu \ll \ld$ and $\om_0 (\et) - \om_0 (\mu) < \ep$.
Lemma~\ref{L_3904_ExistInCuP} provides
a sequence $(\et_n)_{n \in \N}$
in ${\operatorname{Cu}}_{+} (A)$
such that
\[
\et_1 \ll \et_2 \ll \cdots
\andeqn
\lambda = \sup \big( \{ \et_n \colon n \in \Nz \} \big).
\]
{}From Definition~\ref{D_2Y25_CzFcnl}(\ref{2Y24_CuFcnl_Sup}),
we conclude that there is $n$ such that
$\om_0 ( \et_n ) > \om_0 (\ld) - \ep$.
The claim is proved.

We claim that for every
$\ld \in {\operatorname{Cu}} (A) \setminus \{ 0 \}$
such that $\om (\ld) < \infty$
and every $\ep > 0$,
there are $\mu_1, \mu_2, \rh \in {\operatorname{Cu}}_{+} (A)$
such that
\begin{equation}\label{Eq_4624_Star}
\mu_1 \ll \mu_2 \leq \ld \leq \rh
\andeqn
\om_0 (\rh) - \om_0 (\mu_1) < \ep.
\end{equation}
%
If $\ld \in {\operatorname{Cu}}_{+} (A)$,
we take $\mu_2 = \rh = \ld$
and use the previous claim to find~$\mu_1$.
Otherwise, use Lemma~\ref{L-2817Sp01} to choose $b \in A_{+}$
such that $\spec (b) = [0, 1]$.
Then $\langle 1 \oplus b \rangle \in {\operatorname{Cu}}_{+} (A)$
by Lemma~\ref{L_2Y24_WhenP}
and Corollary~\ref{C_2Y24_PureAbs},
and $\langle 1 - b \rangle \in {\operatorname{Cu}}_{+} (A)$
by Lemma~\ref{L_2Y24_WhenP}.
So, using
Lemma~\ref{L:CzBasic}(\ref{L:CzBasic:LCzCmpSum})
at the first step,
\[
\om_0 ( \langle 1 \oplus b \rangle )
  \leq \om_0 ( \langle 1 - b \rangle )
        + \om_0 ( \langle b \rangle )
        + \om_0 ( \langle b \rangle )
  \leq 3 \sup_{\ld \in H ( \langle 1 \rangle )} \om_0 ( \ld )
  < \infty.
\]
Choose $n \in \N$
such that $n \ep > 2 \om_0 ( \langle 1 \oplus b \rangle )$.
By definition,
$\ld$ is the class of a \pj{} in $K \otimes A$.
So we can use Lemma~\ref{L_3905_InterpPj}
to find $\mu_2, \kp \in {\operatorname{Cu}}_{+} (A)$
such that $\mu_2 \leq \ld \leq \mu_2 + \kp$
and $n \kp \leq \langle 1 \oplus b \rangle$.
Set $\rh = \mu_2 + \kp$.
Then
\[
\om_0 (\rh) - \om_0 (\mu_2)
 = \om_0 (\kp)
 \leq \frac{\om_0 (\langle 1 \oplus b) \rangle}{n}
 < \frac{\ep}{2}.
\]
Use the previous claim to find
$\mu_1 \in {\operatorname{Cu}}_{+} (A)$ such that $\mu_1 \ll \mu_2$
and such that $\om_0 (\mu_2) - \om_0 (\mu_1) < \frac{\ep}{2}$.
Then
$\mu_1 \ll \mu_2 \leq \ld \leq \rh$
and $\om_0 (\rh) - \om_0 (\mu_1) < \ep$.
The claim is proved.

We now claim that if $\mu, \et \in {\operatorname{Cu}} (A)$
satisfy $\mu \leq \et$,
then $\om (\mu) \leq \om (\et)$.
This claim is immediate from~(\ref{Eq_4629_NS})
and the observation that $H (\mu) \subset H (\et)$.

We next claim that $\om$ is additive.
So let $\mu, \et \in {\operatorname{Cu}} (A)$.
If $\om (\mu) = \infty$ or $\om (\et) = \infty$,
then $\om (\mu + \et) = \infty$
follows from $\mu, \et \leq \mu + \et$.
So assume $\om (\mu)$ and $\om (\et)$ are both finite.
Since $\ld \in H (\mu)$ and $\rh \in H (\et)$
imply $\ld + \rh \in H (\mu + \et)$,
it is obvious that $\om (\mu + \et) \geq \om (\mu) + \om (\et)$.
To prove the reverse inequality,
let $\ep > 0$.
By a simplified version of the claim giving~(\ref{Eq_4624_Star}) above
(also valid when $\ld$ there is zero),
there are
$\ld_1, \ld_2, \rh_1, \rh_2
 \in {\operatorname{Cu}}_{+} (A) \cup \{ 0 \}$
such that
\[
\ld_1 \leq \mu \leq \ld_2,
\,\,\,\,\,\,
\rh_1 \leq \et \leq \rh_2,
\,\,\,\,\,\,
\om_0 (\ld_2) - \om_0 (\ld_1) < \frac{\ep}{2},
\andeqn
\om_0 (\rh_2) - \om_0 (\rh_1) < \frac{\ep}{2}.
\]
Then,
using the fact that $\om$ is order preserving at the first step
and $\ld_2 + \rh_2 \in {\operatorname{Cu}}_{+} (A) \cup \{ 0 \}$
at the second step,
we get
\[
\om (\mu + \et)
 \leq \om (\ld_2 + \rh_2)
 = \om_0 (\ld_2 + \rh_2)
 < \om_0 (\ld_1) + \frac{\ep}{2} + \om_0 (\rh_1) + \frac{\ep}{2}
 \leq \om (\mu) + \om (\et) + \ep.
\]
Since $\ep > 0$ is arbitrary,
additivity follows.

It remains to prove that
if $\et_0 \leq \et_1 \leq \cdots$ in ${\operatorname{Cu}} (A)$,
and $\et = \sup \big( \{ \et_n \colon n \in \Nz \} \big)$,
then $\om (\et) = \sup \big( \{ \om (\et_n) \colon n \in \Nz \} \big)$.
Since
$\bigcup_{n = 0}^{\infty} H (\et_n) \subset H (\et)$,
we clearly get
$\om (\et) \geq \sup \big( \{ \om (\et_n) \colon n \in \Nz \} \big)$.

We prove the reverse inequality.
It is trivial when $\et = 0$.
Next assume that $\et \neq 0$ and $\om (\et) < \infty$.
Let $\ep > 0$.
The claim giving~(\ref{Eq_4624_Star}) above provides
$\mu_1, \mu_2, \rh \in {\operatorname{Cu}}_{+} (A)$
such that
\[
\mu_1 \ll \mu_2 \leq \eta \leq \rh
\andeqn
\om_0 (\rh) - \om_0 (\mu_1) < \ep.
\]
By Definition~\ref{D_4621_CC},
there is $n$ such that $\et_n \geq \mu_1$.
Then
$\om (\et_n)
 \geq \om (\mu_1)
 \geq \om (\rh) - \ep
 \geq \om (\et) - \ep$.
Since $\ep > 0$ is arbitrary,
we get
$\om (\et) \leq \sup \big( \{ \om (\et_n) \colon n \in \Nz \} \big)$.

Now suppose that $\om (\et) = \infty$.
We begin by showing that $\et \in {\operatorname{Cu}}_{+} (A)$.
Let $p \in K \otimes A$ be any \pj.
Then there are $l \in \Z_{>0}$ and a \pj{} $q \in M_l \otimes A$
such that $p \sim q$.
It follows that
\[
\om ( \langle p \rangle )
 = \om (\langle q \rangle )
 \leq l \om ( \langle 1 \rangle )
 = l
 < \infty.
\]
So $\langle p \rangle \neq \et$.
We thus have $\et \in {\operatorname{Cu}}_{+} (A)$.
So $\om_0 (\et) = \infty$.
Lemma~\ref{L_3904_ExistInCuP} provides
a sequence $(\rh_n)_{n \in \N}$
in ${\operatorname{Cu}}_{+} (A)$
such that $\rh_1 \ll \rh_2 \ll \cdots$
and $\eta = \sup \big( \{ \rh_n \colon n \in \Nz \} \big)$.
Let $M \in [0, \infty)$.
Since $\om_0 (\et) = \sup_{n \in \N} \om_0 (\rh_n)$,
there is $m \in \N$
such that $\om_0 (\rh_m) > M$.
By Definition~\ref{D_4621_CC},
there is $n$ such that $\et_n \geq \rh_m$.
Then $\om (\et_n) \geq \om (\rh_m) > M$.
Since $M$ is arbitrary, we get
$\sup \big( \{ \om (\et_n) \colon n \in \Nz \} \big) = \infty$.
This completes the proof that $\om$ is a functional,
hence of surjectivity of the restriction map.

To complete the proof of the lemma,
we show that the restriction map is injective.
Let $\om_1$ and $\om_2$ be functionals on ${\operatorname{Cu}} (A)$
such that
$\om_1 |_{{\operatorname{Cu}}_{+} (A) \cup \{ 0 \} }
 = \om_2 |_{{\operatorname{Cu}}_{+} (A) \cup \{ 0 \} }$.
Clearly $\om_1 (0) =\om_2 (0) = 0$.
Now let $\et \in {\operatorname{Cu}} (A) \setminus \{ 0 \}$.
For $j = 1, 2$,
use Lemma~\ref{L_3905_UpToEp} to get
\[
\om_j (\et)
 = \sup \big( \big\{ \om_j (\mu) \colon
        {\mbox{$\mu \in {\operatorname{Cu}}_{+} (A)$
          and $\mu \ll \et$}} \big\} \big).
\]
Therefore $\om_1 (\et) = \om_2 (\et)$.
\end{proof}

\section{The definition of a large subalgebra}\label{Sec_Dfn}

\indent
In this section,
we give the definition of a large subalgebra
and some convenient variants of the definition,
both formally stronger and formally weaker.
We also define an important special case:
large subalgebras of crossed product type.
The main point of this definition
is to provide a convenient way to show that
a subalgebra is large
(in fact, centrally large---see~\cite{ArPh}).

Some basic facts about large subalgebras
are in Section~\ref{Sec_FP},
the main theorems
are in Section~\ref{Sec_CuLg},
and a class of examples is in Section~\ref{Sec_VK}.

By convention,
if we say that $B$ is a unital subalgebra of a \ca~$A$,
we mean that $B$ contains the identity of~$A$.

\begin{dfn}\label{D_Large}
Let $A$ be an infinite dimensional simple unital \ca.
A unital subalgebra $B \subset A$ is said to be
{\emph{large}} in~$A$ if
for every $m \in \N$,
$a_1, a_2, \ldots, a_m \in A$,
$\ep > 0$, $x \in A_{+}$ with $\| x \| = 1$,
and $y \in B_{+} \setminus \{ 0 \}$,
there are $c_1, c_2, \ldots, c_m \in A$ and $g \in B$
such that:
\begin{enumerate}
\item\label{D_Large:Cut1}
$0 \leq g \leq 1$.
\item\label{D_Large:Cut2}
For $j = 1, 2, \ldots, m$ we have
$\| c_j - a_j \| < \ep$.
\item\label{D_Large:Cut3}
For $j = 1, 2, \ldots, m$ we have
$(1 - g) c_j \in B$.
\item\label{D_Large:Cut4}
$g \precsim_B y$ and $g \precsim_A x$.
\item\label{D_Large:Cut5}
$\| (1 - g) x (1 - g) \| > 1 - \ep$.
\end{enumerate}
\end{dfn}

We emphasize that the Cuntz subequivalence
involving $y$ in~(\ref{D_Large:Cut4})
is relative to~$B$,
not~$A$.

The condition~(\ref{D_Large:Cut5})
or $g \precsim_A x$ is needed to avoid trivialities.
Otherwise, even if we require that $B$ be simple
and that the restriction map
${\operatorname{QT}} (A) \to {\operatorname{QT}} (B)$ be surjective,
and that $A$ be stably finite,
we can take $A$ to be any UHF algebra
and take $B = \C$.
The choice $g = 1$ will always work.

In condition~(\ref{D_Large:Cut3}),
we can require $c_j (1 - g) \in B$ instead
for some or all of the elements by taking adjoints.
In our original definition,
we required both $(1 - g) c_j \in B$
and $c_j (1 - g) \in B$.
The version with only one side is
needed in~\cite{EN1},
and none of the original proofs required both sides.
We therefore use the one sided version.

The definition is meaningful even if $A$ is not simple,
and
a number of the results we prove do not actually require
simplicity of~$A$.
Without simplicity,
though, Definition~\ref{D_Large} is too restrictive.
For example,
if $A$ is simple and $B \subset A$ is a proper subalgebra
which is large in~$A$,
then $B \oplus B$ ought to be large in $A \oplus A$.
However,
the condition in the definition will not be satisfied
if, for example,
$x \in A \oplus 0$ or $y \in B \oplus 0$.

\begin{lem}\label{L_4627_Dense}
In Definition~\ref{D_Large},
it suffices to let $S \subset A$ be a subset whose linear span
is dense in~$A$,
and verify the hypotheses only when
$a_1, a_2, \ldots, a_m \in S$.
\end{lem}

\begin{proof}
The proof is immediate.
\end{proof}

\begin{rmk}\label{R_4628_OtherDense}
The same reduction applies to various conditions for
a subalgebra to be large given below,
such as Proposition~\ref{P_4624_Approx},
Proposition~\ref{P:FinLarge},
and other similar results.
It also applies to conditions for a subalgebra to be large of crossed
product type,
such as the definition
(Definition~\ref{D-2717CPType} below)
and Proposition~\ref{P-2729AltCPT}.
\end{rmk}

Unlike other approximation properties
(such as tracial rank),
it seems not to be possible to take $S$
in Lemma~\ref{L_4627_Dense}
to be a generating subset,
or even a selfadjoint generating subset.

The weaker form of the definition in the following proposition,
in which we merely require that $(1 - g) a_j$ be close to~$B$
instead of the existence of the elements~$c_j$,
was suggested by Zhuang Niu.
We prove that it is equivalent.

\begin{prp}\label{P_4624_Approx}
Let $A$ be an infinite dimensional simple unital \ca,
and let $B \subset A$ be a unital subalgebra.
Suppose that every finite set $F \subset A$,
$\ep > 0$, $x \in A_{+}$ with $\| x \| = 1$,
and $y \in B_{+} \setminus \{ 0 \}$,
there is $g \in B$
such that:
\begin{enumerate}
\item\label{P_4624_Approx_Cut1}
$0 \leq g \leq 1$.
\item\label{P_4624_Approx_Cut2}
$\dist ( (1 - g) a, \, B) < \ep$ for all $a \in F$.
\item\label{P_4624_Approx_Cut4}
$g \precsim_B y$ and $g \precsim_A x$.
\item\label{P_4624_Approx_Cut5}
$\| (1 - g) x (1 - g) \| > 1 - \ep$.
\end{enumerate}
Then $B$ is large in~$A$.
\end{prp}

\begin{proof}
Define \cfn{s}
$f_0, f_1, f_2 \colon [0, 1] \to [0, 1]$
by
\[
f_0 (\ld)
 = \begin{cases}
   3 \ep^{-1} \ld & \hspace{3em} 0 \leq \ld \leq \frac{\ep}{3}
       \\
   1              & \hspace{3em} \frac{\ep}{3} \leq \ld \leq 1,
\end{cases}
\]
\[
f_1 (\ld)
 = \begin{cases}
   0        & \hspace{3em} 0 \leq \ld \leq \frac{\ep}{3}
        \\
   \left( 1 - \frac{2 \ep}{3} \right)^{-1}
            \left( \ld - \frac{\ep}{3} \right)
            & \hspace{3em} \frac{\ep}{3} \leq \ld \leq 1 - \frac{\ep}{3}
       \\
   1        & \hspace{3em} 1 - \frac{\ep}{3} \leq \ld \leq 1,
\end{cases}
\]
and
\[
f_1 (\ld)
 = \begin{cases}
   0                        & \hspace{3em}
                    0 \leq \ld \leq 1 - \frac{\ep}{3}
       \\
   3 \ep^{-1} (\ld - 1) + 1 & \hspace{3em}
                    1 - \frac{\ep}{3} \leq \ld \leq 1.
\end{cases}
\]
Then $f_0 f_1 = f_1$,
$f_1 f_2 = f_2$,
\begin{equation}\label{Eq_4624_OneSt}
\sup_{\ld \in [0, 1]} | f_1 (\ld) - \ld | = \frac{\ep}{3},
\end{equation}
and
\begin{equation}\label{Eq_4624_TwoSt}
\sup_{\ld \in [0, 1]} | f_0 (\ld) \ld - \ld | = \frac{\ep}{3}.
\end{equation}

We verify Definition~\ref{D_Large}.
Let $m \in \Nz$,
let $a_1, a_2, \ldots, a_m \in A$,
let $\ep > 0$,
let $x \in A_{+}$ satisfy $\| x \| = 1$,
and let $y \in B_{+} \setminus \{ 0 \}$.
\Wolog{} $\ep < 1$ and $\| a_j \| \leq 1$ for $j = 1, 2, \ldots, m$.
Apply the hypothesis with $F = \{ a_1, a_2, \ldots, a_m \}$
and with $\frac{\ep}{3}$ in place of~$\ep$,
getting $g_0 \in B$.
Define $r_0 = 1 - g_0$.
Set $g = 1 - f_2 (r_0)$.

For $j = 1, 2, \ldots, m$,
we thus have
$\dist ( r_0 a_j, \, B) < \frac{\ep}{3}$.
Choose $b_j \in B$
such that $\| r_0 a_j - b_j \| < \frac{\ep}{3}$.
Define $c_j = (1 - f_1 (r_0) ) a_j + f_0 (r_0) b_j \in A$.

Definition~\ref{D_Large}(\ref{D_Large:Cut1}) ($0 \leq g \leq 1$)
is immediate.
Definition~\ref{D_Large}(\ref{D_Large:Cut4})
follows because $g_0 \precsim_B y$ and $g_0 \precsim_A x$,
and because the computation
$g = 1 - f_0 (1 - g_0) = f_2 (g_0)$,
combined with
Lemma~\ref{L:CzBasic}(\ref{L:CzBasic:LCzOneWay}),
shows that $g \precsim_B g_0$.

We estimate $\| c_j - a_j \|$.
Using~(\ref{Eq_4624_OneSt}) and $\| a_j \| \leq 1$,
we get
\[
\| f_1 (r_0) a_j - r_0 a_j \|
  \leq \| a_j \| \cdot \| f_1 (r_0) - r_0 \|
  \leq \frac{\ep}{3}.
\]
Using~(\ref{Eq_4624_TwoSt}) at the second step,
we get
\[
\| f_0 (r_0) b_j - r_0 a_j \|
  \leq \| f_0 (r_0) \| \cdot \| b_j - r_0 a_j \|
     + \| f_0 (r_0) r_0 - r_0 \| \cdot \| a_j \|
  < \frac{\ep}{3} + \frac{\ep}{3}
  = \frac{2 \ep}{3}.
\]
Combining these two estimates for the third step,
we get
\begin{align*}
\| c_j - a_j \|
& = \| f_0 (r_0) b_j - f_1 (r_0) a_j \|
\\
& \leq \| f_0 (r_0) b_j - r_0 a_j \| + \| f_1 (r_0) a_j - r_0 a_j \|
  < \frac{\ep}{3} + \frac{2 \ep}{3}
  = \ep.
\end{align*}
This is Definition~\ref{D_Large}(\ref{D_Large:Cut2}).

Since $f_2 (r_0) (1 - f_1 (r_0) ) = 0$
and $f_2 (r_0) f_0 (r_0) \in B$,
we get
\[
(1 - g) c_j
 = f_2 (r_0) \big[ (1 - f_1 (r_0) ) a_j + f_0 (r_0) b_j \big]
 = f_2 (r_0) f_0 (r_0) b_j
 \in B.
\]
This is Definition~\ref{D_Large}(\ref{D_Large:Cut3}).

Finally,
we verify Definition~\ref{D_Large}(\ref{D_Large:Cut5}).
We have $(1 - g)^2 = f_0 (r_0)^2 \geq r_0^2 = (1 - g_0)^2$,
so
\begin{align*}
\| (1 - g) x (1 - g) \|
& = \big\| x^{1/2} (1 - g)^2 x^{1/2} \big\|
\\
& \geq \big\| x^{1/2} (1 - g_0)^2 x^{1/2} \big\|
  = \| (1 - g_0) x (1 - g_0) \|
  > 1 - \frac{\ep}{3}
  = \ep.
\end{align*}
This completes the proof.
\end{proof}

When $A$ is finite,
we do not need condition~(\ref{D_Large:Cut5})
of Definition~\ref{D_Large}.

\begin{prp}\label{P:FinLarge}
Let $A$ be a finite infinite dimensional simple unital \ca,
and let $B \subset A$ be a unital subalgebra.
Suppose that for $m \in \N$,
$a_1, a_2, \ldots, a_m \in A$,
$\ep > 0$, $x \in A_{+} \setminus \{ 0 \}$,
and $y \in B_{+} \setminus \{ 0 \}$,
there are $c_1, c_2, \ldots, c_m \in A$ and $g \in B$
such that:
\begin{enumerate}
\item\label{D:FinLarge:Cut1}
$0 \leq g \leq 1$.
\item\label{D:FinLarge:Cut2}
For $j = 1, 2, \ldots, m$ we have
$\| c_j - a_j \| < \ep$.
\item\label{D:FinLarge:Cut3}
For $j = 1, 2, \ldots, m$ we have
$(1 - g) c_j \in B$.
\item\label{D:FinLarge:Cut4}
$g \precsim_B y$ and $g \precsim_A x$.
\end{enumerate}
Then $B$ is large in~$A$.
\end{prp}

\begin{rmk}\label{R_4628_OtherFin}
The proof of Proposition~\ref{P:FinLarge}
also shows that when $A$ is finite,
we can omit (\ref{P_4624_Approx_Cut5})
in Proposition~\ref{P_4624_Approx},
(\ref{D-2717CPType-Cut5a}) in Definition~\ref{D-2717CPType}
(see Proposition~\ref{P-2729AltCPT}),
and similar conditions in other results.
\end{rmk}

\begin{proof}[Proof of Proposition~\ref{P:FinLarge}]
Let $a_1, a_2, \ldots, a_m \in A$,
let $\ep > 0$,
let $x \in A_{+} \setminus \{ 0 \}$, 
and let $y \in B_{+} \setminus \{ 0 \}$.
\Wolog{} $\| x \| = 1$.

Apply Lemma~\ref{L-2726Big},
obtaining $x_0 \in \big( {\overline{x A x}} \big)_{+} \setminus \{ 0 \}$
such that whenever $g \in A_{+}$ satisfies $0 \leq g \leq 1$
and $g \precsim_A x_0$,
then $\| (1 - g) x (1 - g) \| > 1 - \ep$.
Apply the hypothesis
with $x_0$ in place of $x$
and everything else as given,
getting $c_1, c_2, \ldots, c_m \in A$ and $g \in B$.
We need only prove that $\| (1 - g) x (1 - g) \| > 1 - \ep$.
But this is immediate from the choice of~$x_0$.
\end{proof}

The following slight strengthening of the definition is
often convenient.

\begin{lem}\label{L:LargeDecNorm}
Let $A$ be an infinite dimensional simple unital \ca,
and let $B \subset A$ be a large subalgebra.
In Definition~\ref{D_Large},
the elements $c_1, c_2, \dots, c_m$ may be chosen so that
$\| c_j \| \leq \| a_j \|$ for $j = 1, 2, \ldots, m$.
\end{lem}

\begin{proof}
Let $m \in \Nz$,
let $a_1, a_2, \ldots, a_m \in A$,
let $\ep > 0$,
let $x \in A_{+}$ satisfy $\| x \| = 1$,
and let $y \in B_{+} \setminus \{ 0 \}$.
Without loss of generality we may assume that $\| a_j \| \leq 1$
for $j = 1, 2, \ldots, m$.
Apply Definition~\ref{D_Large}
with $\frac{\ep}{2}$ in place of~$\ep$
and all other elements as given.
Call the resulting elements $g$ and $b_1, b_2, \ldots, b_m$.
Then for $j = 1, 2, \ldots, m$
we have
\[
\| b_j \|
 \leq \| a_j \| + \frac{\ep}{2}
 \leq 1 + \frac{\ep}{2}.
\]
Define $c_j = \big( 1 + \frac{\ep}{2} \big)^{-1} b_j$.
Then $\| c_j \| \leq \| a_j \|$,
and
\[
\| c_j - b_j \|
  = \left[ 1 - \big( 1 + \tfrac{\ep}{2} \big)^{-1} \right]
      \| b_j \|
  \leq \left[ 1 - \big( 1 + \tfrac{\ep}{2} \big)^{-1} \right]
      \big( 1 + \tfrac{\ep}{2} \big)
  = \frac{\ep}{2}.
\]
So $\| c_j - a_j \| < \ep$.
The conditions (\ref{D_Large:Cut1}), (\ref{D_Large:Cut3}),
(\ref{D_Large:Cut4}), and~(\ref{D_Large:Cut5})
of Definition~\ref{D_Large} are immediate.
\end{proof}

If we cut down on both sides instead of
on one side,
and the elements $a_j$ are positive, 
then we may take the elements $c_j$ to be positive.

\begin{lem}\label{L:LargeStaysPositive}
Let $A$ be an infinite dimensional simple unital \ca,
and let $B \subset A$ be a large subalgebra.
Let $m, n \in \Nz$,
let $a_1, a_2, \ldots, a_m \in A$,
let $b_1, b_2, \ldots, b_n \in A_{+}$,
let $\ep > 0$,
let $x \in A_{+}$ satisfy $\| x \| = 1$,
and let $y \in B_{+} \setminus \{ 0 \}$.
Then there are
$c_1, c_2, \ldots, c_m \in A$,
$d_1, d_2, \ldots, d_n \in A_{+}$,
and $g \in B$
such that:
\begin{enumerate}
\item\label{L:LargeStaysPositive:Cut1}
$0 \leq g \leq 1$.
\item\label{L:LargeStaysPositive:Cut2a}
For $j = 1, 2, \ldots, m$ we have
$\| c_j - a_j \| < \ep$,
and for $j = 1, 2, \ldots, n$ we have
$\| d_j - b_j \| < \ep$.
\item\label{L:LargeStaysPositive:Cut2b}
For $j = 1, 2, \ldots, m$ we have
$\| c_j \| \leq \| a_j \|$,
and for $j = 1, 2, \ldots, n$ we have
$\| d_j \| \leq \| b_j \|$.
\item\label{L:LargeStaysPositive:Cut3}
For $j = 1, 2, \ldots, m$ we have
$(1 - g) c_j \in B$,
and for $j = 1, 2, \ldots, n$ we have
$(1 - g) d_j (1 - g) \in B$.
\item\label{L:LargeStaysPositive:Cut4}
$g \precsim_B y$ and $g \precsim_A x$.
\item\label{L:LargeStaysPositive:Cut5}
$\| (1 - g) x (1 - g) \| > 1 - \ep$.
\end{enumerate}
\end{lem}

\begin{proof}
By scaling,
and changing $\ep$ as appropriate,
we may assume $\| b_j \| \leq 1$
for $j = 1, 2, \ldots, n$.
Apply Lemma~\ref{L:LargeDecNorm}
with $x$ and $y$ as given,
$\frac{\ep}{2}$ in place of~$\ep$,
with $m + n$ in place of~$m$,
and with
\[
a_1, \, a_2, \, \ldots, \, a_m,
 \, b_1^{1/2}, \, b_2^{1/2}, \, \ldots, \, b_n^{1/2}
\]
in place of $a_1, a_2, \ldots, a_m$,
getting
\[
c_1, c_2, \ldots, c_m, r_1, r_2, \ldots, r_n \in A
\andeqn
g \in B.
\]
We immediately get
all parts of the conclusion of the lemma which don't involve
$b_j$ and~$d_j$
(with $\frac{\ep}{2}$ in place of~$\ep$).
For $j = 1, 2, \ldots, n$ set $d_j = r_j r_j^*$.
Then
\[
(1 - g) d_j (1 - g) = [(1 - g) r_j] [ (1 - g) r_j]^* \in B,
\]
\[
\| d_j \| \leq \| r_j \|^2 \leq \big\| b_j^{1/2} \big\|^2 = \| b_j \|,
\]
and
\[
\| d_j - b_j \|
  \leq \big\| r_j - b_j^{1/2} \big\| \cdot \| r_j^* \|
       + \big\| b_j^{1/2} \big\| \cdot \big\| r_j^* - b_j^{1/2} \big\|
  < \frac{\ep}{2} + \frac{\ep}{2}
  = \ep.
\]
This completes the proof.
\end{proof}

One of the motivating examples
for the concept of large subalgebras 
is crossed products.
Therefore, large subalgebras of crossed
product type are explored in~\cite{ArPh}.
We will exhibit 
examples of such subalgebras in Theorem~\ref{T-2819AYLg}.

\begin{dfn}\label{D-2717CPType}
Let $A$ be
an infinite dimensional simple separable unital \ca.
A subalgebra $B \subset A$ is said to be a
{\emph{large subalgebra of crossed product type}}
if there exist a subalgebra
$C \subset B$ and a subset $G$ of the unitary group of~$A$ such that:
\begin{enumerate}
\item\label{D-2717CPType-Sb}
\begin{enumerate}
\item\label{D-2717CPType-Sb1}
$C$ contains the identity of~$A$.
\item\label{D-2717CPType-Sb2}
$C$ and $G$ generate $A$ as a \ca.
\item\label{D-2717CPType-Sb3}
$u C u^* \subset C$ and $u^* C u \subset C$ for all $u \in G$.
\end{enumerate}
\item\label{D-2717CPType-Cut}
For every $m \in \N$,
$a_1, a_2, \ldots, a_m \in A$,
$\ep > 0$, $x \in A_{+}$ with $\| x \| = 1$,
and $y \in B_{+} \setminus \{ 0 \}$,
there are $c_1, c_2, \ldots, c_m \in A$ and $g \in C$
such that:
\begin{enumerate}
\item\label{D-2717CPType-Cut1}
$0 \leq g \leq 1$.
\item\label{D-2717CPType-Cut2}
For $j = 1, 2, \ldots, m$ we have
$\| c_j - a_j \| < \ep$.
\item\label{D-2717CPType-Cut3}
For $j = 1, 2, \ldots, m$ we have
$(1 - g) c_j \in B$.
\item\label{D-2717CPType-Cut4}
$g \precsim_B y$ and $g \precsim_A x$.
\item\label{D-2717CPType-Cut5a}
$\| (1 - g) x (1 - g) \| > 1 - \ep$.
\end{enumerate}
\end{enumerate}
\end{dfn}

The conditions in~(\ref{D-2717CPType-Cut})
are the same as the conditions in Definition~\ref{D_Large};
the difference is that
we require that $g \in C$,
not merely that $g \in B$.
In particular,
the following result is immediate.

\begin{prp}\label{P_4725_CPT}
Let $A$ be an infinite dimensional simple unital \ca.
Let $B \subset A$ be a large subalgebra of crossed product type.
Then $B$ is large in $A$
in the sense of Definition~\ref{D_Large}.
\end{prp}

The following is what we will actually check
when we prove (Theorem~\ref{T-2819AYLg})
that suitable orbit breaking subalgebras
are large of crossed product type.
There is an analogous statement for ordinary large subalgebras,
with essentially the same proof,
which we omit.

\begin{prp}\label{P-2729AltCPT}
Let $A$ be an infinite dimensional simple unital \ca,
and let $B \subset A$ be a unital subalgebra.
Let $C \subset B$ be a subalgebra,
let $G$ be a subset $G$ of the unitary group of~$A$,
and assume that the following conditions are satisfied:
\begin{enumerate}
\item\label{P-2729AltCPT-Fin}
$A$ is finite.
\item\label{P-2729AltCPT-Sb}
\begin{enumerate}
\item\label{P-2729AltCPT-Sb1}
$C$ contains the identity of~$A$.
\item\label{P-2729AltCPT-Sb2}
$C$ and $G$ generate $A$ as a \ca.
\item\label{P-2729AltCPT-Sb3}
$u C u^* \subset C$ and $u^* C u \subset C$ for all $u \in G$.
\item\label{P-2729AltCPT-Sb4}
For every $x \in A_{+} \setminus \{ 0 \}$
and $y \in B_{+} \setminus \{ 0 \}$,
there exists $z \in B_{+} \setminus \{ 0 \}$
such that $z \precsim_A x$ and $z \precsim_B y$.
\end{enumerate}
\item\label{P-2729AltCPT-AltCut}
For every $m \in \N$,
$a_1, a_2, \ldots, a_m \in A$,
$\ep > 0$,
and $y \in B_{+} \setminus \{ 0 \}$,
there are $c_1, c_2, \ldots, c_m \in A$ and $g \in C$
such that:
\begin{enumerate}
\item\label{P-2729AltCPT-AltCut1}
$0 \leq g \leq 1$.
\item\label{P-2729AltCPT-AltCut2}
For $j = 1, 2, \ldots, m$ we have
$\| c_j - a_j \| < \ep$.
\item\label{P-2729AltCPT-AltCut3}
For $j = 1, 2, \ldots, m$ we have
$(1 - g) c_j \in B$.
\item\label{P-2729AltCPT-AltCut4}
$g \precsim_B y$.
\end{enumerate}
\end{enumerate}
Then $B$ is a large subalgebra of~$A$ of crossed product type
in the sense of Definition~\ref{D-2717CPType}.
\end{prp}

\begin{proof}
Let $m \in \N$,
let $a_1, a_2, \ldots, a_m \in A$,
let $\ep > 0$,
let $x \in A_{+}$ satisfy $\| x \| = 1$,
and let $y \in B_{+} \setminus \{ 0 \}$.
Use Lemma~\ref{L-2726Big}
to choose $x_0 \in A_{+} \setminus \{ 0 \}$
such that whenever $g \in A_{+}$ satisfies $0 \leq g \leq 1$
and $g \precsim_A x_0$,
then $\| (1 - g) x (1 - g) \| > 1 - \ep$.
Use Lemma~\ref{L-2718CuSub} to choose $x_1 \in A_{+} \setminus \{ 0 \}$
such that $x_1 \precsim_A x_0$ and $x_1 \precsim_A x$.
By condition~(\ref{P-2729AltCPT-Sb4}) of the hypothesis,
there is $z \in B_{+} \setminus \{ 0 \}$
such that $z \precsim_A x_1$ and $z \precsim_B y$.
Apply condition~(\ref{P-2729AltCPT-AltCut}) of the hypothesis
with $m, a_1, a_2, \ldots, a_m, \ep$
as given and with $z$ in place of~$y$.
The resulting element $g$ satisfies
$g \precsim_B z \precsim_B y$
and $g \precsim_A z \precsim_A x$.
Also, $g \precsim_A z \precsim_A x_0$,
so $\| (1 - g) x (1 - g) \| > 1 - \ep$.
This shows that the definition of a large subalgebra
of crossed product type is satisfied.
\end{proof}

\section{First properties of large subalgebras}\label{Sec_FP}

\indent
In this section,
we give some basic properties of
large subalgebras.
We prove (Proposition~\ref{P_4624_TLarge})
that if the minimal tensor product of the containing
algebras is finite,
then the tensor product of large subalgebras is large.
This result is needed in~\cite{EN2}.
In particular, if $A$ is stably finite
and $B$ is large in~$A$,
then $M_n (B)$ is large in $M_n (A)$.
Without finiteness,
we had technical problems with
condition~(\ref{D_Large:Cut5})
of Definition~\ref{D_Large}.
(In the finite case,
we have seen that this condition is not needed.)
Therefore we define stably large subalgebras.

For the proof of Proposition~\ref{P_4624_TLarge},
we will need to know
that large subalgebras
are simple (Proposition~\ref{P_2627_NoSmp})
and infinite dimensional (Proposition~\ref{P-2729LgInfD}),
and we will also need several lemmas.

\begin{dfn}\label{D_4619_StLg}
Let $A$ be an infinite dimensional simple unital \ca.
A unital subalgebra $B \subset A$ is said to be
{\emph{stably large}} in~$A$ if
$M_n (B)$ is large in $M_n (A)$ for all $n \in \Nz$.
\end{dfn}

One can also define stably large subalgebras of crossed product type.
This refinement seems not to be needed.

As indicated above,
at the end of this section
we prove that a large subalgebra
of a stably finite algebra is stably large.
We do not know whether stable finiteness is needed.

\begin{prp}\label{P_2627_NoSmp}
Let $A$ be an infinite dimensional simple unital \ca,
and let $B \subset A$ be a large subalgebra.
Then $B$ is simple.
\end{prp}

\begin{proof}
Let $b \in B_{+} \setminus \{ 0 \}$.
We show that there are $n \in \N$
and $r_1, r_2, \ldots, r_n \in B$
such that $\sum_{k = 1}^n r_k b r_k^*$ is invertible.

Since $A$ is simple,
Corollary~\ref{C_4619_PosSmp}
provides $m \in \N$
and $x_1, x_2, \ldots, x_m \in A$
such that $\sum_{k = 1}^m x_k b x_k^* = 1$.
Set
\[
M = \max \big( 1, \, \| x_1 \|, \, \| x_2 \|, \, \ldots, \, \| x_m \|,
       \, \| b \| \big)
\andeqn
\dt = \min \left( 1, \, \frac{1}{3 m M (2 M + 1)} \right).
\]
By definition,
there are $y_1, y_2, \ldots, y_m \in A$
and $g \in B_{+}$
such that
$0 \leq g \leq 1$,
such that $\| y_j - x_j \| < \dt$
and $(1 - g) y_j \in B$
for $j = 1, 2, \ldots, m$,
and such that $g \precsim_B b$.

Set $z = \sum_{k = 1}^m y_j b y_j^*$.
We claim that $\| z - 1 \| < \frac{1}{3}$.
For $j = 1, 2, \ldots, m$,
we have
$\| y_j \| < \| x_j \| + \dt \leq M + 1$,
so
\begin{align*}
\| y_j b y_j^* - x_j b x_j^* \|
& \leq \| y_j - x_j \| \cdot \| b \| \cdot \| y_j^* \|
         + \| x_j \| \cdot \| b \| \cdot \| y_j^* - x_j^* \|
     \\
& < \dt M (M + 1) + M^2 \dt
  = M (2 M + 1) \dt.
\end{align*}
Therefore
\[
\| z - 1 \|
  = \left\| \sum_{k = 1}^m y_j b y_j^*
        - \sum_{k = 1}^m x_j b x_j^* \right\|
  \leq \sum_{k = 1}^m \| y_j b y_j^* - x_j b x_j^* \|
  < m M (2 M + 1) \dt
  \leq \frac{1}{3},
\]
as claimed.

Set $h = 2 g - g^2$.
Lemma \ref{L:CzBasic}(\ref{L:CzBasic:LCzFCalc}),
applied to the function $\ld \mapsto 2 \ld - \ld^2$,
implies that $h \sim_B g$.
Therefore $h \precsim_B b$.
So there is $v \in B$ such that $\| v b v^* - h \| < \frac{1}{3}$.
Now
take $n = m + 1$,
take $r_j = (1 - g) y_j$ for $j = 1, 2, \ldots, m$,
and take $r_{m + 1} = v$.
Then $r_1, r_2, \ldots, r_n \in B$.
We have
\[
\big\| (1 - g) z (1 - g) - (1 - g)^2 \big\|
  \leq \| 1 - g \| \cdot \| z - 1 \| \cdot \| 1 - g \|
  < \frac{1}{3}.
\]
So, using $(1 - g)^2 + h = 1$ at the second step,
we get
\begin{align*}
\left\| 1 - \sum_{k = 1}^n r_k b r_k^* \right\|
& = \big\| 1 - (1 - g) z (1 - g) - v b v^* \big\|
\\
& \leq \big\| (1 - g)^2 - (1 - g) z (1 - g) \big\|
         + \| h - v b v^* \|
  < \frac{2}{3}.
\end{align*}
Therefore $\sum_{k = 1}^n r_k b r_k^*$ is invertible,
as desired.
\end{proof}

\begin{lem}\label{L-2720L1}
Let $A$ be an infinite dimensional simple unital \ca,
and let $B \subset A$ be a large subalgebra.
Let $r \in B_{+} \setminus \{ 0 \}$,
let $a \in {\overline{r A r}}$ be positive and satisfy $\| a \| = 1$,
and let $\ep > 0$.
Then there is a positive element $b \in {\overline{r B r}}$ such that:
\begin{enumerate}
\item\label{L-2720L1-N1}
$\| b \| = 1$.
\item\label{L-2720L1-Sub}
$b \precsim_A a$.
\item\label{L-2720L1-Cl}
$\| a b - b \| < \ep$.
\end{enumerate}
\end{lem}

\begin{proof}
\Wolog{} $\| r \| = 1$.
Set
$\dt = \min \left( 1, \, \frac{\ep}{25} \right) > 0$.
Define \cfn{s} $f_0, f_1 \colon [0, \infty) \to [0, 1]$
by
\[
f_0 (\ld)
 = \begin{cases}
   (1 - \dt)^{-1} \ld & \hspace{3em} 0 \leq \ld \leq 1 - \dt
        \\
   1                  & \hspace{3em} 1 - \dt \leq \ld
\end{cases}
\]
and
\[
f_1 (\ld)
 = \begin{cases}
   0                           & \hspace{3em} 0 \leq \ld \leq 1 - \dt
        \\
   \dt^{-1} [ \ld - (1 - \dt)] & \hspace{3em} 1 - \dt \leq \ld \leq 1
       \\
   1                           & \hspace{3em} 1 \leq \ld.
\end{cases}
\]
Define $a_0 = f_0 (a)$ and $a_1 = f_1 (a)$.
Then
\begin{equation}\label{Eq:2723a0a}
a_0 a_1^{1/2} = a_1^{1/2},
\,\,\,\,\,\,
\| a_0 - a \| < \dt,
\,\,\,\,\,\,
\| a_0 \| \leq 1,
\andeqn
\big\| a_1^{1/2} \big\| \leq 1.
\end{equation}

Since $B$ is large in~$A$,
by Lemma~\ref{L:LargeDecNorm} there is $g \in B$ and $c \in A$ such that
\[
0 \leq g \leq 1,
\,\,\,\,\,\,
\big\| c - a_1^{1/2} \big\| < \dt,
\,\,\,\,\,\,
\| c \| \leq 1,
\]
\[
(1 - g) c \in B,
\andeqn
\| (1 - g) a_1 (1 - g) \| > 1 - \dt.
\]

Since $(r^{1/n})_{n \in \N}$
is an approximate identity for ${\overline{r A r}}$,
there is $n$ such that the element $e = r^{1/n}$
satisfies $\big\| a_1^{1/2} e - a_1^{1/2} \big\| < \dt$.
Also $\| e \| \leq 1$.
Moreover,
\begin{equation}\label{Eq:2728SqRtEst}
\big\| a_1^{1/2} - c e \big\|
  \leq \big\| a_1^{1/2} - a_1^{1/2} e \big\|
      + \big\| a_1^{1/2} - c \big\| \cdot \| e \|
  < \dt + \dt
  = 2 \dt.
\end{equation}

Set $b_0 = e c^* (1 - g)^2 c e$.
Because $(1 - g) c \in B$,
it follows that $b_0$ is a positive element of ${\overline{r B r}}$.

Using the first and third parts of~(\ref{Eq:2723a0a})
at the first step,
we get
\[
\| a_0 e c^* - e c^* \|
  \leq 2 \big\| e c^* - a_1^{1/2} \big\|
  < 4 \dt.
\]
Using  $\| c \| \leq 1$
and the second part of~(\ref{Eq:2723a0a}) at the second step,
it then follows that
\[
\| a e c^* - e c^* \|
  \leq \| a - a_0 \| \cdot \| e \| \cdot \| c^* \|
      + \| a_0 e c^* - e c^* \|
  < \dt + 4 \dt
  = 5 \dt.
\]
Therefore
\begin{equation}\label{Eq_4814_ND}
\| a b_0 - b_0 \|
 \leq \| a e c^* - e c^* \|
        \cdot \| (1 - g)^2 \| \cdot \| c \| \cdot \| e \|
 < 5 \dt.
\end{equation}
So
\begin{equation}\label{Eq:2723-12dt}
\| a b_0 a - b_0 \| < 10 \dt.
\end{equation}

We have
\[
\| b_0 \|
 = \| e c^* (1 - g)^2 c e \|
 = \| (1 - g) c e^2 c^* (1 - g) \|
 \geq \| (1 - g) a_1 (1 - g) \| - \| c e^2 c^* - a_1 \|
\]
and, using~(\ref{Eq:2728SqRtEst}),
\[
\| c e^2 c^* - a_1 \|
  \leq \big\| c e - a_1^{1/2} \big\| \cdot \| e \| \cdot \| c^* \|
      + \big\| a_1^{1/2} \big\| \cdot \big\| e c^* - a_1^{1/2} \big\|
  < 2 \dt + 2 \dt
  = 4 \dt.
\]
So, using the choice of~$g$,
\begin{equation}\label{Eq:2723bLB}
\| b_0 \|
 > 1 - \dt - 4 \dt
 = 1 - 5 \dt.
\end{equation}

Define a \cfn{} $f \colon [0, \infty) \to [0, 1]$
by
\[
f (\ld)
 = \begin{cases}
   0                                 & \hspace{3em}
                                        0 \leq \ld \leq 10 \dt
        \\
   (1 - 20 \dt)^{-1} ( \ld - 10 \dt) & \hspace{3em}
                                        10 \dt \leq \ld \leq 1 - 10 \dt
       \\
   1                                 & \hspace{3em} 1 - 10 \dt \leq \ld.
\end{cases}
\]
Set $b = f (b_0)$.
We have $\| b \| = 1$ by~(\ref{Eq:2723bLB}),
which is part~(\ref{L-2720L1-N1}) of the conclusion.
Also,
using $\| b_0 \| \leq \| c \|^2 \leq1$,
we get
$\| b - b_0 \| \leq 10 \dt$.
Therefore, using~(\ref{Eq_4814_ND}) at the second step,
\[
\| a b - b \|
  \leq \| a b_0 - b_0 \| + 2 \| b - b_0 \|
  < 5 \dt + 2 (10 \dt)
  = \ep.
\]
This is part~(\ref{L-2720L1-Cl}) of the conclusion.
Finally,
using Lemma~\ref{L:CzBasic}(\ref{L:CzBasic:LCzFCalc})
at the first step,
using~(\ref{Eq:2723-12dt})
and Lemma~\ref{L:CzBasic}(\ref{L:CzBasic:LCzWithinEp})
at the second step,
and using Lemma \ref{L:CzBasic}(\ref{L:CzBasic:Her})
at the third step,
we have
\[
b \sim_A (b_0 - 10 \dt)_{+}
  \precsim_A a b_0 a
  \precsim_A a.
\]
This is part~(\ref{L-2720L1-Sub}) of the conclusion.
\end{proof}

We record for convenient reference the following
semiprojectivity result.

\begin{prp}\label{P_4619_ConeSj}
Let $n \in \N$.
Then for every $\dt > 0$ there is $\rh > 0$
such that whenever $D$ is a \ca{}
and $b_1, b_2, \ldots, b_n \in D$
satisfy $0 \leq b_j \leq 1$ and $\| b_j b_k \| < \rh$
for distinct $j, k \in \{ 1, 2, \ldots, n \}$,
then there exist $y_1, y_2, \ldots, y_n \in D$
such that
%
\[
0 \leq y_j \leq 1,
\,\,\,\,\,\,
y_j y_k = 0,
\andeqn
\| y_j - b_j \| < \dt
\]
%
for $j, k = 1, 2, \ldots, n$ with $j \neq k$.
\end{prp}

\begin{proof}
Theorem 10.1.11 of~\cite{Lr}
and the proof of Proposition 10.1.10 of~\cite{Lr}
show that $\bigoplus_{k = 1}^n C ((0, 1])$ is projective.
Therefore this algebra is semiprojective.
Apply Theorem 14.1.4 of~\cite{Lr}.
\end{proof}

\begin{prp}\label{P-2729LgInfD}
Let $A$ be an infinite dimensional simple unital \ca,
and let $B \subset A$ be a large subalgebra.
Then $B$ is infinite dimensional.
\end{prp}

\begin{proof}
Let $n \in \N$;
we prove that $\dim (B) \geq n$.

Since $A$ is simple and in\fd,
Lemma~\ref{L:DivInSmp} provides
$a_1, a_2, \ldots, a_n \in A_{+} \setminus \{ 0 \}$
such that $a_j a_k = 0$ for distinct $j, k \in \{ 1, 2, \ldots, n \}$.
Choose $\rh > 0$ as in Proposition~\ref{P_4619_ConeSj}
with $\dt = \frac{1}{2}$.
Use Lemma~\ref{L-2720L1}
to choose $b_1, b_2, \ldots, b_l \in B_{+}$
such that for $j = 1, 2, \ldots, n$,
we have
\[
\| b_j \| = 1
\andeqn
\| a_j b_j - b_j \| < \frac{\rh}{2}.
\]
Then for distinct $j, k \in \{ 1, 2, \ldots, n \}$,
we have
\[
\| b_j b_k \|
 = \| b_j b_k - b_j a_j a_k b_k \|
 \leq \| b_j - b_j a_j \| \cdot \| b_k \|
          + \| b_j \| \cdot \| a_j \| \cdot \| b_k - a_k b_k \|
 < \frac{\rh}{2} + \frac{\rh}{2}
 = \rh.
\]
By the choice of~$\rh$ using Proposition~\ref{P_4619_ConeSj},
there are orthogonal positive
elements $y_1, y_2, \ldots, y_n \in B$
of norm at most~$1$
such that $\| y_j - b_j \| < \frac{1}{2}$
for $j = 1, 2, \ldots, n$.
Then $\| y_j \| > \| b_j \| - \| y_j - b_j \| > \frac{1}{2}$,
so $y_j \neq 0$.
Thus $y_1, y_2, \ldots, y_n$ are nonzero orthogonal elements,
hence linearly independent.
\end{proof}

\begin{prp}\label{P_4624_TLarge}
Let $A_1$ and $A_2$ be infinite dimensional simple unital \ca{s},
and let $B_1 \subset A_1$ and $B_2 \subset A_2$ be
large subalgebras.
Assume that $A_1 \otimes_{\min} A_2$
is finite.
Then $B_1 \otimes_{\min} B_2$
is a large subalgebra of $A_1 \otimes_{\min} A_2$.
\end{prp}

To keep the the notation simple,
we isolate the following part as a lemma.

\begin{lem}\label{L_4624_TComp}
Let $A$ and $B$ be infinite dimensional simple unital \ca{s},
and let $x \in (A \otimes_{\min} B)_{+} \setminus \{ 0 \}$.
Then there exist
$a \in A_{+} \setminus \{ 0 \}$ and $b \in B_{+} \setminus \{ 0 \}$
such that,
whenever $g \in A_{+}$ and $h \in B_{+}$
satisfy $g \precsim_A a$ and $h \precsim_B b$,
then
\[
g \otimes 1 + 1 \otimes h \precsim_{A \otimes_{\min} B} x.
\]
\end{lem}

\begin{proof}
Since $A \otimes_{\min} B$ is in\fd, simple, and unital,
Lemma~\ref{L:DivInSmp}
provides orthogonal nonzero positive elements
$x_1, x_2 \in {\overline{x (A \otimes_{\min} B) x}}$.
Use Kirchberg's Slice Lemma (Lemma 4.1.9 of~\cite{Rrd})
to find $y_1, y_2 \in A_{+} \setminus \{ 0 \}$
and $z_1, z_2 \in B_{+} \setminus \{ 0 \}$
such that $y_1 \otimes z_1 \precsim_{A \otimes_{\min} B} x_1$
and $y_2 \otimes z_2 \precsim_{A \otimes_{\min} B} x_2$.
By Corollary~\ref{C_4619_PosSmp},
there are $m, n \in \N$,
$c_1, c_2, \ldots, c_m \in A$,
and $d_1, d_2, \ldots, d_n \in B$
such that $\sum_{k = 1}^m c_k^* y_2 c_k = 1$
and $\sum_{k = 1}^n d_k^* z_1 d_k = 1$.
By Lemma~\ref{L:CzBasic}(\ref{L:CzBasic:LCzCmpSum}),
we get $\langle 1_A \rangle \leq m \langle y_2 \rangle$
and $\langle 1_B \rangle \leq n \langle z_1 \rangle$.
With the help of
Lemma~\ref{L:DivInSmp},
find 
$a \in A_{+} \setminus \{ 0 \}$ and $b \in B_{+} \setminus \{ 0 \}$
such that $n \langle a \rangle \leq \langle y_1 \rangle$
and $m \langle b \rangle \leq \langle z_2 \rangle$.

Now assume that
$g \precsim_A a$ and $h \precsim_B b$.
Repeated application of Lemma~\ref{L_4624_CuTr} gives
\[
\langle g \otimes 1_B \rangle
 \leq \langle a \otimes 1_B \rangle
 \leq n \langle a \otimes z_1 \rangle
 \leq \langle y_1 \otimes z_1 \rangle
 \leq \langle x_1 \rangle
\]
and similarly
$\langle 1_A \otimes h \rangle \leq \langle x_2 \rangle$.
Therefore,
using Lemma~\ref{L:CzBasic}(\ref{L:CzBasic:LCzCmpSum})
at the first step,
Lemma~\ref{L:CzBasic}(\ref{L:CzBasic:Orth})
at the second step,
and Lemma~\ref{L:CzBasic}(\ref{L:CzBasic:Her})
at the third step,
we get
\[
g \otimes 1_B + 1_A \otimes h
 \precsim_{A \otimes_{\min} B} (g \otimes 1_B) \oplus (1_A \otimes h)
 \precsim_{A \otimes_{\min} B} x_1 + x_2
 \precsim_{A \otimes_{\min} B} x.
\]
This completes the proof.
\end{proof}

\begin{proof}[Proof of Proposition~\ref{P_4624_TLarge}]
The span of the elementary tensors is dense.
So, by Proposition~\ref{P:FinLarge}
and Remark~\ref{R_4628_OtherDense},
it suffices to do the following.
Let $m \in \Nz$,
let $a_{1, 1}, a_{1, 2}, \ldots, a_{1, m} \in A_1$
and $a_{2, 1}, a_{2, 2}, \ldots, a_{2, m} \in A_2$
all have norm at most~$1$,
let $\ep > 0$,
let $x \in (A_1 \otimes_{\min} A_2)_{+} \setminus \{ 0 \}$,
and let $y \in (B_1 \otimes_{\min} B_2)_{+} \setminus \{ 0 \}$.
Then we find
$c_1, c_2, \ldots, c_m \in A_1 \otimes_{\min} A_2$
and $g \in B_1 \otimes_{\min} B_2$
such that:
\begin{enumerate}
\item\label{P_4624_TLarge_Cut1}
$0 \leq g \leq 1$.
\item\label{P_4624_TLarge_Cut2}
For $j = 1, 2, \ldots, m$ we have
$\| c_j - a_{1, j} \otimes a_{2, j} \| < \ep$.
\item\label{P_4624_TLarge_Cut3}
For $j = 1, 2, \ldots, m$ we have
$(1 - g) c_j \in B$.
\item\label{P_4624_TLarge_Cut4}
$g \precsim_B y$ and $g \precsim_A x$.
\setcounter{TmpEnumi}{\value{enumi}}
\end{enumerate}

It follows from Proposition~\ref{P_2627_NoSmp}
that $B_1$ and $B_2$ are simple
and from Proposition~\ref{P-2729LgInfD}
that $B_1$ and $B_2$ are infinite dimensional.
Applying Lemma~\ref{L_4624_TComp},
we find $x_1 \in (A_1)_{+} \setminus \{ 0 \}$,
$x_2 \in (A_2)_{+} \setminus \{ 0 \}$,
$y_1 \in (B_1)_{+} \setminus \{ 0 \}$,
and $y_2 \in (B_2)_{+} \setminus \{ 0 \}$
such that
whenever $g_1 \in (A_1)_{+}$ and $g_2 \in (A_2)_{+}$
satisfy $g_1 \precsim_{A_1} x_1$ and $g_2 \precsim_{A_2} x_2$,
then
\[
g_1 \otimes 1 + 1 \otimes g_2 \precsim_{A_1 \otimes_{\min} A_2} x
\]
and
whenever $g_1 \in (B_1)_{+}$ and $g_2 \in (B_2)_{+}$
satisfy $g_1 \precsim_{B_1} y_1$ and $g_2 \precsim_{B_2} y_2$,
then
\[
g_1 \otimes 1 + 1 \otimes g_2 \precsim_{B_1 \otimes_{\min} B_2} y.
\]
For $l = 1, 2$,
apply Lemma~\ref{L:LargeDecNorm} to $A_l$, $B_l$,
$a_{l, 1}, a_{l, 2}, \ldots, a_{l, m}$,
$\frac{\ep}{2}$,
$x_l$, and $y_l$,
getting $c_{l, 1}, c_{l, 2}, \ldots, c_{l, m} \in A_l$
and $g_l \in  B_l$
such that:
\begin{enumerate}
\setcounter{enumi}{\value{TmpEnumi}}
\item\label{Pf_UseCut1}
$0 \leq g_l \leq 1$.
\item\label{Pf_UseCut2}
For $j = 1, 2, \ldots, m$ we have
$\| c_{l, j} - a_{l, j} \| < \frac{\ep}{2}$,
$(1 - g) c_{l, j} \in B_l$,
and $\| c_{l, j} \| \leq 1$.
\item\label{Pf_UseCut4}
$g_l \precsim_{B_l} y_l$ and $g_l \precsim_{A_l} x_l$.
\end{enumerate}
Define $c_j = c_{1, j} \otimes c_{2, j}$
for $j = 1, 2, \ldots, m$
and define $g = 1 - (1 - g_1) \otimes (1 - g_2)$.
Conditions~(\ref{P_4624_TLarge_Cut1})
and~(\ref{P_4624_TLarge_Cut3}) are clear.
Recalling that $\| c_{l, j} \| \leq 1$ and $\| c_{l, j} \| \leq 1$
for $l = 1, 2$ and $j = 1, 2, \ldots, m$,
we get condition~(\ref{D:FinLarge:Cut2}) from the norm
estimate
\[
\| c_{1, j} \otimes c_{2, j} - a_{1, j} \otimes a_{2, j} \|
  \leq \| c_{1, j} - a_{1, j} \| \cdot \| c_{2, j} \|
        + \| a_{1, j} \| \cdot \| c_{2, j} -  a_{2, j} \|
  < \frac{\ep}{2} + \frac{\ep}{2}
  = \ep.
\]
Finally,
we observe that
\[
g = g_1 \otimes 1 + 1 \otimes g_2 - g_1 \otimes g_2
  \leq g_1 \otimes 1 + 1 \otimes g_2.
\]
Condition~(\ref{P_4624_TLarge_Cut4})
now follows from the choices of $x_1$, $x_2$, $y_1$, and~$y_2$.
\end{proof}

\begin{cor}\label{C_4619_StFinStLg}
Let $A$ be a stably finite infinite dimensional simple unital \ca,
and let $B \subset A$ be a large subalgebra.
Then $B$ is stably large in~$A$.
\end{cor}

\begin{proof}
In Proposition~\ref{P_4624_TLarge}
take $A_1 = B_1 = M_n$,
$A_2 = A$, and $B_2 = B$.
\end{proof}

\section{The Cuntz semigroup of a large subalgebra}\label{Sec_CuLg}

\indent
In this section,
we prove our main results.
If $B \subset A$ is stably large
(sometimes merely large suffices),
then $A$ and $B$ have the same traces and the same quasitraces.
Moreover,
$A$ is finite or purely infinite \ifo{}
$B$ has the same property.
If also $A$ is stably finite,
then $A$ and $B$ have the same
purely positive part of the Cuntz semigroup
(but not necessarily the same $K_0$-group)
and they have the same radius of comparison.

We consider traces first.

\begin{lem}\label{L_4622_ExtTr}
Let $A$ be an infinite dimensional simple unital \ca,
and let $B \subset A$ be a large subalgebra.
Let $\ta \in {\operatorname{T}} (B)$.
Then there exists a unique state $\om$ on~$A$
such that $\om |_B = \ta$.
\end{lem}

\begin{proof}
Existence of $\om$ follows from the Hahn-Banach Theorem.

For uniqueness,
let $\om_1$ and $\om_2$ be states on~$A$ such that
$\om_1 |_B = \om_2 |_B = \ta$,
let $a \in A_{+}$,
and let $\ep > 0$.
We prove that $| \om_1 (a) - \om_2 (a) | < \ep$.
\Wolog{} $\| a \| \leq 1$.

It follows from Proposition~\ref{P_2627_NoSmp}
that $B$ is simple
and from Proposition~\ref{P-2729LgInfD}
that $B$ is infinite dimensional.
So Corollary~\ref{C-2718CuDiv}
provides $y \in B_{+} \setminus \{ 0 \}$
such that $d_{\ta} (y) < \frac{\ep^2}{64}$
(for the particular choice of $\ta$ we are using).
Use Lemma~\ref{L:LargeStaysPositive}
to find $c \in A_{+}$ and $g \in B_{+}$
such that
\[
\| c \| \leq 1,
\,\,\,\,\,\,
\| g \| \leq 1,
\,\,\,\,\,\,
\| c - a \| < \frac{\ep}{4},
\,\,\,\,\,\,
(1 - g) c (1 - g) \in B,
\andeqn
g \precsim_B y.
\]

For $j = 1, 2$,
the Cauchy-Schwarz inequality
gives
\begin{equation}\label{Eq_4624_CauSch}
| \om_j (r s) |
 \leq \om_j (r r^*)^{1/2} \om_j (s^* s)^{1/2}
\end{equation}
for all $r, s \in A$.
Also,
by Lemma~\ref{L:CzBasic}(\ref{L:CzBasic:LCzFCalc})
we have $g^2 \sim_B g \precsim_B y$.
Since $\| g^2 \| \leq 1$ and $\om_j |_{B} = \ta$
is a \tst,
it follows that
$\om_j (g^2) \leq d_{\ta} (y) < \frac{\ep^2}{64}$.
Using $\| c \| \leq 1$,
we then get
\[
| \om_j (g c) |
 \leq \om_j (g^2)^{1/2} \om_j (c^2)^{1/2}
 < \frac{\ep}{8}
\]
and
\[
| \om_j ((1 - g) c g) |
  \leq \om_j \big( (1 - g) c^2 (1 - g) \big)^{1/2} \om_j (g^2)^{1/2}
 < \frac{\ep}{8}.
\]
So
\begin{align*}
\big| \om_j (c) - \ta ( (1 - g) c (1 - g) ) \big|
& = \big| \om_j (c) - \om_j ( (1 - g) c (1 - g) ) \big|
\\
& \leq | \om_j (g c) | + | \om_j ((1 - g) c g) |
  < \frac{\ep}{4}.
\end{align*}
Also $| \om_j (c) - \om_j (a) | < \frac{\ep}{4}$.
So
\[
\big| \om_j (a) - \ta ( (1 - g) c (1 - g) ) \big| < \frac{\ep}{2}.
\]
Thus $| \om_1 (a) - \om_2 (a) | < \ep$.
\end{proof}

\begin{thm}\label{L_4622_SameT}
Let $A$ be an infinite dimensional simple unital \ca,
and let $B \subset A$ be a large subalgebra.
Then the restriction map
${\operatorname{T}} (A) \to {\operatorname{T}} (B)$
is bijective.
\end{thm}

\begin{proof}
Let $\ta \in {\operatorname{T}} (B)$.
We show that there is a unique $\om \in {\operatorname{T}} (A)$
such that $\om |_B = \ta$.
Lemma~\ref{L_4622_ExtTr}
shows that there is a  unique state $\om$ on~$A$
such that $\om |_B = \ta$,
and it suffices to show that $\om$ is a trace.
Thus let $a_1, a_2 \in A$ satisfy $\| a_1 \| \leq 1$
and $\| a_2 \| \leq 1$,
and let $\ep > 0$.
We show that $| \om (a_1 a_2) - \om (a_2 a_1) | < \ep$.

It follows from Proposition~\ref{P_2627_NoSmp}
that $B$ is simple
and from Proposition~\ref{P-2729LgInfD}
that $B$ is infinite dimensional.
So Corollary~\ref{C-2718CuDiv}
provides $y \in B_{+} \setminus \{ 0 \}$
such that $d_{\ta} (y) < \frac{\ep^2}{64}$.
Use Lemma~\ref{L:LargeDecNorm}
to find $c_1, c_2 \in A$ and $g \in B_{+}$
such that
\[
\| c_j \| \leq 1,
\,\,\,\,\,\,
\| c_j - a_j \| < \frac{\ep}{8},
\andeqn
(1 - g) c_j \in B
\]
for $j = 1, 2$,
and such that $\| g \| \leq 1$ and $g \precsim_B y$.
By Lemma~\ref{L:CzBasic}(\ref{L:CzBasic:LCzFCalc})
we have $g^2 \sim g \precsim_B y$.
Since $\| g^2 \| \leq 1$ and $\om |_{B} = \ta$
is a \tst,
it follows that
$\om (g^2) \leq d_{\ta} (y) < \frac{\ep^2}{64}$.

We claim that
\[
\big| \om ( (1 - g) c_1 (1 - g) c_2 ) - \om ( c_1 c_2 ) \big|
 < \frac{\ep}{4}.
\]
Using the Cauchy-Schwarz inequality
((\ref{Eq_4624_CauSch}) in the previous proof),
we get
\[
| \om (g c_1 c_2) |
 \leq \om (g^2)^{1/2} \om (c_2^* c_1^* c_1 c_2)^{1/2}
 \leq \om (g^2)^{1/2}
 < \frac{\ep}{8}.
\]
Similarly,
and also at the second step using $\| c_2 \| \leq 1$,
$(1 - g) c_1 g \in B$,
and the fact that $\om |_{B}$ is a \tst,
\begin{align*}
\big| \om ( (1 - g) c_1 g c_2) \big|
& \leq \om \big( (1 - g) c_1 g^2 c_1^* (1 - g) \big)^{1/2}
      \om (c_2^* c_2)^{1/2}
\\
& \leq \om \big( g c_1^* (1 - g)^2 c_1 g \big)^{1/2}
  \leq \om (g^2)^{1/2}
  < \frac{\ep}{8}.
\end{align*}
The claim now follows from the estimate
\[
\big| \om ( (1 - g) c_1 (1 - g) c_2 ) - \om ( c_1 c_2 ) \big|
 \leq \big| \om ( (1 - g) c_1 g c_2) \big| + | \om (g c_1 c_2) |.
\]

Similarly
\[
\big| \om ( (1 - g) c_2 (1 - g) c_1 ) - \om ( c_2 c_1 ) \big|
 < \frac{\ep}{4}.
\]

Since $(1 - g) c_1, \, (1 - g) c_2 \in B$
and  $\om |_{B}$ is a \tst,
we get
\[
\om ( (1 - g) c_1 (1 - g) c_2 ) = \om ( (1 - g) c_2 (1 - g) c_1 ).
\]
Therefore $| \om ( c_1 c_2 ) - \om ( c_2 c_1 ) | < \frac{\ep}{2}$.

Now,
using $\| c_2 \| \leq 1$ and $\| a_1 \| \leq 1$,
we have
\[
\| c_1 c_2 - a_1 a_2 \|
 \leq \| c_1 - a_1 \| \cdot \| c_2 \|
         + \| a_1 \| \cdot \| c_2 - a_2 \|
  < \frac{\ep}{8} + \frac{\ep}{8}
  = \frac{\ep}{4},
\]
and similarly
$\| c_2 c_1 - a_2 a_1 \| < \frac{\ep}{4}$.
It now follows that $| \om (a_1 a_2) - \om (a_2 a_1) | < \ep$.
\end{proof}

We now prove the two key lemmas relating the Cuntz semigroup
of a stably large subalgebra
to that of the containing algebra.
The first is that if we have two elements in the Cuntz semigroup
of the containing algebra with a gap between them,
then one can find (up to~$\ep$)
an element of the Cuntz semigroup of the subalgebra which lies
between them.

\begin{lem}\label{L-2720SDomToB}
Let $A$ be an infinite dimensional simple unital \ca,
and let $B \subset A$ be a stably large subalgebra.
Let $a, b, x \in (K \otimes A)_{+}$,
with $x \neq 0$,
and suppose that $a \oplus x \precsim_A b$.
Then for every $\ep > 0$ there are
$n \in \N$,
$c \in (M_n \otimes B)_{+}$,
and $\dt > 0$
such that $(a - \ep)_{+} \precsim_A c \precsim_A (b - \dt)_{+}$.
\end{lem}

\begin{proof}
We first assume that $a, b \in A$,
and that there is $x \in A_{+} \setminus \{ 0 \}$
such that $a \oplus x \precsim_A b$.

Choose $\ep_0 > 0$ such that
\[
\ep_0 < \min \left( \| x \|, \, \frac{\ep}{3} \right).
\]
In particular,
$(x - \ep_0)_{+} \neq 0$.
Use Lemma~\ref{L:CzBasic}(\ref{L:CzBasic:LMinusEp})
to choose $\dt > 0$ such that
\begin{equation}\label{Eq:2723axEbD}
(a - \ep_0)_{+} \oplus (x - \ep_0)_{+} \precsim_A (b - \dt)_{+}.
\end{equation}
Use Lemma~\ref{L-2720L1} (ignoring most of the parts of the conclusion)
to choose $y \in B_{+} \setminus \{ 0 \}$
such that $y \precsim_A (x - \ep_0)_{+}$.
Use Lemma~\ref{L:LargeStaysPositive}
to choose $g \in B_{+}$ and $a_0, b_0 \in A_{+}$
such that $g \precsim_B y$ and
\[
(1 - g) a_0 (1 - g) \in B,
\,\,\,\,\,\,
(1 - g) b_0 (1 - g) \in B,
\,\,\,\,\,\,
\| a_0 - a \| < \ep_0,
\andeqn
\| b_0 - b \| < \ep_0.
\]
Set
\[
a_1 = \big[ (1 - g) a_0 (1 - g) - 2 \ep_0 \big]_{+}
\andeqn
c = a_1 \oplus g.
\]
Lemma~\ref{L:C2} implies that
\begin{equation}\label{Eq:2723a0c}
(a_0 - 2 \ep_0)_{+} \precsim_A c.
\end{equation}

Using $\| a_0 - a \| < \ep_0$,
Corollary~\ref{C:MMvsM},
and $3 \ep_0 < \ep$,
we get
\begin{equation}\label{Eq:2723aa0}
(a - \ep)_{+} \precsim_A (a_0 - 2 \ep_0)_{+}.
\end{equation}
We also have
\begin{equation}\label{Eq:2723gxe0}
g \precsim_B y \precsim_A (x - \ep_0)_{+}.
\end{equation}

We next claim that
$a_1 \precsim_A (a - \ep_0)_{+}$.
To prove the claim,
use $\| a_0 - a \| < \ep_0$ to get
\[
\| (1 - g) a_0 (1 - g) - (1 - g) a (1 - g) \| < \ep_0.
\]
Therefore,
using Corollary~\ref{C:MMvsM} at the first step,
Lemma~\ref{L:CzBasic}(\ref{L:CzBasic:LCzCommEp}) at the second step,
and Lemma~\ref{L:CzCompIneq}
and $a^{1/2} (1 - g)^2 a^{1/2} \leq a$ at the third step,
we have
\[
a_1 \precsim_A \big[ (1 - g) a (1 - g) - \ep_0 \big]_{+}
    \sim_A \big[ a^{1/2} (1 - g)^2 a^{1/2} - \ep_0 \big]_{+}
    \precsim_A (a - \ep_0)_{+},
\]
as desired.

Combining this claim with the definition of~$c$
and~(\ref{Eq:2723gxe0}),
we get
\begin{equation}\label{Eq:2723a1ae0}
c \precsim_A (a - \ep_0)_{+} \oplus (x - \ep_0)_{+}.
\end{equation}
Using, in order, (\ref{Eq:2723aa0}),
(\ref{Eq:2723a0c}),
(\ref{Eq:2723a1ae0}),
and (\ref{Eq:2723axEbD}),
we now get
\[
(a - \ep)_{+}
 \precsim_A (a_0 - 2 \ep_0)_{+}
 \precsim_A c
 \precsim_A (a - \ep_0)_{+} \oplus (x - \ep_0)_{+}
 \precsim_A (b - \dt)_{+}.
\]
This completes the proof of the special case of the lemma.

We now consider the general case.
\Wolog{} $\ep < 1$ and $\| x \| = 1$.
Use Lemma~\ref{L:CzBasic}(\ref{L:CzBasic:LMinusEp})
to choose $\dt_0 > 0$ such that
\[
\big( a - \tfrac{\ep}{2} \big)_{+}
   \oplus \big( x - \tfrac{\ep}{2} \big)_{+}
 = \big[ (a \oplus x) - \tfrac{\ep}{2} \big]_{+}
 \precsim_A (b - \dt_0)_{+}.
\]
Use Lemma~\ref{L-2720KToMn}
to choose $n \in \N$ and $x_0 \in (M_n \otimes A)_{+}$
and $a_0, b_0 \in (M_n \otimes B)_{+}$
such that
\begin{equation}\label{Eq_4626_Str}
\big( x - \tfrac{\ep}{2} \big)_{+} \sim_A x_0,
\,\,\,\,\,\,
\big( a - \tfrac{\ep}{2} \big)_{+} \sim_B a_0,
\andeqn
\big( b - \tfrac{\ep}{2} \big)_{+} \sim_B b_0.
\end{equation}
Then $x_0 \neq 0$.
Since $M_n (A)$ is large in $M_n (B)$,
the case already done gives $r \in \N$,
$c \in (M_{r n} \otimes A)_{+}$,
and $\dt_1 > 0$ such that
\[
\big( a_0 - \tfrac{\ep}{2} \big)
 \precsim_A c
 \precsim_A (b_0 - \dt_1)_{+}.
\]
Substituting using~(\ref{Eq_4626_Str}),
setting $\dt = \dt_0 + \dt_1$,
and using Lemma~\ref{L:CzBasic}(\ref{L:CzBasic:MinIter}),
we get
$(a - \ep)_{+}
 \precsim_A c
 \precsim_A ( b - \dt )_{+}$.
This completes the proof.
\end{proof}

The second key lemma is that if two elements
in the Cuntz semigroup of a stably large subalgebra
are comparable in the containing algebra,
with a mild condition on the gap between them,
then they are comparable in the subalgebra.
For later use,
we divide the proof of this lemma in two steps.

\begin{lem}\label{L-2721AToB_Pre}
Let $A$ be an infinite dimensional simple unital \ca,
and let $B \subset A$ be a large subalgebra.
Let $a, b \in B_{+}$ and $c, x \in A_{+}$
satisfy $x \neq 0$,
$a \precsim_A c$,
$c x = 0$,
and $c + x \in {\overline{b A b}}$.
Then $a \precsim_B b$.
\end{lem}

\begin{proof}
If $a = 0$, there is nothing to prove,
so assume $a \neq 0$.
Then $c$,  $x$, and $b$ are nonzero.
Thus \wolog{}
\[
\| a \| = \| c \| = \| x \| = \| b \| = 1.
\]
Let $\ep > 0$.
We prove that $(a - \ep)_{+} \precsim_B b$.
\Wolog{} $\ep < 1$.

Use Proposition~\ref{P_4619_ConeSj}
(and rescaling) to find $\rh > 0$
such that whenever $D$ is a \ca{}
and $r_0, s_0 \in D$
satisfy $0 \leq r_0, s_0 \leq 9$ and $\| r_0 s_0 \| < \rh$,
then there exist $r, s \in D$
such that
\[
0 \leq r, s \leq 9,
\,\,\,\,\,\,
r s = 0,
\,\,\,\,\,\,
\| r - r_0 \| < \frac{\ep}{2},
\andeqn
\| s - s_0 \| < \frac{\ep}{2}.
\]

Set
\[
\dt = \min \left( 1, \, \frac{\rh}{18}, \, \frac{\ep}{22} \right).
\]
Since $a \precsim_A c$,
there is $v \in A$
such that $\| v^* c v - a \| < \dt$.
Set $w = c^{1/2} v$.
Then $\| w^* w - a \| < \dt$,
so $\| w^* w \| < 1 + \dt$,
whence $\| w \| < (1 + \dt)^{1/2} < 1 + \dt$.

Since $(b^{1/n})_{n \in \N}$
is an approximate identity for ${\overline{b A b}}$,
there is $n$ such that the element $e = b^{1/n} \in B$
satisfies $\big\| e c^{1/2} - c^{1/2} \big\| < (1 + \| v \|)^{-1} \dt$.
Thus
\[
\| e w - w \|
 = \big\| e c^{1/2} v - c^{1/2} v \big\|
 < \dt.
\]
Also $\| e \| \leq 1$.

Use Lemma~\ref{L-2720L1}
to choose a positive element $y \in {\overline{b B b}}$ such that
$\| y \| = 1$, $y \precsim_A x$,
and $\| x y - y \| < \dt$.

Apply Definition~\ref{D_Large} to $B \subset A$
with $m = 1$,
with $a_1 = w^*$,
with $\big( y - \tfrac{\ep}{2} \big)_{+}$ in place of~$y$,
and with $\dt$ in place of~$\ep$.
We get $w_0 \in A$ and $g \in B_{+}$
such that
$\| w_0 - w \| < \dt$, $w_0 (1 - g) \in B$, $0 \leq g \leq 1$,
and $g \precsim_B \big( y - \tfrac{\ep}{2} \big)_{+}$.
Then $\| w_0 \| < \| w \| + \dt < 1 + 2 \dt$.
Since $\| e w - w \| < \dt$,
we have $\| e w_0 - w \| < 2 \dt$
and $\| e w_0 - w_0 \| < 3 \dt$.
Also $e w_0 (1 - g) \in B$.

Since $c^{1/2} x = 0$,
we have $w^* x = 0$,
whence
\[
\| w^* y \|
 \leq \| w^* \| \cdot \| y - x y \|
 < (1 + \dt) \dt
 < 2 \dt.
\]
So
\[
\| w_0^* y \|
 \leq \| w_0^* - w^* \| \cdot \| y \| + \| w^* y \|
 < \dt + 2 \dt
 = 3 \dt,
\]
and
\[
\| w_0^* e y \|
 \leq \| w_0^* e - w_0^* \| \cdot \| y \| + \| w_0^* y \|
 < 3 \dt + 3 \dt
 = 6 \dt.
\]
Therefore
\[
\big\| e w_0 (1 - g)^2 w_0^* e y \big\|
 \leq \| e \| \cdot \| w_0 \| \cdot \| 1 - g \|^2 \cdot \| w_0^* e y \|
 < (1 + 2 \dt) 6 \dt
 < 18 \dt
 \leq \rh.
\]
{}From $\dt < 1$ and $\| w_0 \| < 1 + 2 \dt$, we get
$\big\| e w_0 (1 - g)^2 w_0^* e \big\| < (1 + 2 \dt)^2 < 9$.
Since also $\| y \| \leq 9$,
and since $e w_0 (1 - g)^2 w_0^* e, \, y \in {\overline{b B b}}$,
the choice of~$\rh$ provides $r, z \in {\overline{b B b}}$
such that
\begin{equation}\label{Eq:2723rzNorm}
0 \leq r, z \leq 9,
\,\,\,\,\,\,
r z = 0,
\,\,\,\,\,\,
\big\| e w_0 (1 - g)^2 w_0^* e - r \big\| < \frac{\ep}{2},
\andeqn
\| y - z \| < \frac{\ep}{2}.
\end{equation}

We saw above that $\| e w_0 - w \| < 2 \dt$,
so
\begin{align*}
\| w_0^* e^2 w_0 - a \|
& \leq \| w_0^* e - w^* \| \cdot \| e \| \cdot \| w_0 \|
    + \| w^* \| \cdot \| e w_0 - w \| + \| w^* w - a \|
  \\
& < 2 \dt (1 + 2 \dt) + (1 + \dt) 2 \dt + \dt
  < 11 \dt
  \leq \frac{\ep}{2}.
\end{align*}
Therefore
\begin{equation}\label{Eq:2728Add1mg}
\big\| (1 - g) a (1 - g) - (1 - g) w_0^* e^2 w_0 (1 - g) \big\|
 < \frac{\ep}{2}.
\end{equation}
In~$B$,
we now have the following chain of subequivalences,
in which we use Lemma~\ref{L:C2} at the first step,
(\ref{Eq:2728Add1mg})~and
Corollary~\ref{C:MMvsM} on the first summand at the second step,
Lemma~\ref{L:CzBasic}(\ref{L:CzBasic:LCzCommEp})
at the third step,
the estimates in~(\ref{Eq:2723rzNorm})
and Lemma \ref{L:CzBasic}(\ref{L:CzBasic:LCzWithinEp})
at the fourth step,
Lemma~\ref{L:CzBasic}(\ref{L:CzBasic:Orth}) at the fifth step,
and $r + z \in {\overline{b B b}}$
and Lemma \ref{L:CzBasic}(\ref{L:CzBasic:Her}) at the sixth step:
\begin{align*}
(a - \ep)_{+}
& \precsim_B [ (1 - g) a (1 - g) - \ep]_{+} \oplus g
\\
& \precsim_B
     \big[ (1 - g) w_0^* e^2 w_0 (1 - g) - \tfrac{\ep}{2} \big]_{+}
        \oplus \big( y - \tfrac{\ep}{2} \big)_{+}
\\
& \sim_B
     \big( e w_0 (1 - g)^2 w_0^* e - \tfrac{\ep}{2} \big)_{+}
        \oplus \big( y - \tfrac{\ep}{2} \big)_{+}
\\
& \precsim_B r \oplus z
  \sim_B r + z
  \precsim_B b.
\end{align*}
Since $\ep > 0$ is arbitrary,
Lemma~\ref{L:CzBasic}(\ref{L:CzBasic:LMinusEp})
implies that $a \precsim_B b$.
\end{proof}

\begin{lem}\label{L-2721AToB}
Let $A$ be an infinite dimensional simple unital \ca,
and let $B \subset A$ be a stably large subalgebra.
Let $a, b \in (K \otimes B)_{+}$ and $c, x \in (K \otimes A)_{+}$
satisfy $x \neq 0$,
$a \precsim_A c$,
and $c \oplus x \precsim_A b$.
Then $a \precsim_B b$.
\end{lem}

\begin{proof}
We first suppose that $c x = 0$
and $c + x \in {\overline{b (K \otimes A) b}}$.
If $a = 0$, there is nothing to prove,
so assume $a \neq 0$.
Then $b$, $c$, and $x$ are nonzero.
Thus \wolog{}
\[
\| a \| = \| b \| = \| c \| = \| x \| = 1.
\]

Let $\ep > 0$.
Use Lemma \ref{L:CzBasic}(\ref{L:CzBasic:LMinusEp})
to choose $\dt > 0$ such that
\begin{equation}\label{Eq_4629_55S}
( a - \ep )_{+} \precsim_A (c - \dt)_{+}.
\end{equation}
%
Set $\dt_0 = \tfrac{1}{2} \min ( 1, \dt )$.

Use Lemma~\ref{L-2720KToMn}
to choose $k \in \N$ and $a_0 \in M_k (B)_{+}$
such that
\begin{equation}\label{Eq_4820_StSt}
( a - \ep )_{+} \sim_B a_0.
\end{equation}
Use Proposition~\ref{P_4619_ConeSj} to find $\rh > 0$
such that whenever $D$ is a \ca{}
and $r_0, s_0 \in D$
satisfy $0 \leq r_0, s_0 \leq 1$ and $\| r_0 s_0 \| < \rh$,
then there exist $r, s \in D$
such that
\[
0 \leq r, s \leq 1,
\,\,\,\,\,\,
r s = 0,
\,\,\,\,\,\,
\| r - r_0 \| < \dt_0,
\andeqn
\| s - s_0 \| < \dt_0.
\]
Set
$\ep_0 = \tfrac{1}{6} \min (\rh, \dt_0)$.
Thus
\begin{equation}\label{Eq_3817St}
6 \ep_0 + \dt_0 \leq \dt
\andeqn
6 \ep_0 + \dt_0 \leq 1.
\end{equation}
For sufficiently large $n \in \N$,
the element $e_0 = b^{1/n} \in K \otimes B$
satisfies
\[
\| e_0 c - c \| < \ep_0
\andeqn
\| e_0 x - x \| < \ep_0.
\]
Choose $l \in \N$ with $l \geq k$ and $e_1 \in M_l (B)_{+}$
such that $\| e_1 - e_0 \| < \ep_0$.
Then
\[
\| e_1 c - c \| < 2 \ep_0
\andeqn
\| e_1 x - x \| < 2 \ep_0.
\]
Set $e = (e_1 - \ep_0)_{+}$.
We have $\| e \| \leq1$ since $\| e_1 \| < 1 + \ep_0$.
Lemma~\ref{L:CzBasic}(\ref{L:CzBasic:LCzWithinEp})
and Lemma~\ref{L:CzBasic}(\ref{L:CzBasic:LCzFCalc})
imply that
\begin{equation}\label{Eq_3817StSt}
e \precsim_B e_0 \sim_B b.
\end{equation}
Also,
$\| e - e_1 \| \leq \ep_0$,
so
\[
\| e c - c \| < 3 \ep_0
\andeqn
\| e x - x \| < 3 \ep_0.
\]
Define $d_0, y_0 \in {\overline{e M_l (A) e}}$
by $d_0 = e c e$ and $y_0 = e x e$.
Then
\[
0 \leq d_0 \leq 1,
\,\,\,\,\,\,
0 \leq y_0 \leq 1,
\,\,\,\,\,\,
\| d_0 - c \| < 6 \ep_0,
\andeqn
\| y_0 - x \| < 6 \ep_0.
\]
Using $c x = 0$,
we get
\[
\| d_0 y_0 \|
 = \| (e c e) (e x e) - e c x e \|
 \leq \| e \| \cdot \| c e - c \| \cdot \| e x e \|
       + \| e c \| \cdot \| e x - x \| \cdot \| e \|
 < 6 \ep_0
 \leq \rh.
\]
Therefore there exist $d, y \in {\overline{e M_l (A) e}}$
such that
\[
0 \leq d, y \leq 1,
\,\,\,\,\,\,
d y = 0,
\,\,\,\,\,\,
\| d - d_0 \| < \dt_0,
\andeqn
\| y - y_0 \| < \dt_0.
\]
It follows that
\[
\| d - c \| < 6 \ep_0 + \dt_0
\andeqn
\| y - x \| < 6 \ep_0 + \dt_0.
\]

We are going to apply Lemma~\ref{L-2721AToB_Pre} with
\[
a = a_0,
\,\,\,\,\,\,
b = e,
\,\,\,\,\,\,
c = d,
\andeqn
x = y.
\]
We check its hypotheses.
Since $\| x \| = 1$ and $6 \ep_0 + \dt_0 \leq 1$
(by~\ref{Eq_3817St}),
we have $y \neq 0$.
Using~(\ref{Eq_4820_StSt}) at the first step,
(\ref{Eq_4629_55S}) at the second step,
and $6 \ep_0 + \dt_0 \leq \dt$ (from~(\ref{Eq_3817St}))
and Lemma \ref{L:CzBasic}(\ref{L:CzBasic:LCzWithinEp})
at the third step,
we have
\[
a_0
  \sim_B (a - \ep)_{+}
  \precsim_A (c - \dt)_{+}
  \precsim_A d.  
\]
We have $d y = 0$
and $d + y \in {\overline{e M_l (A) e}}$
by construction.
Since $M_l (B)$ is large in $M_l (A)$,
we have verified the hypotheses of Lemma~\ref{L-2721AToB_Pre}.
So
$a_0 \precsim_B e$.
Therefore, using~(\ref{Eq_3817StSt}) at the last step,
\[
(a - \ep)_{+}
 \sim_B a_0
 \precsim_B e
 \precsim_B b.
\]
Since $\ep > 0$ is arbitrary,
it follows from Lemma \ref{L:CzBasic}(\ref{L:CzBasic:LMinusEp})
that $a \precsim_B b$.

Now
we prove the lemma as stated.
Let $\ep > 0$.
Use Lemma \ref{L:CzBasic}(\ref{L:CzBasic:LMinusEp})
to choose $\dt > 0$ such that
$(a - \ep)_{+} \precsim_A [ (c \oplus x) - \dt]_{+}$.
We also require that $\dt < \| x \|$,
so that $(x - \dt)_{+} \neq 0$.
Lemma \ref{L:CzBasic}(\ref{L:CzBasic:N6}),
applied in $M_2 (K \otimes A)$,
provides $v \in M_2 (K \otimes A)$
such that $v^* v = [ (c \oplus x) - \dt]_{+}$
and
$v v^* \in {\overline{(b \oplus 0) M_2 (K \otimes A) (b \oplus 0)}}$.
Lemma \ref{L:CzBasic}(\ref{L:CzBasic:N5})
gives an isomorphism
\[
\ph \colon {\overline{v^* v M_2 (K \otimes A) v^* v}}
   \to {\overline{v v^* M_2 (K \otimes A) v v^*}}
\]
such that, for every positive
$z \in {\overline{v^* v M_2 (K \otimes A) v^* v}}$,
we have
$z \sim_A \ph (z)$.
Set
\[
c_0 = \ph ( [c - \dt]_{+} \oplus 0 )
\andeqn
x_0 = \ph ( 0 \oplus [x - \dt]_{+} ).
\]
Then $c_0$ and $x_0$
are orthogonal positive elements of
\[
{\overline{b (K \otimes A) b}}
 = {\overline{(b \oplus 0) M_2 (K \otimes A) (b \oplus 0)}}
\]
such that $x_0 \neq 0$,
and,
using Lemma \ref{L:CzBasic}(\ref{L:CzBasic:Orth})
at the second step,
\[
(a - \ep)_{+}
 \precsim_A (c - \dt)_{+} \oplus (x - \dt)_{+}
 \sim_A c_0 + x_0.
\]
Therefore the result obtained above implies that
$(a - \ep)_{+} \precsim_B b$.
Since $\ep > 0$ is arbitrary,
it follows from Lemma \ref{L:CzBasic}(\ref{L:CzBasic:LMinusEp})
that $a \precsim_B b$.
\end{proof}

\begin{thm}\label{T-2720CuSurj}
Let $A$ be an infinite dimensional simple unital \ca,
and let $B \subset A$ be a stably large subalgebra.
Let $\io \colon B \to A$ be the inclusion map.
For every $\et \in {\operatorname{Cu}} (A)$
which is not the class of a \pj,
there is
$\mu \in {\operatorname{Cu}} (B)$ such that $\io_* (\mu) = \et$.
\end{thm}

\begin{proof}
Choose $y \in (K \otimes A)_{+}$
such that $\et = \langle y \rangle$.
Since $\et$ is not the class of a \pj,
we have
\begin{equation}\label{Eq:2723SpecCl}
0 \in {\overline{\spec (y) \setminus \{ 0 \} }}.
\end{equation}

We construct sequences $(c_n)_{n \in \Nz}$ in $(K \otimes B)_{+}$
and $(\ep_n)_{n \in \Nz}$ and $(\rh_n)_{n \in \Nz}$
of positive numbers such that
\[
\ep_0 > \rh_0 > \ep_1 > \rh_1 > \cdots > 0,
\,\,\,\,\,\,
\limi{n} \ep_n = 0,
\]
\[
c_0 \precsim_A (y - \ep_0)_{+} \precsim_A (y - \rh_0)_{+}
 \precsim_A c_1 \precsim_A (y - \ep_1)_{+} \precsim_A (y - \rh_1)_{+}
 \precsim_A c_2
 \precsim_A \cdots \precsim_A y,
\]
and $\spec (y) \cap (\rh_n, \ep_n) \neq \varnothing$ for $n \in \Nz$.

The construction is by induction on~$n$.
To get the condition $\limi{n} \ep_n = 0$,
it suffices to require $\ep_{n + 1} \leq \tfrac{1}{2} \ep_n$
for $n \in \Nz$.
We take $c_0 = 0$ and $\ep_0 = 1$.
By~(\ref{Eq:2723SpecCl}), there is $\rh_0 \in (0, \ep_0)$
such that $\spec (y) \cap (\rh_0, \ep_0) \neq \varnothing$.

Suppose now that $c_n$, $\ep_n$, and $\rh_n$ are given.
By~(\ref{Eq:2723SpecCl}),
we have
$\spec (y) \cap ( 0, \rh_n ) \neq \varnothing$.
Therefore there is a \cfn{} $f \colon [0, \infty) \to [0, \infty)$
such that $\supp (f) \subset ( 0, \rh_n )$
and $f (y) \neq 0$.
We have
$( y - \rh_n )_{+} \oplus f (y) \precsim_A y$,
so Lemma~\ref{L-2720SDomToB}
provides $c_{n + 1} \in (K \otimes B)_{+}$
and $\dt > 0$
such that
$( y - \rh_n )_{+} \precsim_A c_{n + 1} \precsim_A (y - \dt)_{+}$.
Take
$\ep_{n + 1} = \min \big( \tfrac{1}{2} \rh_n, \dt \big)
 < \tfrac{1}{2} \ep_n$.
Then use~(\ref{Eq:2723SpecCl})
to choose $\rh_{n + 1} \in (0, \ep_{n + 1})$
such that $\spec (y) \cap (\rh_{n + 1}, \ep_{n + 1}) \neq \varnothing$.
This completes the construction.

For $n \in \Nz$,
choose a \cfn{} $f_n \colon [0, \infty) \to [0, \infty)$
such that $\supp (f_n) \subset (\rh_n, \ep_n)$
and $f_n (y) \neq 0$.
Apply Lemma~\ref{L-2721AToB}
with $a = c_n$,
$b = c_{n + 1}$,
$c = (y - \ep_n)_{+}$,
and $x = f_n (y)$,
to get $c_{n} \precsim_B c_{n + 1}$.
We can now take $\mu = \sup_{n \in \Nz} \langle c_n \rangle$,
which exists in ${\operatorname{Cu}} (B)$
by Theorem~\ref{T_2Y25_CxSup}(\ref{T_2Y25_CxSup_Ex}).
Since $\io_*$ preserves supremums
(by Theorem~\ref{T_2Y25_CxSup}(\ref{T_2Y25_CxSup_Prv}))
and
$\langle a \rangle = \sup_{n \in \Nz} \langle (y - \ep_n)_{+} \rangle$
(by Lemma~\ref{L_3905_CuWay}(\ref{L_3905_CuWay_SupEp})),
we get $\io_* (\mu) = \et$.
\end{proof}

\begin{thm}\label{T-2725Inj}
Let $A$ be an infinite dimensional simple unital \ca,
and let $B \subset A$ be a stably large subalgebra.
Let $\io \colon B \to A$ be the inclusion map.
Let $\mu, \et \in {\operatorname{Cu}} (B)$,
and suppose that $\et$ is not the class of a \pj.
Then:
\begin{enumerate}
\item\label{T-2725Inj-Ord}
If $\io_* (\mu) \leq \io_* (\et)$,
then $\mu \leq \et$.
\item\label{T-2725Inj-Inj}
If $\mu$ is also not the class of a \pj,
and $\io_* (\mu) = \io_* (\et)$,
then $\mu = \et$.
\end{enumerate}
\end{thm}

If $A$ is stably finite,
then in~(\ref{T-2725Inj-Inj})
it is automatic that $\mu$ is not the class of a \pj.
Using Proposition~\ref{P_4626_Fin} below,
this can be deduced from Lemma~\ref{L_2Y24_WhenP}.

\begin{proof}[Proof of Theorem~\ref{T-2725Inj}]
Choose $a, b \in (K \otimes B)_{+}$ such that
$\mu = \langle a \rangle$
and $\et = \langle b \rangle$.

For~(\ref{T-2725Inj-Ord}),
let $\ep > 0$ and
(using Lemma \ref{L:CzBasic}(\ref{L:CzBasic:LMinusEp}))
choose $\dt > 0$ such that
$(a - \ep)_{+} \precsim_A (b - \dt)_{+}$.
Since $\et$ is not the class of a \pj,
there is a \cfn{} $f \colon [0, \infty) \to [0, \infty)$
such that $\supp (f) \subset ( 0, \dt )$
and $f (b) \neq 0$.
Apply Lemma~\ref{L-2721AToB}
with $(a - \ep)_{+}$ in place of~$a$,
with $(b - \dt)_{+}$ in place of~$c$,
with $f (b)$ in place of~$x$,
and
with $b$ as given,
to get $(a - \ep)_{+} \precsim_B b$.
Since $\ep > 0$ is arbitrary,
it follows from Lemma \ref{L:CzBasic}(\ref{L:CzBasic:LMinusEp})
that $a \precsim_B b$.

Under the hypotheses of~(\ref{T-2725Inj-Inj}),
we can use~(\ref{T-2725Inj-Ord}) to get $\et \leq \mu$ as well.
Thus $\mu = \et$.
\end{proof}

\begin{thm}\label{C_3X08_IsoOnPure}
Let $A$ be a stably finite infinite dimensional simple unital \ca,
and let $B \subset A$ be a large subalgebra.
Let $\io \colon B \to A$ be the inclusion map.
Then $\io_*$ defines an order and semigroup isomorphism
from ${\operatorname{Cu}}_{+} (B) \cup \{ 0 \}$
(as in Definition~\ref{D_2Y24_Pure})
to ${\operatorname{Cu}}_{+} (A) \cup \{ 0 \}$.
\end{thm}

It is not true that $\io_*$ defines an isomorphism
from ${\operatorname{Cu}} (B)$ to ${\operatorname{Cu}} (A)$.
Example~\ref{E_4819_NoKIso}
shows that
$\io_* \colon {\operatorname{Cu}} (B) \to {\operatorname{Cu}} (A)$
need not be injective.

\begin{proof}[Proof of Theorem~\ref{C_3X08_IsoOnPure}]
It follows from Corollary~\ref{C_4619_StFinStLg}
that $B$ is stably large in~$A$.
Also, $B$ is stably finite because it is a subalgebra of~$A$.
So Corollary~\ref{C_2Y24_PureAbs} implies
that ${\operatorname{Cu}}_{+} (A) \cup \{ 0 \}$
and ${\operatorname{Cu}}_{+} (B) \cup \{ 0 \}$
are in fact ordered semigroups.
It is clear that $\io_*$ is order preserving and additive.
We must therefore prove the following four statements:
\begin{enumerate}
\item\label{C_3X08_IsoOnPure_Inside}
$\io_* \big( {\operatorname{Cu}}_{+} (B) \cup \{ 0 \} \big)
   \subset {\operatorname{Cu}}_{+} (A) \cup \{ 0 \}$.
\item\label{C_3X08_IsoOnPure_Cont}
$\io_* \big( {\operatorname{Cu}}_{+} (B) \cup \{ 0 \} \big)
   \supset {\operatorname{Cu}}_{+} (A) \cup \{ 0 \}$.
\item\label{C_3X08_IsoOnPure_Inj}
$\io_* |_{{\operatorname{Cu}}_{+} (B) \cup \{ 0 \}}$ is injective.
\item\label{C_3X08_IsoOnPure_IfOrder}
If $\mu, \et \in {\operatorname{Cu}}_{+} (B) \cup \{ 0 \}$
and $\io_* (\mu) \leq \io_* (\et)$,
then $\mu \leq \et$.
\end{enumerate}

Our first step is to prove that
\begin{equation}\label{Eq_4819_StStStA}
\io_* (0) = 0,
\,\,\,\,\,\,
\io_* ( {\operatorname{Cu}}_{+} (B) )
  \subset {\operatorname{Cu}}_{+} (A),
\end{equation}
and
\begin{equation}\label{Eq_4819_StStSt2nd}
\io_* \big( {\operatorname{Cu}} (B) \setminus
    [{\operatorname{Cu}}_{+} (B) \cup \{ 0 \}] \big)
  \subset {\operatorname{Cu}} (A) \setminus
    [{\operatorname{Cu}}_{+} (A) \cup \{ 0 \}].
\end{equation}
It is obvious that $\io_* (0) = 0$.
By Lemma~\ref{L_2Y24_WhenP}, for any stably finite simple \ca~$D$,
the set
${\operatorname{Cu}} (D) \setminus
    [{\operatorname{Cu}}_{+} (D) \cup \{ 0 \}]$
is the set of classes $\langle p \rangle$
of nonzero \pj{s} $p \in K \otimes D$.
So the relation~(\ref{Eq_4819_StStSt2nd})
is also clear.
To prove the second part of~(\ref{Eq_4819_StStStA}),
let $\et \in {\operatorname{Cu}}_{+} (B)$.
Choose $b \in (K \otimes B)_{+}$
such that $\langle b \rangle = \et$.
Lemma~\ref{L_2Y24_WhenP}
implies that $0$ is not isolated in $\spec_{K \otimes B} (b)$.
So $0$ is not isolated in $\spec_{K \otimes A} (b)$.
Therefore Lemma~\ref{L_2Y24_WhenP}
implies that $\langle \io (b) \rangle \in {\operatorname{Cu}} (A)$
is actually in ${\operatorname{Cu}}_{+} (A)$,
as desired.

The statement~(\ref{C_3X08_IsoOnPure_Inside})
is now immediate from~(\ref{Eq_4819_StStStA}).
For~(\ref{C_3X08_IsoOnPure_Cont}),
let $\et \in {\operatorname{Cu}}_{+} (A) \cup \{ 0 \}$.
If $\et = 0$,
clearly $\et \in \io_* ( {\operatorname{Cu}}_{+} (B) \cup \{ 0 \} )$.
Otherwise, Theorem~\ref{T-2720CuSurj} provides
$\mu \in {\operatorname{Cu}} (B)$
such that $\io_* (\mu) = \et$.
It follows from the first part of~(\ref{Eq_4819_StStStA})
that $\mu \neq 0$,
and (\ref{Eq_4819_StStSt2nd}) now implies that
$\mu \in {\operatorname{Cu}}_{+} (B)$.

For (\ref{C_3X08_IsoOnPure_Inj})
and~(\ref{C_3X08_IsoOnPure_IfOrder}),
by~(\ref{Eq_4819_StStStA})
it is enough to consider ${\operatorname{Cu}}_{+} (B)$
in place of ${\operatorname{Cu}}_{+} (B) \cup \{ 0 \}$.
Now (\ref{C_3X08_IsoOnPure_Inj})
follows from Theorem \ref{T-2725Inj}(\ref{T-2725Inj-Inj}),
and (\ref{C_3X08_IsoOnPure_IfOrder})
follows from Theorem \ref{T-2725Inj}(\ref{T-2725Inj-Ord}).
\end{proof}

\begin{prp}\label{P-2724QTBij}
Let $A$ be an infinite dimensional simple unital \ca,
and let $B \subset A$ be a stably large subalgebra.
Then the restriction map
${\operatorname{QT}} (A) \to {\operatorname{QT}} (B)$ is bijective.
\end{prp}

\begin{proof}
Let $\io \colon B \to A$ be the inclusion map.
Then
$\io_* \colon {\operatorname{Cu}}_{+} (B) \cup \{ 0 \}
   \to {\operatorname{Cu}}_{+} (A) \cup \{ 0 \}$
is a semigroup and order isomorphism
by Theorem~\ref{C_3X08_IsoOnPure}.
Therefore
$\om \mapsto \om \circ \io_*$
is a bijection from the functionals $\om$
(as in Definition~\ref{D_2Y25_CzFcnl})
on ${\operatorname{Cu}}_{+} (A) \cup \{ 0 \}$
such that
\[
\sup \big( \big\{ \om (\et) \colon
 {\mbox{$\et \in {\operatorname{Cu}}_{+} (A) \cup \{ 0 \}$
     and $\et \leq \langle 1 \rangle$ in
     ${\operatorname{Cu}} (A)$}} \big\} \big)
 = 1
\]
to the analogous set of functionals
on ${\operatorname{Cu}}_{+} (B) \cup \{ 0 \}$.
So Lemma~\ref{L_3904_FclBj}
implies that $\om \mapsto \om \circ \io_*$
is a bijection from the functionals $\om$ on ${\operatorname{Cu}} (A)$
such that $\om ( \langle 1 \rangle ) = 1$
to the analogous set of functionals on ${\operatorname{Cu}} (B)$.
The proof is completed by applying
Theorem \ref{T_2Y25_DTau}(\ref{T_2Y25_DTau_Cu}).
\end{proof}

We recall the following definition.
We are relying on Theorem~\ref{T_2Y25_DTau}(\ref{T_2Y25_DTau_W})
(equivalently,
the discussion before Definition~6.1 of \cite{Tm1})
for the equivalence of our formulation
with the original version.

\begin{dfn}[Definition~6.1 of \cite{Tm1}]\label{D_4814_rc}
Let $A$ be a simple \uca.
For $r \in [0, \infty)$,
we say that {\emph{$A$ has $r$-comparison}}
(or {\emph{$W (A)$ has $r$-comparison}})
if whenever $a, b \in M_{\infty} (A)_{+}$
satisfy $d_{\ta} (a) + r < d_{\ta} (b)$
for all $\ta \in \QT (A)$,
then $a \precsim_A b$.
We further define the {\emph{radius of comparison of $A$}}
to be
\[
{\mathrm{rc}} (A)
 = \inf \big( \big\{ r \in [0, \infty)
   \colon {\mbox{$A$ has $r$-comparison}} \big\} \big).
\]
\end{dfn}

We warn that $r$-comparison and ${\mathrm{rc}} (A)$
are sometimes defined using tracial states
rather than quasitraces.
We presume that analogs of the results below are true
for those versions as well,
but we have not checked this.

We can also define a version using ${\operatorname{Cu}} (A)$.
The number one gets turns out to be just ${\mathrm{rc}} (A)$
(see Proposition~\ref{P_4817_rCompCu} below),
and this definition is only intended for convenience
of exposition in this paper.
Again, we use quasitraces, not just tracial states.

\begin{dfn}\label{D_4814_rcCu}
Let $A$ be a simple \uca.
For $r \in [0, \infty)$,
we say that {\emph{${\operatorname{Cu}} (A)$ has $r$-comparison}}
if whenever $a, b \in (K \otimes A)_{+}$
satisfy $d_{\ta} (a) + r < d_{\ta} (b)$
for all $\ta \in \QT (A)$,
then $a \precsim_A b$.
\end{dfn}

\begin{prp}\label{P_4817_rCompCu}
Let $A$ be a simple \uca{} and let $r \in [0, \infty)$.
Then $W (A)$ has $r$-comparison
\ifo{} ${\operatorname{Cu}} (A)$ has $r$-comparison.
\end{prp}

The comment after Definition~3.1 of~\cite{BRTTW}
claims that Proposition~\ref{P_4817_rCompCu} is true.
There seems to be a misprint,
since the reason given for this,
in Subsection~2.4 of~\cite{BRTTW},
does not address the following difficulty.
Suppose $a, b \in (K \otimes A)_{+}$
and $d_{\ta} (a) + r < d_{\ta} (b)$
for all $\ta \in \QT (A)$.
For $\ep > 0$ one needs to find $c \in M_{\infty} (A)_{+}$
such that $d_{\ta} ( (a - \ep)_{+}) + r < d_{\ta} (c)$
for all $\ta \in \QT (A)$.
The obvious approach only allows one to do this for
one choice of~$\ta$ at a time.

The following form of Dini's Theorem
solves this difficulty.
It is surely well known,
but we have not found a reference.

\begin{lem}\label{L_4817_Dini}
Let $X$ be a \chs,
let $(f_n)_{n \in \N}$
be a sequence of lower semi\cfn{s}
$f_n \colon X \to \R \cup \{ \infty \}$ such that
for all $x \in X$ we have
$f_1 (x) \leq f_2 (x) \leq \cdots$,
and let $g \colon X \to \R$ be a \cfn{}
such that $g (x) < \limi{n} f_n (x)$ for all $x \in X$.
Then
there is $n \in \N$ such that
for all $x \in X$ we have
$f_n (x) > g (x)$.
\end{lem}

\begin{proof}
For $n \in \N$ define
\[
U_n = \big\{ x \in X \colon f_n (x) - g (x) > 0 \big\}.
\]
Then $U_n$ is open because $f_n$ is lower semi\ct.
We have
\[
U_1 \subset U_2 \subset \cdots
\andeqn
\bigcup_{n = 1}^{\infty} U_n = X.
\]
Since $X$ is compact,
there is $n \in \N$ such that $U_n = X$.
\end{proof}

\begin{proof}[Proof of Proposition~\ref{P_4817_rCompCu}]
It is easy to see that if ${\operatorname{Cu}} (A)$ has $r$-comparison
then $W (A)$ has $r$-comparison.
So assume that $W (A)$ has $r$-comparison,
and let $a, b \in (K \otimes A)_{+}$
satisfy $d_{\ta} (a) + r < d_{\ta} (b)$
for all $\ta \in \QT (A)$.
Let $\ep > 0$.
We prove that $(a - \ep)_{+} \precsim_A b$.
By Lemma \ref{L:CzBasic}(\ref{L:CzBasic:LMinusEp}),
this suffices.

Define a \cfn{} $g \colon [0, \infty) \to [0, 1]$ by
\[
g (\ld) = \begin{cases}
   \ep^{-1} \ld & \hspace{3em} 0 \leq \ld \leq \ep
       \\
   1            & \hspace{3em} \ep \leq \ld.
\end{cases}
\]
For $\ta \in \QT (A)$ we have
$d_{\ta} ( (a - \ep)_{+} ) \leq \ta (g (a)) \leq d_{\ta} (a)$,
so $\ta (g (a)) + r < d_{\ta} (b)$.
Also,
$\ta \mapsto \ta (g (a))$ is \ct.
Define $f_n \colon \QT (A) \to [0, \infty]$
by
$f_n (\ta) = d_{\ta} \big( \big( b - \frac{1}{n} \big)_{+} \big)$
for $\ta \in \QT (A)$ and $n \in \N$.
Then for $\ta \in \QT (A)$
we have $f_1 (\ta) \leq f_2 (\ta) \leq \cdots$,
and it follows from Lemma~\ref{L_3905_CuWay}(\ref{L_3905_CuWay_SupEp})
and Theorem~\ref{T_2Y25_DTau}(\ref{T_2Y25_DTau_Cu})
that $\limi{n} f_n (\ta) = d_{\ta} (a)$.
Since $f_n$ is lower semi\ct{} for $n \in \N$
(by Lemma~\ref{L_4815_lsc}),
{}from Lemma~\ref{L_4817_Dini}
we get $n \in \N$
such that for all $\ta \in \QT (A)$
we have
$f_n (\ta) > \ta (g (a)) + r$,
whence
\[
d_{\ta} \big( \big( b - \tfrac{1}{n} \big)_{+} \big)
 \geq d_{\ta} ( (a - \ep)_{+} ) + r.
\]
Lemma~\ref{L-2720KToMn}
implies that $\big\langle \big( b - \frac{1}{n} \big)_{+} \big\rangle$
and $\langle (a - \ep)_{+} \rangle$ are in $W (A)$.
So the hypothesis gives the first step of the calculation
$(a - \ep)_{+} \precsim_A \big( b - \frac{1}{n} \big)_{+} \precsim_A b$.
\end{proof}

\begin{thm}\label{T_4814_RCEq}
Let $A$ be an infinite dimensional stably finite
simple separable unital \ca.
Let $B \subset A$ be large
in the sense of Definition \ref{D_Large},
and let ${\mathrm{rc}} (-)$ be as in Definition~\ref{D_4814_rc}.
Then ${\mathrm{rc}} (A) = {\mathrm{rc}} (B)$.
\end{thm}

\begin{proof}
The subalgebra $B$ is stably large by Corollary~\ref{C_4619_StFinStLg},
and $B$ is stably finite because it is a subalgebra of~$A$.

We must show that $W (A)$ has $r$-comparison
\ifo{} $W (B)$ has $r$-comparison.
By Proposition~\ref{P_4817_rCompCu},
it suffices to show that ${\operatorname{Cu}} (A)$ has $r$-comparison
\ifo{} ${\operatorname{Cu}} (B)$ has $r$-comparison.
The two directions are similar,
so we omit some details
of the proof that $r$-comparison for ${\operatorname{Cu}} (A)$
implies $r$-comparison for ${\operatorname{Cu}} (B)$.

Let $r \in [0, \infty)$,
suppose that
${\operatorname{Cu}} (B)$ has $r$-comparison,
and let $a, b \in (K \otimes A)_{+}$
satisfy $d_{\ta} (a) + r < d_{\ta} (b)$
for all $\ta \in \QT (A)$.
We must show that $a \precsim_A b$.
There are three cases,
the last of which will be done by reduction to previous cases.

Case~1:
Neither $\langle a \rangle$ nor $\langle b \rangle$
is the class of a \pj.
Use Theorem~\ref{T-2720CuSurj}
to find $x, y \in (K \otimes B)_{+}$
such that $x \sim_A a$ and $y \sim_A b$.
Applying Proposition~\ref{P-2724QTBij},
we get $d_{\ta} (x) + r < d_{\ta} (y)$
for all $\ta \in \QT (B)$.
Since ${\operatorname{Cu}} (B)$ has $r$-comparison,
it follows that $x \precsim_B y$.
Thus $a \precsim_A b$.

Case~2:
$\langle b \rangle$ is the class of a \pj{}
but $\langle a \rangle$ is not.
Theorem~\ref{C_3X08_IsoOnPure}
provides $x \in (K \otimes B)_{+}$
such that $x \sim_A a$
and $\langle x \rangle \in {\operatorname{Cu}} (B)$
is not the class of a \pj.
It is enough to prove that $x \precsim_A b$.
By Lemma \ref{L:CzBasic}(\ref{L:CzBasic:LMinusEp}),
it is enough to let $\ep > 0$
and prove that $(x - \ep)_{+} \precsim_A b$.

Choose a \cfn{} $f \colon [0, \infty) \to [0, 1]$
such that $f (\ld) > 0$ for $\ld \in (0, \ep)$
and $f (\ld) = 0$ for $\ld \in \{ 0 \} \cup [\ep, \infty)$.
Then $f (x) \neq 0$ by Lemma~\ref{L_2Y24_WhenP}.
Therefore
$\rh = \inf_{\ta \in \QT (A)} \ta (f (x))$
satisfies $\rh > 0$.
For $\ta \in \QT (A)$,
we have $d_{\ta} (f (x)) \geq \rh$,
so
\[
d_{\ta} \big( (x - \ep)_{+} \big) + r + \rh < d_{\ta} (b).
\]
Choose $n \in \N$ such that $\frac{1}{n} < \rh$.
Use Lemma~\ref{L_3905_InterpPj}
to find $y_0 \in (K \otimes A)_{+}$
such that $\langle y_0 \rangle \in {\operatorname{Cu}}_{+} (A)$,
and $\kp \in {\operatorname{Cu}}_{+} (A)$,
such that
\[
\langle y_0 \rangle
  \leq \langle b \rangle
  \leq \langle y_0 \rangle + \kp
\andeqn
n \kp \leq \langle 1 \rangle.
\]
For $\ta \in \QT (A)$,
we have $d_{\ta} (\kp) < \rh$.
Since $d_{\ta} (b) < \infty$,
we get
$d_{\ta} (y_0) > d_{\ta} (b) - \rh$,
so $d_{\ta} \big( (x - \ep)_{+} \big) + r < d_{\ta} (y_0)$.
Theorem~\ref{T-2720CuSurj}
gives $y \in (K \otimes B)_{+}$
such that $y \sim_A y_0$.
Applying Proposition~\ref{P-2724QTBij},
we get $d_{\ta} \big( (x - \ep)_{+} \big) + r < d_{\ta} (y)$
for all $\ta \in \QT (B)$.
Using  $r$-comparison for ${\operatorname{Cu}} (B)$ at the first step,
we get $(x - \ep)_{+} \precsim_B y \sim_A y_0 \precsim_A b$.

Case~3:
$\langle a \rangle$ is the class of a \pj.
We can clearly assume $\langle a \rangle \neq 0$.
Then $\ta \mapsto d_{\ta} (a)$ is \ct{} on $\QT (A)$.
So Lemma~\ref{L_4815_lsc} implies that
$\ta \mapsto d_{\ta} (b) - r - d_{\ta} (a)$
is lower semi\ct{} on $\QT (A)$.
Since this function is strictly positive and $\QT (A)$ is compact,
it follows that
$\rh = \inf_{\ta \in \QT (A)}
   \big( d_{\ta} (b) - r - d_{\ta} (a) \big)$
satisfies $\rh > 0$.
Choose $n \in \N$ such that $\frac{1}{n} < \rh$.
Use Lemma~\ref{L_3905_InterpPj}
to find $\mu, \kp \in {\operatorname{Cu}}_{+} (A)$,
such that
\[
\mu
  \leq \langle a \rangle
  \leq \mu + \kp
\andeqn
n \kp \leq \langle 1 \rangle.
\]
For $\ta \in \QT (A)$,
we have $d_{\ta} (\kp) < \rh$,
so
\[
d_{\ta} (\mu + \kp) + r
 < d_{\ta} (a) + \rh + r
 \leq d_{\ta} (b).
\]
Corollary~\ref{C_2Y24_PureAbs}
implies that $\mu + \kp \in {\operatorname{Cu}}_{+} (A)$.
Now, depending on whether or not
$\langle b \rangle$ is the class of a \pj,
Case~1 or Case~2
implies that $\mu + \kp \leq \langle b \rangle$
in ${\operatorname{Cu}} (A)$.
Since $\langle a \rangle \leq \mu + \kp$,
we get $a \precsim_A b$,
as desired.

This completes the proof that if
${\operatorname{Cu}} (B)$ has $r$-comparison,
then so does ${\operatorname{Cu}} (A)$.

Now suppose that ${\operatorname{Cu}} (A)$ has $r$-comparison.
Let $a, b \in (K \otimes B)_{+}$
satisfy $d_{\ta} (a) + r < d_{\ta} (b)$
for all $\ta \in \QT (B)$.
We use the same case division as above.

In Case~1,
we get $d_{\ta} (a) + r < d_{\ta} (b)$
for all $\ta \in \QT (A)$
by Proposition~\ref{P-2724QTBij}.
So $a \precsim_A b$ by hypothesis,
and $a \precsim_B b$ by Theorem~\ref{C_3X08_IsoOnPure}.

Case~2 requires an extra trick.
Let $\ep > 0$ as before.
Applying Lemma~\ref{L_2Y24_WhenP} to~$a$,
choose $\ep_0 \in (0, \ep)$
such that $\spec (a) \cap (\ep_0, \ep) \neq \varnothing$.
Choose \cfn{s} $f, g \colon [0, \infty) \to [0, 1]$
such that $f (\ld) > 0$ for $\ld \in (0, \ep_0)$
and $f (\ld) = 0$ for $\ld \in \{ 0 \} \cup [\ep_0, \infty)$,
and such that $g (\ld) > 0$ for $\ld \in (\ep_0, \ep)$
and $g (\ld) = 0$ for $\ld \in [0, \ep_0] \cup [\ep, \infty)$.
Then $f (a)$ and $g (a)$ are both nonzero.
Therefore
$\rh = \inf_{\ta \in \QT (B)} \ta (f (a))$
satisfies $\rh > 0$.
For $\ta \in \QT (B)$,
we have
\[
d_{\ta} \big( (a - \ep_0)_{+} \big) + r + \rh
 \leq d_{\ta} \big( (a - \ep_0)_{+} \big) + d_{\ta} (f (a)) + r
 \leq d_{\ta} (a) + r
 < d_{\ta} (b).
\]
Choose $n \in \N$ such that $\frac{1}{n} < \rh$.
Use Lemma~\ref{L_3905_InterpPj}
to find $y \in (K \otimes B)_{+}$
such that $\langle y \rangle \in {\operatorname{Cu}}_{+} (B)$,
and $\kp \in {\operatorname{Cu}}_{+} (B)$,
satisfying
\begin{equation}\label{Eq_4819q}
\langle y \rangle
  \leq \langle b \rangle
  \leq \langle y \rangle + \kp
\andeqn
n \kp \leq \langle 1 \rangle.
\end{equation}
Use Lemma~\ref{L-2817Sp01}
to choose a positive element $z \in {\overline{g (a) B g (a)}}$
such that $\spec (z) = [0, 1]$.
Then
\[
z \oplus (a - \ep)_{+} \precsim (a - \ep_0)_{+}
\]
by Lemma \ref{L:CzBasic}(\ref{L:CzBasic:Her})
and Lemma \ref{L:CzBasic}(\ref{L:CzBasic:Orth}).
For $\ta \in \QT (B)$,
we therefore get
\[
d_{\ta} \big( z \oplus (a - \ep)_{+} \big) + r + \rh
 \leq d_{\ta} \big( (a - \ep_0)_{+} \big) + r + \rh
 < d_{\ta} (b).
\]
Since $d_{\ta} (b) < \infty$,
and $d_{\ta} (\kp) \leq \frac{1}{n} < \rh$ by
the second part of~(\ref{Eq_4819q}),
the first part of~(\ref{Eq_4819q})
gives
\[
d_{\ta} \big( z \oplus (a - \ep)_{+} \big) + r
 < d_{\ta} (y).
\]
Proposition~\ref{P-2724QTBij} implies that this inequality holds
for all $\ta \in \QT (A)$.
So $z \oplus (a - \ep)_{+} \precsim_A y$ by hypothesis.
Now $\langle y \rangle \in {\operatorname{Cu}}_{+} (B)$
by construction,
and $z \oplus (a - \ep)_{+} \in {\operatorname{Cu}}_{+} (B)$
by Lemma~\ref{L_2Y24_WhenP}
and Corollary~\ref{C_2Y24_PureAbs},
so Theorem~\ref{C_3X08_IsoOnPure}
implies $z \oplus (a - \ep)_{+} \precsim_B y$.
Therefore $(a - \ep)_{+} \precsim_B y \precsim_B b$.
This completes the proof of Case~2.

Case~3 is the same as before,
except with $B$ in place of~$A$ everywhere.
\end{proof}

We now show that if $B$ is large in~$A$,
then $A$ is finite or purely infinite
\ifo{} $B$ has the same property.
We don't directly use Theorem~\ref{C_3X08_IsoOnPure},
because we don't assume that $B$ is stably large.

\begin{prp}\label{P_4626_Fin}
Let $A$ be an infinite dimensional simple unital \ca,
and let $B \subset A$ be a large subalgebra.
Then $A$ is finite \ifo{} $B$ is finite.
\end{prp}

We do not need $B$ to be stably large.

\begin{proof}[Proof of Proposition~\ref{P_4626_Fin}]
If $A$ is finite, then obviously $B$ is finite.
So assume $A$ is infinite; we prove that $B$ is infinite.
Choose $s \in A$ such that $s^* s = 1$
and $s s^* \neq 1$.
Set $q = s s^*$.
With the help of Lemma~\ref{L:DivInSmp}, find
$x_1, x_2 \in \big( (1 - q) A (1- q) \big)_{+}$
such that $x_1x_2= 0$ and $\| x_1 \| = \| x_2 \| = 1$.

Choose $\ep > 0$ such that $28 \ep < 1$.
Choose $\rh > 0$ as in Proposition~\ref{P_4619_ConeSj}
with $n = 2$ and $\dt = \ep$.
We also require $\rh \leq \ep$.
Apply Lemma~\ref{L-2720L1},
getting $d_1, d_2 \in B_{+}$
such that for $j = 1, 2$ we have
\[
\| d_j \| = 1,
\,\,\,\,\,\,
d_j \precsim_A x_j,
\andeqn
\| x_j d_j - d_j \| < \frac{\rh}{2}.
\]
Since $x_1 x_2 = 0$,
we get
\begin{align*}
\| d_1 d_2 \|
& = \| d_1 d_2 - d_1 x_1 x_2 d_2 \|
\\
& \leq \| d_1 - d_1 x_1 \| \cdot \| d_2 \|
        + \| d_1 x_1 \| \cdot \| d_2 - x_2 d_2 \|
  < \frac{\rh}{2} + \frac{\rh}{2}
  = \rh.
\end{align*}
By the choice of~$\rh$,
there are $c_1, c_2 \in B_{+}$
such that $c_1 c_2 = 0$
and for $j = 1, 2$ we have
$0 \leq c_j \leq 1$
and $\| c_j - d_j \| < \ep$.
In particular,
$\| c_j \| > 1 - \ep$.

Define \cfn{s} $f_0, f_1 \colon [0, 1] \to [0, 1]$
by
\[
f_0 (\ld)
 = \begin{cases}
   (1 - 2 \ep)^{-1} \ld & \hspace{3em} 0 \leq \ld \leq 1 - 2 \ep
        \\
   1                    & \hspace{3em} 1 - 2 \ep \leq \ld \leq 1
\end{cases}
\]
and
\[
f_1 (\ld)
 = \begin{cases}
   0                           & \hspace{3em} 0 \leq \ld \leq 1 - 2 \ep
        \\
 \ep^{-1} [ \ld - (1 - 2 \ep)]
         & \hspace{3em} 1 - 2 \ep \leq \ld \leq 1 - \ep
       \\
   1                           & \hspace{3em} 1 - \ep \leq \ld \leq 1.
\end{cases}
\]
For $j = 1, 2$ set
$z_j = f_0 (c_j)$ and $y_j = f_1 (c_j)$.
Then $\| c_j - z_j \| \leq 2 \ep$,
so $\| d_j - z_j \| < 3 \ep$.
Also,
$\| y_j \| = 1$ and $z_j y_j = y_j$.
Furthermore, $z_1 z_2 = 0$,
so $y_1 y_2 = z_1 y_2 = y_1 z_2 = 0$.

Define $y = 1 - z_1 - z_2$.
Then $y y_1 = y y_2 = 0$.
We have
\[
\| x_j z_j - z_j \|
 \leq \| x_j \| \cdot \| z_j - d_j \|
        + \| x_j d_j - d_j \| + \| d_j - z_j \|
 < 3 \ep + \frac{\rh}{2} + 3 \ep
 < 7 \ep.
\]
Since $q x_j = 0$,
we therefore get
$\| q z_j \| = \| q z_j - q x_j z_j \| < 7 \ep$.
So
\[
\| q y - q \| \leq \| q z_1 \| +\| q z_2 \| < 14 \ep
\andeqn
\| y q y - q \| < 28 \ep.
\]
Now use the definition of $q$ at the first step,
$28 \ep < 1$ at the second step,
Lemma~\ref{L:CzBasic}(\ref{L:CzBasic:LCzWithinEp})
at the third step,
and Lemma~\ref{L:CzBasic}(\ref{L:CzBasic:LCzFCalc})
at the fifth step,
getting
\[
1 \sim_A q
  \sim_A (q - 28 \ep)_{+}
  \precsim_A y q y
  \leq y^2
  \sim_A y.
\]
Apply Lemma~\ref{L-2721AToB_Pre}
with $a = 1$,
with $b = y + y_1$,
with $c = y$,
and with $x = y_1$.
We get $1 \precsim_B y + y_1$.
Thus,
there is $v \in B$
such that $\| v (y + y_1) v^* - 1 \| < \tfrac{1}{2}$.
So $v (y + y_1)^{1/2}$ has a right inverse.
But $v (y + y_1)^{1/2} y_2 = 0$,
whence $v (y + y_1)^{1/2}$ is not invertible.
Thus $B$ is infinite.
\end{proof}

\begin{cor}\label{C_4626_StFin}
Let $A$ be an infinite dimensional simple unital \ca,
and let $B \subset A$ be a stably large subalgebra.
Then $A$ is stably finite \ifo{} $B$ is stably finite.
\end{cor}

\begin{proof}
The result is immediate from Proposition~\ref{P_4626_Fin}.
\end{proof}

\begin{prp}\label{P-2724PI}
Let $A$ be an infinite dimensional simple unital \ca,
and let $B \subset A$ be a large subalgebra.
Then $A$ is purely infinite \ifo{} $B$ is purely infinite.
\end{prp}

Again,
we do not need to assume that $B$ is stably large in~$A$.
Combining this result with Proposition~\ref{P_4626_Fin},
we can deduce that
if $B$ is large in~$A$,
then $B$ is infinite but not purely infinite \ifo{} $A$ is.
Also,
if $B$ is large in~$A$ and $A$ is stably finite,
then $B$ is stably finite because it is a subalgebra of~$A$.
But, for now, we need $B$ to be stably large in~$A$
to deduce that if $A$ is finite but not stably finite,
then the same is true of~$B$.

\begin{proof}[Proof of Proposition~\ref{P-2724PI}]
Assume first that $B$ is purely infinite.
Let $a \in A_{+} \setminus \{ 0 \}$.
We must show that ${\overline{a A a}}$ contains a \pj{}
which is infinite in~$A$.
\Wolog{} $\| a \| = 1$.

Choose $\ep \in \big( 0, \tfrac{1}{8} \big)$
and so small that whenever $D$ is a \ca{} and $x \in D_{+}$
satisfies $\| x^2 - x \| < 12 \ep$,
then there is a \pj{} $q \in D$ such that $\| q - x \| < \tfrac{1}{2}$.
Lemma~\ref{L-2720L1} provides $b \in B_{+}$
such that $\| b \| = 1$ and $\| a b - b \| < \ep$.
Define \cfn{s} $f_0, f_1 \colon [0, \infty) \to [0, 1]$
by
\[
f_0 (\ld)
 = \begin{cases}
   (1 - \ep)^{-1} \ld & \hspace{3em} 0 \leq \ld \leq 1 - \ep
        \\
   1                  & \hspace{3em} 1 - \ep \leq \ld
\end{cases}
\]
and
\[
f_1 (\ld)
 = \begin{cases}
   0                           & \hspace{3em} 0 \leq \ld \leq 1 - \ep
        \\
   \ep^{-1} [ \ld - (1 - \ep)] & \hspace{3em} 1 - \ep \leq \ld \leq 1
       \\
   1                           & \hspace{3em} 1 \leq \ld.
\end{cases}
\]
Since $f_1 (b) \neq 0$ and  $B$ is purely infinite,
there is an infinite \pj{}
$p \in {\overline{f_1 (b) B f_1 (B)}}$.
Then $f_0 (b) p = p$.
Since $\| b - f_0 (b) \| \leq \ep$,
we get
$\| b p b - p \| \leq 2 \ep$,
so $\| a b p b a - p \| < 4 \ep$,
and thus $\| (a b p b a)^2 - a b p b a \| < 12 \ep$.
Therefore there is a \pj{} $q \in {\overline{a A a}}$
such that $\| q - a b p b a \| < \tfrac{1}{2}$.
Then
\[
\| q - p \|
 \leq \| q - a b p b a \| + \| a b p b a - p \|
 < \tfrac{1}{2} + 4 \ep
 < 1.
\]
Thus $q$ is \mvnt{} to~$p$ by Proposition 4.6.6 of~\cite{Bl3},
and is hence also infinite.

Now assume that $A$ is purely infinite.
We will prove that if $a, b \in B_{+} \setminus \{ 0 \}$,
than $a \precsim_B b$.
This shows that $B$ is purely infinite
in the sense of Definition~4.1 of~\cite{KR},
and pure infiniteness in the usual sense
now follows from Proposition~5.4 of~\cite{KR}.

Let $(e_{j, k})_{j, k \in \{ 1, 2 \}}$ be the standard system
of matrix units for~$M_2$.
Since $M_2 \otimes A$ is purely infinite,
there are a \nzp{} $p \in {\overline{b A b}}$
and $s \in M_2 \otimes A$ such that $s^* s = 1 \otimes 1$
and $s s^* = e_{1, 1} \otimes p$.
Then there are \nzp{s} $c, x \in p A p$
such that
\[
s (e_{1, 1} \otimes 1) s^*
 = \left( \begin{matrix}
  c     &  0        \\
  0     &  0
\end{matrix} \right)
\andeqn
s (e_{2, 2} \otimes 1) s^*
 = \left( \begin{matrix}
  x     &  0        \\
  0     &  0
\end{matrix} \right).
\]
We want to apply Lemma~\ref{L-2721AToB_Pre}
with $a$, $b$, $c$, and $x$ as given.
We have $a \precsim_A c$
since $A$ is purely infinite.
(See Theorem~2.2 of~\cite{LZ},
in particular condition~(vi).)
The remaining hypotheses of Lemma~\ref{L-2721AToB_Pre}
are easily checked.
So $a \precsim_B b$.
\end{proof}

\section{The orbit breaking subalgebra for an infinite set meeting each
  orbit at most once}\label{Sec_VK}

\indent
In this section,
we let $h \colon X \to X$ be a \hme{}
of a \chs~$X$.
Following Putnam~\cite{Pt1}, for $Y \subset X$ closed we define
the $Y$-orbit breaking subalgebra
$C^* (\Z, X, h)_Y \subset C^* (\Z, X, h)$.
We prove that if $X$ is infinite,
$h$~is minimal,
and $Y$ intersects each orbit at most once,
then $C^* (\Z, X, h)_Y$ is a large subalgebra of $C^* (\Z, X, h)$
of crossed product type,
in the sense of Definition~\ref{D-2717CPType}.

\begin{ntn}\label{N-Old32}
Let $G$ be a discrete group,
let $A$ be a \ca,
and let $\af \colon G \to \Aut (A)$ be an action of $G$ on~$A$.
We identify $A$ with a subalgebra of $C^*_{\mathrm{r}} (G, A, \af)$
in the standard way.
We let $u_g \in M (C^*_{\mathrm{r}} (G, A, \af))$
be the standard unitary
corresponding to $g \in G$.
When $G = \Z$, we write just $u$ for the unitary~$u_1$
corresponding to the generator $1 \in \Z$.
We let $A [G]$ denote the dense *-subalgebra
of $C^*_{\mathrm{r}} (G, A, \af)$
consisting of sums $\sum_{g \in S} a_g u_g$
with $S \subset G$ finite and $a_g \in A$ for $g \in S$.
We may always assume $1 \in S$.
We let $E_{\af} \colon C^*_{\mathrm{r}} (G, A, \af) \to A$
denote the standard conditional expectation,
defined on $A [G]$ by
$E_{\af} \left( \sum_{g \in S} a_g u_g \right) = a_1$.
When $\af$ is understood, we just write~$E$.

When $G$ acts on a \chs~$X$,
we use obvious analogs of this notation for $C^*_{\mathrm{r}} (G, X)$,
with the action of $G$ on $C (X)$
being given by
$\af_g (f) (x) = f (g^{-1} x)$ for $f \in C (X)$,
$g \in G$, and $x \in X$.
For a \hme{} $h \colon X \to X$,
this means that the action is generated by
the automorphism $\af (f) = f \circ h^{-1}$ for $f \in C_0 (X)$.
In particular,
we have $u f u^* = f \circ h^{-1}$.
\end{ntn}

\begin{ntn}\label{N-2818C0U}
For a locally \chs~$X$
and an open subset $U \subset X$,
we use the abbreviation
\[
C_0 (U)
 = \big\{ f \in C_0 (X) \colon
     {\mbox{$f (x) = 0$ for all $x \in X \setminus U$}} \big\}
 \subset C_0 (X).
\]
This subalgebra is of course canonically isomorphic to
the usual algebra $C_0 (U)$ when $U$ is considered
as a locally \chs{} in its own right.
\end{ntn}

In particular,
if $Y \subset X$ is closed, then
\begin{equation}\label{Eq:2818C0XY}
C_0 (X \setminus Y)
 = \big\{ f \in C_0 (X) \colon
  {\mbox{$f (x) = 0$ for all $x \in Y$}} \big\}.
\end{equation}

\begin{dfn}\label{D-2623VSubalg}
Let $X$ be a locally \chs{} and let $h \colon X \to X$ be a \hme.
Let $Y \subset X$ be a nonempty closed subset,
and, following~(\ref{Eq:2818C0XY}), define
\[
C^* (\Z, X, h)_Y = C^* \big( C (X), \, C_0 (X \setminus Y) u \big)
  \subset C^* (\Z, X, h).
\]
We call it the {\emph{$Y$-orbit breaking subalgebra}}
of $C^* (\Z, X, h)$.
\end{dfn}

The idea of using subalgebras of this type is due to Putnam~\cite{Pt1}.

We have used a different convention from that used elsewhere,
where one usually takes
\begin{equation}\label{Eq:OldOB}
C^* (\Z, X, h)_Y = C^* \big( C (X), \, u C_0 (X \setminus Y) \big).
\end{equation}
The choice of convention in Definition~\ref{D-2623VSubalg}
has the advantage that,
when used in connection with Rokhlin towers,
the bases of the towers are subsets of~$Y$
rather than of $h (Y)$.

Orbit breaking subalgebras
(without the name,
and using the convention~(\ref{Eq:OldOB})),
have a long history.
For example:
\begin{itemize}
\item
The version with $Y$ taken to consist of one point
has been used many places.
It was introduced when $X$ is the Cantor set by Putnam~\cite{Pt1},
along with the version in which $Y$ is a nonempty compact
open set.
An early application of the one point version
when $X$ is not the Cantor set
is in~\cite{LqP} and Section~4 of~\cite{Ph7}.
\item
The one point version plays a key role in~\cite{LhP}.
\item
The version with two points on different orbits
has been used by Toms and Winter~\cite{TW}.
\item
Let $X$ be the Cantor set
and let $h \colon X \times S^1 \to X \times S^1$
be a minimal \hme.
For any $x \in X$, the set $Y = \{ x \} \times S^1$
intersects each orbit at most once.
The algebra $C^* (\Z, \, X \times S^1, \, h)_Y$
is introduced before Proposition~3.3 of~\cite{LM1},
where it is called~$A_x$.
\item
A similar construction,
with $X \times S^1 \times S^1$ in place of $X \times S^1$
and with $Y = \{ x \} \times S^1 \times S^1$,
appears in Section~1 of~\cite{Sn}.
\item
A six term exact sequence
for the K-theory of some orbit breaking subalgebras
is given in Example~2.6 of~\cite{Pt3}.
\item
Orbit breaking subalgebras of irrational rotation algebras
are among the examples studied in their own right in~\cite{FJLX},
and certain orbit breaking subalgebras of some
higher dimensional noncommutative tori
are among the examples studied in~\cite{Sn2}.
\item
The algebras $C^* (\Z, X, h)_Z$,
for $Z \subset X$ closed and with nonempty interior,
are used to obtain information about the orbit breaking subalgebras
mentioned above.
For every nonempty~$Y$,
the algebra $C^* (\Z, X, h)_Y$
is a direct limit of algebras $C^* (\Z, X, h)_Z$
for $Z \subset X$ with $\sint (Z) \neq \varnothing$,
and $\sint (Z) \neq \varnothing$ implies that $C^* (\Z, X, h)_Z$
is a recursive subhomogeneous algebra
in the sense of Definition~1.1 of~\cite{Ph6}.
\end{itemize}

We show that if $Y$ intersects each orbit of~$h$ at most once,
then $C^* (\Z, X, h)_Y$ is a large subalgebra of $C^* (\Z, X, h)$
of crossed product type.

\begin{lem}\label{L-2623ZPart}
Let $X$ be a \chs{}
and let $h \colon X \to X$ be a minimal homeomorphism.
Let $K \subset X$ be a compact set
such that $h^n (K) \cap K = \varnothing$
for all $n \in \Z \setminus \{ 0 \}$.
Let $U \subset X$ be a nonempty open subset.
Then there exist $l \in \Nz$, compact sets
$K_1, K_2, \ldots, K_l \subset X$,
and $n_1, n_2, \ldots, n_l \in \N$,
such that
$K \subset \bigcup_{j = 1}^l K_j$
and such that
$h^{n_1} (K_1), \, h^{n_2} (K_2), \, \ldots, \, h^{n_l} (K_l)$
are disjoint subsets of~$U$.
\end{lem}

\begin{proof}
Choose a nonempty open subset $V \subset X$
such that ${\overline{V}}$ is compact and contained in~$U$.
Minimality of the action implies that
$\bigcup_{n = 1}^{\infty} h^{- n} (V) = X$.
Therefore there are distinct $n_1, n_2, \ldots, n_l \in \N$
such that
$K \subset \bigcup_{j = 1}^l h^{- n_j} (V)$.
For $j = 1, 2, \ldots, l$,
define $K_j = h^{- n_j} \big( {\overline{V}} \big) \cap K$,
which is a compact subset of~$X$.
Clearly $K \subset \bigcup_{j = 1}^l K_j$.
For $j = 1, 2, \ldots, l$,
we have $h^{n_j} (K_j) \subset {\overline{V}} \subset U$.
Finally,
for distinct $i, j \in \{ 1, 2, \ldots, l \}$,
we have
\[
h^{n_i} (K_i) \cap h^{n_j} (K_j)
  \subset h^{n_i} \big( K \cap h^{n_j - n_i} (K_j) \big)
  = \varnothing.
\]
This completes the proof.
\end{proof}

\begin{prp}\label{P-2816CharOB}
Let $X$ be a \chs{} and let $h \colon X \to X$ be a \hme.
Let $u \in C^* (\Z, X, h)$
and $E \colon C^* (\Z, X, h) \to C (X)$
be as in Notation~\ref{N-Old32}.
Let $Y \subset X$ be a nonempty closed subset.
For $n \in \Z$, set
\[
Y_n = \begin{cases}
   \bigcup_{j = 0}^{n - 1} h^j (Y)    & \hspace{3em} n > 0
        \\
   \varnothing                        & \hspace{3em} n = 0
       \\
   \bigcup_{j = 1}^{- n} h^{-j} (Y)   & \hspace{3em} n < 0.
\end{cases}
\]
Then
\begin{equation}\label{Eq:2816CharOB}
C^* (\Z, X, h)_Y
 = \big\{ a \in C^* (\Z, X, h) \colon
    {\mbox{$E (a u^{-n}) \in C_0 (X \setminus Y_n)$
          for all $n \in \Z$}} \big\}
\end{equation}
and
\begin{equation}\label{Eq:2816Dense}
{\overline{C^* (\Z, X, h)_Y \cap C (X) [\Z]}} = C^* (\Z, X, h)_Y.
\end{equation}
\end{prp}

\begin{proof}
Define
\[
B = \big\{ a \in C^* (\Z, X, h) \colon
    {\mbox{$E (a u^{-n}) \in C_0 (X \setminus Y_n)$
          for all $n \in \Z$}} \big\}
\]
and
\[
B_0 = B \cap C (X) [\Z].
\]

We claim that $B_0$ is dense in~$B$.
To see this,
let $b \in B$ and for $k \in \Z$ define
$b_k = E (b u^{-k}) \in C_0 (X \setminus Y_k)$.
Then for $n \in \N$, the element
\[
a_n = \sum_{k = - n + 1}^{n - 1}
     \left( 1 - \frac{|k|}{n} \right) b_k u^k.
\]
is clearly in~$B_0$,
and Theorem VIII.2.2 of~\cite{Dv}
implies that $\limi{n} a_n = b$.
The claim follows.
In particular,
(\ref{Eq:2816Dense})~will now follow from~(\ref{Eq:2816CharOB}),
so we need only prove~(\ref{Eq:2816CharOB}).

For $0 \leq m \leq n$ and $0 \geq m \geq n$,
we clearly have $Y_m \subset Y_n$.

We claim that for all $n \in \Z$,
we have
\begin{equation}\label{Eq:2818hnyn}
h^{-n} (Y_n) = Y_{-n}.
\end{equation}
The case $n = 0$ is trivial,
the case $n > 0$ is easy,
and the case $n < 0$ follows from the case $n > 0$.

We next claim that for all $m, n \in \Z$,
we have
\[
Y_{m + n} \subset Y_m \cup h^m (Y_n).
\]
The case $m = 0$ or $n = 0$ is trivial.
For $m, n > 0$ and also for $m, n < 0$, it is easy to check that
$Y_{m + n} = Y_m \cup h^m (Y_n)$.

Now suppose $m > 0$ and $- m \leq n < 0$.
Then $0 \leq m + n \leq m$,
so
\[
Y_{m + n} \subset Y_m \subset Y_m \cup h^m (Y_n).
\]
If $m > 0$ and $n < - m$,
then $m + n < 0$,
so
\[
Y_{m + n}
  = \bigcup_{j = m + n}^{- 1} h^{j} (Y)
  \subset \bigcup_{j = m + n}^{m - 1} h^{j} (Y)
  = \bigcup_{j = 0}^{m - 1} h^j (Y)
        \cup \bigcup_{j = m + n}^{m - 1} h^{j} (Y)
  = Y_m \cup h^m (Y_n).
\]

Finally, suppose $m < 0$ and $n > 0$.
Then, using~(\ref{Eq:2818hnyn}) at the first and third steps,
and the already done case $m > 0$ and $n < 0$ at the second step,
we get
\[
Y_{m + n}
  = h^{m + n} (Y_{- m - n})
  \subset h^{m + n} \big( Y_{- m} \cup h^{- m} (Y_{- n}) \big)
  = h^n (Y_m) \cup Y_n.
\]
This completes the proof of the claim.

We now claim that $B_0$ is a *-algebra.
It is enough to prove that if $f \in C_0 (X \setminus Y_m)$
and $g \in C_0 (X \setminus Y_n)$,
then $(f u^m) (g u^n) \in B_0$
and $(f u^m)^* \in B_0$.
For the first,
we have $(f u^m) (g u^n) = f \cdot (g \circ h^{- m}) \cdot u^{m + n}$.
Now $f \cdot (g \circ h^{- m})$ vanishes on $Y_m \cup h^m (Y_n)$,
so the previous claim implies that
$f \cdot (g \circ h^{- m}) \in C_0 (X \setminus Y_{m + n})$.
Also,
\[
(f u^m)^* = u^{- m} {\overline{f}}
        = \big( {\overline{f \circ h^m}} \big) u^{- m},
\]
and, using~(\ref{Eq:2818hnyn}),
the function $f \circ h^m$ vanishes on $h^{- m} (Y_m) = Y_{- m}$,
so $(f u^m)^* \in B_0$.
This proves the claim.

Since $C (X) \subset B_0$ and $C_0 (X \setminus Y) u \subset B_0$,
it follows that $C^* (\Z, X, h)_Y \subset {\overline{B_0}} = B$.

We next claim that for all $n \in \Z$,
we have
$C_0 (X \setminus Y_n) \subset C^* (\Z, X, h)_Y$.
For $n = 0$ this is trivial.
Let $n > 0$,
and let $f \in C_0 (X \setminus Y_n)$.
Define $f_0 = (\sgn \circ f) | f |^{1/n}$
and for $j = 1, 2, \ldots, n - 1$
define
$f_j = | f \circ h^{j} |^{1/n}$.
The definition of $Y_n$
implies that $f_0, f_1, \ldots, f_{n - 1} \in C_0 (X \setminus Y)$.
Therefore the element
\[
a = (f_0 u) (f_1 u) \cdots (f_{n - 1} u)
\]
is in $C^* (\Z, X, h)_Y$.
Moreover,
we can write
\begin{align*}
a
& = f_0 (u f_1 u^{-1}) (u^2 f_2 u^{-2})
    \cdots \big( u^{n - 1} f_{n - 1} u^{- (n - 1)} \big) u^n
  \\
& = f_0 (f_1 \circ h^{-1}) (f_2 \circ h^{-2})
    \cdots \big( f_{n - 1} \circ h^{- (n - 1)} \big) u^n
  = (\sgn \circ f) \big( | f |^{1/n} \big)^n u^n
  = f u^n.
\end{align*}
Finally,
suppose $n < 0$,
and let $f \in C_0 (X \setminus Y_n)$.
It follows from~(\ref{Eq:2818hnyn})
that $f \circ h^n \in C_0 (X \setminus Y_{- n})$,
whence also ${\overline{f \circ h^n}} \in C_0 (X \setminus Y_{- n})$.
Since $- n > 0$,
we therefore get
\[
f u^n = \big( u^{-n} {\overline{f}} \big)^*
      = \big( \big( {\overline{f \circ h^n}} \big) u^{- n} \big)^*
      \in C^* (\Z, X, h)_Y.
\]
The claim is proved.

It now follows that $B_0 \subset C^* (\Z, X, h)_Y$.
Combining this result with ${\overline{B_0}} = B$
and $C^* (\Z, X, h)_Y \subset B$,
we get $C^* (\Z, X, h)_Y = B$.
\end{proof}

\begin{cor}\label{C-2816AYOpp}
Let $X$ be a \chs{} and let $h \colon X \to X$ be a \hme.
Let $Y \subset X$ be a nonempty closed subset.
Let $u \in C^* (\Z, X, h)$ be the standard unitary,
as in Notation~\ref{N-Old32},
and let $v \in C^* (\Z, X, h^{-1})$ be the analogous standard unitary
in $C^* (\Z, X, h^{-1})$.
Then there exists
a unique \hm{} $\ph \colon C^* (\Z, X, h^{-1}) \to C^* (\Z, X, h)$
such that $\ph (f) = f$ for $f \in C (X)$ and $\ph (v) = u^*$,
the map $\ph$ is an isomorphism,
and
\[
\ph \big( C^* (\Z, X, h^{-1})_{h^{-1} (Y)} \big) = C^* (\Z, X, h)_Y.
\]
\end{cor}

\begin{proof}
Existence and uniqueness of~$\ph$,
as well as the fact that $\ph$ is an isomorphism,
are all immediate from standard results about crossed products.

Set $Z = h^{-1} (Y)$.
For $n \in \Z$,
let $Y_n \subset X$ be as in the statement
of Proposition~\ref{P-2816CharOB},
and let $Z_n \subset X$ be the set analogous to $Y_n$
but using $Z$ in place of~$Y$ and $h^{-1}$ in place of~$h$.
Since $\ph (f v^n) = f u^{-n}$
for all $f \in C (X)$ and $n \in \Z$,
by Proposition~\ref{P-2816CharOB}
the formula
$\ph \big( C^* (\Z, X, h^{-1})_{h^{-1} (Y)} \big) = C^* (\Z, X, h)_Y$
is equivalent to $Y_n = Z_{- n}$ for all $n \in \Z$.
This equality is immediate from the definitions.
\end{proof}

\begin{lem}\label{L-2623CuSub}
Let $X$ be a \chs{}
and let $h \colon X \to X$ be a minimal homeomorphism.
Let $Y \subset X$ be a compact subset
such that $h^n (Y) \cap Y = \varnothing$
for all $n \in \Z \setminus \{ 0 \}$.
Let $U \subset X$ be a nonempty open subset
and let $n \in \Z$.
Then there exist $f, g \in C (X)_{+}$
such that
\[
f |_{h^n (Y)} = 1,
\,\,\,\,\,\,
0 \leq f \leq 1,
\,\,\,\,\,\,
\supp (g) \subset U,
\andeqn
f \precsim_{C^* (\Z, X, h)_Y} g.
\]
\end{lem}

\begin{proof}
We first prove this when $n = 0$.

Apply Lemma~\ref{L-2623ZPart} with $Y$ in place of~$K$,
obtaining $l \in \Nz$, compact sets
$Y_1, Y_2, \ldots, Y_l \subset X$,
and $n_1, n_2, \ldots, n_l \in \N$.
Set $N = \max (n_1, n_2, \ldots, n_l)$.
Choose disjoint open sets $V_1, V_2, \ldots, V_l \subset U$
such that $h^{n_j} (Y_j) \subset V_j$ for $j = 1, 2, \ldots, l$.
Then $Y_j \subset h^{- n_j} (V_j)$,
so the sets
$h^{- n_1} (V_1), \, h^{- n_2} (V_2), \, \ldots, \, h^{- n_l} (V_l)$
cover~$Y$.
For $j = 1, 2, \ldots, l$,
define
\[
W_j
 = h^{- n_j} (V_j)
   \cap \left( X \setminus \bigcup_{n = 1}^N h^{-n} (Y) \right).
\]
Then $W_1, W_2, \ldots, W_l$ form an open cover of~$Y$.
Therefore there are $f_1, f_2, \ldots, f_l \in C (X)_{+}$
such that for $j = 1, 2, \ldots, l$
we have $\supp (f_j) \subset W_j$ and $0 \leq f_j \leq 1$,
and such that the function $f = \sum_{j = 1}^l f_j$
satisfies $f (x) = 1$ for all $x \in Y$ and $0 \leq f \leq 1$.
Further define $g = \sum_{j = 1}^l f_j \circ h^{- n_j}$.
Then $\supp (g) \subset U$.

Let $u \in C^* (\Z, X, h)$
be as in Notation~\ref{N-Old32}.
For $j = 1, 2, \ldots, l$,
set $a_j = f_j^{1/2} u^{-n_j}$.
Since $f_j$ vanishes on $\bigcup_{n = 1}^{n_j} h^{-n} (Y)$,
Proposition~\ref{P-2816CharOB}
implies that $a_j \in C^* (\Z, X, h)_Y$.
Therefore, in $C^* (\Z, X, h)_Y$ we have
\[
f_j \circ h^{- n_j} = a_j^* a_j \sim_{C^* (\Z, X, h)_Y} a_j a_j^* = f_j.
\]
Consequently,
using Lemma~\ref{L:CzBasic}(\ref{L:CzBasic:LCzCmpSum})
at the second step
and Lemma~\ref{L:CzBasic}(\ref{L:CzBasic:Orth})
and disjointness of the supports of the
functions $f_j \circ h^{- n_j}$ at the last step,
we have
\[
f = \sum_{j = 1}^l f_j
  \precsim_{C^* (\Z, X, h)_Y} \bigoplus_{j = 1}^l f_j
  \sim_{C^* (\Z, X, h)_Y} \bigoplus_{j = 1}^l f_j \circ h^{- n_j}
  \sim_{C^* (\Z, X, h)_Y} g.
\]
This completes the proof for $n = 0$.

Now suppose that $n > 0$.
Choose functions $f$ and $g$ for the case $n = 0$,
and call them $f_0$ and~$g$.
Since $f_0 (x) = 1$ for all $x \in Y$,
and since $Y \cap \bigcup_{l = 1}^n h^{- l} (Y) = \varnothing$,
there is $f_1 \in C (X)$ with $0 \leq f_1 \leq f_0$,
$f_1 (x) = 1$ for all $x \in Y$,
and $f_1 (x) = 0$
for $x \in \bigcup_{l = 1}^n h^{- l} (Y)$.
Set $v = f_1^{1/2} u^{- n}$
and $f = f_1 \circ h^{-n}$.
Then $f (x) = 1$ for all $x \in h^n (Y)$ and $0 \leq f \leq 1$.
Proposition~\ref{P-2816CharOB} implies that $v \in C^* (\Z, X, h)_Y$.
We have
\[
v^* v = u^n f_1 u^{- n} = f_1 \circ h^{- n} = f
\andeqn
v v^* = f_1.
\]
Using Lemma~\ref{L:CzBasic}(\ref{L:CzBasic:LCzComm}),
we thus get
\[
f \sim_{C^* (\Z, X, h)_Y} f_1
  \leq f_0
  \precsim_{C^* (\Z, X, h)_Y} g.
\]
This completes the proof for the case $n > 0$.

Finally,
we consider the case $n < 0$.
In this case, we have $- n - 1 \geq 0$.
Apply the cases already done
with $h^{-1}$ in place of~$h$.
We get $f, g \in C^* (\Z, X, h^{-1})_{h^{-1} (Y)}$
such that $f (x) = 1$
for all $x \in (h^{-1})^{- n - 1} (h^{-1} (Y)) = h^n (Y)$,
such that $0 \leq f \leq 1$,
such that $\supp (g) \subset U$,
and such that $f \precsim_{C^* (\Z, X, h^{-1})_{h^{-1} (Y)}} g$.
Let $\ph \colon C^* (\Z, X, h^{-1}) \to C^* (\Z, X, h)$
be the isomorphism of Corollary~\ref{C-2816AYOpp}.
Then
\[
\ph (f) = f,
\,\,\,\,\,\,
\ph (g) = g,
\andeqn
\ph \big( C^* (\Z, X, h^{-1})_{h^{-1} (Y)} \big) = C^* (\Z, X, h)_Y.
\]
Therefore $f \precsim_{C^* (\Z, X, h)_Y} g$.
\end{proof}

\begin{lem}\label{L:FSubEq}
Let $G$ be a discrete group,
let $X$ be a compact space,
and suppose $G$ acts on $X$ in such a way that
for every finite set $S \subset G$,
the set
\[
\{ x \in X \colon {\mbox{$g x \neq x$ for all $g \in S$}} \}
\]
is dense in~$X$.
Following Notation~\ref{N-Old32},
let $a \in C (X) [G] \subset C^*_{\mathrm{r}} (G, X)$ and let $\ep > 0$.
Then there exists $f \in C (X)$
such that
\[
0 \leq f \leq 1,
\,\,\,\,\,\,
f a^* a f \in C (X),
\andeqn
\| f a^* a f \| \geq \| E_{\af} (a^* a) \| - \ep.
\]
\end{lem}

\begin{proof}
Set $b = a^* a$.
If $E_{\af} (b) \leq \ep$, we can take $f = 0$.
So assume $E_{\af} (b) > \ep$.
Then there are a finite set $T \subset G$
and $b_g \in C (X)$ for $g \in T$
such that $b = \sum_{g \in T} b_g u_g$.
Necessarily $1 \in T$
and $b_1 = E_{\af} (b)$ is a nonzero positive element.
Define
\[
U = \big\{ x \in X \colon
   b_1 (x) > \| E (a^* a) \| - \ep \big\},
\]
which is a nonempty open subset of~$X$.
Since
\[
V = \{ x \in X \colon {\mbox{$g x \neq x$ for all $g \in T$}} \}
\]
is dense in~$X$,
we have $U \cap V \neq \varnothing$,
and there is a nonempty open set $W \subset U \cap V$
such that the sets $g W$, for $g \in T$, are pairwise disjoint.
Fix $x_0 \in W$.
Let $f \in C (X)$ satisfy
\[
0 \leq f \leq 1,
\,\,\,\,\,\,
\supp (f) \subset W,
\andeqn
f (x_0) = 1.
\]
Let $\af \colon G \to \Aut (C (X))$ be as in Notation~\ref{N-Old32}.
Then
\[
f b f
  = f b_1 f + \sum_{g \in T \setminus \{ 1 \} } f b_g u_g f
  = f b_1 f + \sum_{g \in T \setminus \{ 1 \} } f b_g \af_g (f) u_g.
\]
For $g \in T \setminus \{ 1 \}$ we have $\supp (f) \subset W$ and
$\supp (\af_g (f)) \subset g W$,
so $f b_g \af_g (f) = b_g f \af_g (f) = 0$.
Thus $f b f = f b_1 f \in C (X)$,
and
\[
\| f b_1 f \|
 \geq f (x_0) b_1 (x_0) f (x_0)
 = b_1 (x_0)
 > \| E_{\af} (a^* a) \| - \ep.
\]
This completes the proof.
\end{proof}

\begin{lem}\label{L:CXSubEq}
Let $G$ be a discrete group,
let $X$ be a compact space,
and suppose $G$ acts on~$X$ in such a way that
for every finite set $S \subset G$,
the set
\[
\{ x \in X \colon {\mbox{$g x \neq x$ for all $g \in S$}} \}
\]
is dense in~$X$.
Let $B \subset C^*_{\mathrm{r}} (G, X)$ be a unital
subalgebra such that,
following Notation~\ref{N-Old32}:
\begin{enumerate}
\item\label{L:CXSubEq:CXI}
$C (X) \subset B$.
\item\label{L:CXSubEq:FSD}
$B \cap C (X)[G]$ is dense in~$B$.
\end{enumerate}
Let $a \in B_{+} \setminus \{ 0 \}$.
Then there exists $b \in C (X)_{+} \setminus \{ 0 \}$
such that $b \precsim_{B} a$.
\end{lem}

\begin{proof}
We continue to follow Notation~\ref{N-Old32}.
\Wolog{} $\| a \| \leq 1$.
The conditional expectation
$E_{\af} \colon C^*_{\mathrm{r}} (G, X) \to C (X)$
is faithful.
Therefore $E_{\af} (a) \in C (X)$
is a nonzero positive element.
Set $\ep = \tfrac{1}{6} \| E_{\af} (a) \|$.
Choose $c \in B \cap C (X) [G]$ such that
$\| c - a^{1/2} \| < \ep$ and $\| c \| \leq 1$.
Then
\[
\| c c^* - a \| < 2 \ep
\andeqn
\| c^* c - a \| < 2 \ep.
\]
Apply Lemma~\ref{L:FSubEq} with $c$ in place of~$a$
and with $\ep$ as given,
obtaining $f \in C (X)$ as there.
We have
\[
\| f c^* c f \|
 > \| E_{\af} (c^* c) \| - \ep
 > \| E_{\af} (a) \| - 3 \ep
 = 3 \ep.
\]
Therefore
$(f c^* c f - 2 \ep)_{+}$ is a nonzero element of $C (X)$.
Using Lemma~\ref{L:CzBasic}(\ref{L:CzBasic:LCzCommEp})
at the first step,
Lemma~\ref{L:CzCompIneq} and $c f^2 c^* \leq c c^*$ at the second step,
and Lemma~\ref{L:CzBasic}(\ref{L:CzBasic:LCzWithinEp})
and $\| c c^* - a \| < 2 \ep$
at the last step,
we then have
\[
(f c^* c f - 2 \ep)_{+}
 \sim_{B} (c f^2 c^* - 2 \ep)_{+}
 \precsim_{B} (c c^* - 2 \ep)_{+}
 \precsim_{B} a.
\]
This completes the proof.
\end{proof}

\begin{thm}\label{T-2819AYLg}
Let $X$ be a \chs{}
and let $h \colon X \to X$ be a minimal homeomorphism.
Let $Y \subset X$ be a compact subset
such that $h^n (Y) \cap Y = \varnothing$
for all $n \in \Z \setminus \{ 0 \}$.
Then $C^* (\Z, X, h)_Y$ is a large subalgebra of $C^* (\Z, X, h)$
of crossed product type
in the sense of Definition~\ref{D-2717CPType}.
\end{thm}

\begin{proof}
We verify the hypotheses of Proposition~\ref{P-2729AltCPT}.
We follow Notation~\ref{N-Old32}.
Set
\[
A = C^* (\Z, X, h),
\,\,\,\,\,\,
B = C^* (\Z, X, h)_Y,
\,\,\,\,\,\,
C = C (X),
\andeqn
G = \{ u \}.
\]

Since $h$ is minimal,
it is well known that $A$ is simple and finite.
In particular,
condition~(\ref{P-2729AltCPT-Fin})
of Proposition~\ref{P-2729AltCPT} holds.

We next verify condition~(\ref{P-2729AltCPT-Sb})
of Proposition~\ref{P-2729AltCPT}.
All parts are obvious except~(\ref{P-2729AltCPT-Sb4}).
So let $a \in A_{+} \setminus \{ 0 \}$
and $b \in B_{+} \setminus \{ 0 \}$.
Apply Lemma~\ref{L:CXSubEq} with $G = \Z$ twice,
the first time with $A$ in place of $B$ and $a$ as given
and the second time with $B$ as given
(this is justified by Proposition~\ref{P-2816CharOB})
and with $b$ in place of~$a$.
We get $a_0, b_0 \in C (X)_{+} \setminus \{ 0 \}$
such that $a_0 \precsim_{A} a$ and $b_0 \precsim_{B} b$.
Set
\[
U = \{ x \in X \colon a_0 (x) \neq 0 \}
\andeqn
V = \{ x \in X \colon b_0 (x) \neq 0 \}.
\]
Choose a point $z \in Y$.
By minimality, there is $n \in \Z$ such that
$h^n (z) \in U$.
By Lemma~\ref{L-2623CuSub}, there exist $f_0, g \in C (X)_{+}$
such that $f_0 (x) = 1$ for all $x \in h^n (Y)$,
such that $0 \leq f_0 \leq 1$,
such that $\supp (g) \subset V$,
and such that $f_0 \precsim_{B} g$.
Choose $f_1 \in C (X)$ such that
\[
0 \leq f_1 \leq 1,
\,\,\,\,\,\,
f_1 (h^n (z)) = 1,
\andeqn
\supp (f_1) \subset U.
\]
Set $f = f_0 f_1$.
Then $f (h^n (z)) = 1$, so $f \neq 0$, and
\[
f \leq f_1 \precsim_{C (X)} a_0 \precsim_{A} a
\andeqn
f \leq f_0 \precsim_{B} g \precsim_{C (X)} b_0 \precsim_{B} b.
\]
This completes the proof of condition~(\ref{P-2729AltCPT-Sb4}).

We now prove condition~(\ref{P-2729AltCPT-AltCut}).
Let $m \in \N$,
let $a_1, a_2, \ldots, a_m \in A$,
let $\ep > 0$,
and let $b \in B_{+} \setminus \{ 0 \}$.

Choose $c_1, c_2, \ldots, c_m \in C (X) [\Z]$
such that $\| c_j - a_j \| < \ep$
for $j = 1, 2, \ldots, m$.
(This estimate is condition~(\ref{P-2729AltCPT-AltCut2}).)
Choose $N \in \N$
such that for $j = 1, 2, \ldots, m$
there are $c_{j, l} \in C (X)$
for $l = -N, \, - N + 1, \, \ldots, \, N - 1, \, N$
with
\[
c_j = \sum_{l = -N}^N c_{j, l} u^l.
\]

Apply Lemma~\ref{L:CXSubEq} to $B$ in the same way as in the
verification of condition~(\ref{P-2729AltCPT-Sb})
to find $f \in C (X)_{+} \setminus \{ 0 \}$
such that $f \precsim_{B} b$.
Set $U = \{ x \in X \colon f (x) \neq 0 \}$,
and choose nonempty disjoint open sets
$U_l \subset U$ for $l = -N, \, - N + 1, \, \ldots, \, N - 1, \, N$.
For each such~$l$,
use Lemma~\ref{L-2623CuSub}
to choose $f_l, r_l \in C (X)_{+}$
such that $r_l (x) = 1$ for all $x \in h^l (Y)$,
such that $0 \leq r_l \leq 1$,
such that $\supp (f_l) \subset U_l$,
and such that $r_l \precsim_{B} f_l$.

Choose an open set $W$ containing~$Y$
such that
\[
h^{-N} (W), \, h^{- N + 1} (W),
 \, \ldots, \, h^{N - 1} (W), \, h^{N} (W)
\]
are disjoint,
and choose $r \in C (X)$ such that $0 \leq r \leq 1$,
$r (x) = 1$ for all $x \in Y$,
and $\supp (r) \subset W$.
Set
\[
g_0 = r \cdot \prod_{l = - N}^{N} r_l \circ h^{l}.
\]
Set $g_l = g_0 \circ h^{- l}$
for $l = -N, \, - N + 1, \, \ldots, \, N - 1, \, N$.
Then $0 \leq g_l \leq r_l \leq 1$.
Set $g = \sum_{l = - N}^{N} g_l$.
The supports of the functions $g_l$ are disjoint,
so $0 \leq g \leq 1$.
This is condition~(\ref{P-2729AltCPT-AltCut1}).
Using Lemma~\ref{L:CzBasic}(\ref{L:CzBasic:Orth})
at the first and fourth steps
and Lemma~\ref{L:CzBasic}(\ref{L:CzBasic:CmpDSum})
at the third step,
we get
\[
g \sim_B \bigoplus_{l = - N}^{N} g_l
  \leq \bigoplus_{l = - N}^{N} r_l
  \precsim_B \bigoplus_{l = - N}^{N} f_l
  \sim_{C (X)} \sum_{l = - N}^{N} f_l
  \precsim_{C (X)} f
  \precsim_B b.
\]
This is condition~(\ref{P-2729AltCPT-AltCut4}).

It remains to verify condition~(\ref{P-2729AltCPT-AltCut3}).
Since $1 - g$ vanishes on the sets
\[
h^{-N} (Y), \, h^{- N + 1} (Y),
 \, \ldots, \, h^{N - 2} (Y), \, h^{N - 1} (Y),
\]
Proposition~\ref{P-2816CharOB}
implies that $(1 - g) u^l \in B$
for $l = -N, \, - N + 1, \, \ldots, \, N - 1, \, N$.
For $j = 1, 2, \ldots, m$,
since $c_{j, l} \in C (X) \subset B$
for $l = -N, \, - N + 1, \, \ldots, \, N - 1, \, N$,
we get
\[
(1 - g) c_j = \sum_{l = -N}^N c_{j, l} \cdot (1 - g) u^l \in B.
\]
This completes
the verification of condition~(\ref{P-2729AltCPT-AltCut3}),
and the proof of the theorem.
\end{proof}

In the proof,
it is also true that $c_j (1 - g) \in B$.

\begin{cor}\label{C_4819_AYStabLg}
Let $X$ be a \chs{}
and let $h \colon X \to X$ be a minimal homeomorphism.
Let $Y \subset X$ be a compact subset
such that $h^n (Y) \cap Y = \varnothing$
for all $n \in \Z \setminus \{ 0 \}$.
Then $C^* (\Z, X, h)_Y$ is a stably large subalgebra of $C^* (\Z, X, h)$
in the sense of Definition~\ref{D_4619_StLg}.
\end{cor}

\begin{proof}
Since $C^* (\Z, X, h)$ is stably finite,
we can combine Theorem~\ref{T-2819AYLg},
Proposition~\ref{P_4725_CPT},
and Corollary~\ref{C_4619_StFinStLg}.
\end{proof}

In Theorem~\ref{T-2819AYLg},
the condition
$h^n (Y) \cap Y = \varnothing$
for $n \in \Z \setminus \{ 0 \}$
is necessary.
If it fails,
then $C^* (\Z, X, h)_Y$ is not even simple.
Presumably this can be gotten
fairly easily by examining the corresponding groupoid,
but we can give an easy direct proof.

\begin{prp}\label{P_4811_NotSmp}
Let $X$ be an infinite \chs{}
and let $h \colon X \to X$ be a minimal homeomorphism.
Let $Y \subset X$ be a compact subset.
Suppose there is $n \in \Z$ such that $h^n (Y) \cap Y \neq \varnothing$.
Then $C^* (\Z, X, h)_Y$ has a nontrivial finite dimensional quotient.
\end{prp}

\begin{proof}
We first assume that there are $y \in Y$ and $n \in \N$
such that $h^n (y) \in Y$.
Let $\pi \colon C^* (\Z, X, h) \to l^2 (\Z)$
be the regular representation of $C^* (\Z, X, h)$ gotten
from the one dimensional representation $f \mapsto f (y)$
of $C (X)$.
Explicitly,
letting $\dt_m \in l^2 (\Z)$ be the standard basis vector
at $m \in \Z$,
this representation is determined by
$\pi (u) \dt_m = \dt_{m + 1}$
for $m \in \Z$
and $\pi (f) \dt_m = f (h^m (y)) \dt_m$
for $m \in \Z$ and $f \in C(Y)$.

Set $H_0 = l^2 \big( \{ 0, 1, \ldots, n - 1 \} \big) \subset l^2 (\Z)$.
We claim that if $a \in C^* (\Z, X, h)_Y$
then $\pi (a) H_0 \subset H_0$.
It suffices to show that if $f \in C (X)$
and $m \in \{ 0, 1, \ldots, n - 1 \}$,
then
\begin{equation}\label{Eq_4811_f}
\pi (f) \dt_m \in H_0,
\end{equation}
and that if, in addition,
$f |_Y = 0$,
then
\begin{equation}\label{Eq_4811_uffu}
\pi (f u) \dt_m \in H_0
\andeqn
\pi (f u)^* \dt_m \in H_0.
\end{equation}
The relation~(\ref{Eq_4811_f})
is immediate.
For the first part of~(\ref{Eq_4811_uffu}),
assuming $f |_Y = 0$,
we observe that $\pi (f u) \dt_m = f (h^{m + 1} (y)) \dt_{m + 1}$.
If $m \in \{ 0, 1, \ldots, n - 2 \}$,
this expression is clearly in~$H_0$.
If $m = n - 1$,
it is in $H_0$ because $f (h^{m + 1} (y)) = 0$.
The second part of~(\ref{Eq_4811_uffu})
is similar:
$\pi (f u)^* \dt_m = f (h^{m} (y)) \dt_{m - 1}$,
which is clearly in $H_0$ if  $m \in \{ 1, 2, \ldots, n - 1 \}$,
and is zero if $m = 0$.
The claim is proved.

Now let $p \in L (l^2 (\Z))$ be the projection on~$H_0$.
Then $a \mapsto p \pi (a) p$
is a unital \hm{}
from $C^* (\Z, X, h)_Y$ to $L (H_0) \cong M_n$.
This completes the proof under the assumption
that there is $n > 0$ such that $h^n (Y) \cap Y = \varnothing$.

To finish the proof, assume that there is $n \in \N$
such that $h^{-n} (Y) \cap Y = \varnothing$.
Set $Z = h^{-n} (Y)$.
Then $u^{-n} C^* (\Z, X, h)_Y u^n = C^* (\Z, X, h)_Z$,
and $C^* (\Z, X, h)_Z$ has a nontrivial finite dimensional quotient
by the case already done,
so $C^* (\Z, X, h)_Y$ has a nontrivial finite dimensional quotient.
\end{proof}

\begin{exa}\label{E_4819_NoKIso}
We show that the incusion of a large subalgebra
need not be an isomorphism on the Cuntz semigroups.
In particular, Theorem~\ref{C_3X08_IsoOnPure}
fails if one does not delete the classes of
projections in the Cuntz semigroups.

Let $X$ be the Cantor set,
let $h \colon X \to X$ be a \mh,
let $y_1, y_2 \in X$ be points on distinct orbits of~$h$,
and set $Y = \{ y_1, y_2 \}$.
Set $A = C^* (\Z, X, h)$
and $B = C^* (\Z, X, h)_Y$.
Let $\io \colon B \to A$
be the inclusion.
It follows from Theorem~4.1 of~\cite{Pt1}
that
$\io_* \colon K_0 ( B ) \to K_0 ( A )$
is not injective.
Therefore there are two \pj{s}
$p_1, p_2 \in M_{\infty} ( B )$
which are not \mvnt{} in $M_{\infty} (B)$
but are \mvnt{} in $M_{\infty} (A)$.
Since $B$ and $A$ are stably finite,
the maps from the sets of \mvnc{} classes of \pj{s}
over these algebras
to their Cuntz semigroups
(both $W (-)$ and ${\operatorname{Cu}} (-)$)
are injective.
Therefore
$\io_* \colon W ( B ) \to W ( A )$
and
$\io_* \colon {\operatorname{Cu}} ( B )
  \to {\operatorname{Cu}} ( A )$
are not injective.
However, $B$ is stably large in $A$
by Corollary~\ref{C_4819_AYStabLg}.
\end{exa}

We presume that much more complicated things can go wrong
with the map
${\operatorname{Cu}} ( C^* (\Z, X, h)_Y )
  \to {\operatorname{Cu}} ( C^* (\Z, X, h) )$.
In some cases, the map
$K_0 ( C^* (\Z, X, h)_Y ) \to K_0 ( C^* (\Z, X, h) )$
can be computed using Example~2.6 of~\cite{Pt3}.

\end{document}